\documentclass[11pt, a4paper,reqno]{amsart}

\usepackage{amsmath}
\usepackage{amssymb}
\usepackage{amsthm}
\usepackage{mathtools}
\usepackage{mathrsfs}
\usepackage{stmaryrd}
\usepackage[shortlabels]{enumitem}
\usepackage[british]{babel}
\usepackage{hyperref}
\usepackage{xcolor}
\usepackage{bbm}
\usepackage[normalem]{ulem}
\mathtoolsset{showonlyrefs=true}
\usepackage{cancel}

\calclayout
\hypersetup{colorlinks=true,linkcolor=blue,citecolor=blue,urlcolor=blue}
\newtheorem{theorem}{Theorem}[section]
\newtheorem{corollary}[theorem]{Corollary}
\newtheorem{lemma}[theorem]{Lemma}
\newtheorem{proposition}[theorem]{Proposition}
\theoremstyle{definition}
\newtheorem{remark}[theorem]{Remark}
\numberwithin{equation}{section}

\newcommand{\abs}[1]{\left\lvert#1\right\rvert}
\newcommand{\norm}[1]{\left\|#1\right\|}
\newcommand{\R}{\mathbb R}
\newcommand{\CC}{\mathbb C}
\newcommand{\N}{\mathbb N}
\newcommand{\Sph}{\mathbb S^2}
\newcommand{\dd}{\, \mathrm d}
\newcommand{\ddd}{\mathrm d}
\newcommand{\id}{\mathrm{Id}}
\renewcommand{\Re}{\operatorname{Re}}
\renewcommand{\Im}{\operatorname{Im}}

\newcommand{\cB}{\mathcal B}
\newcommand{\cC}{\mathcal C}
\newcommand{\cD}{\mathcal D}

\newcommand{\cF}{\mathcal F}
\newcommand{\cG}{\mathcal G}

\newcommand{\cK}{\mathcal K}
\newcommand{\cL}{\mathcal L}
\newcommand{\cM}{\mathcal M}

\newcommand{\cO}{\mathcal O}

\newcommand{\cR}{\mathcal R}
\newcommand{\cS}{\mathcal S}

\newcommand{\cU}{\mathcal U}
\newcommand{\cV}{\mathcal V}
\newcommand{\cX}{\mathcal X}
\newcommand{\scrA}{\mathscr A}
\newcommand{\scrE}{\mathscr E}
\newcommand{\scrH}{\mathscr H}

\newcommand{\scrJ}{\mathscr J}
\newcommand{\scrK}{\mathscr K}
\newcommand{\scrL}{\mathscr L}
\newcommand{\scrN}{\mathscr N}

\newcommand{\scrR}{\mathscr R}

\newcommand{\frB}{\mathfrak B}

\newcommand{\frP}{\mathfrak P}

\newcommand{\idproj}{\mathsf Q}
\newcommand{\Qin}{\mathcal Q^{\rm in}}
\newcommand{\Qout}{\mathcal Q^{\rm out}}
\renewcommand{\L}{\operatorname{L}}
\newcommand{\C}{\operatorname{C}}
\renewcommand{\H}{\operatorname{H}}
\newcommand{\din}{\rho^{\rm in}}
\newcommand{\dout}{{\rho^{\rm out}}}
\newcommand{\ein}{\phi^{\rm in}}
\newcommand{\eout}{\phi^{\rm out}}
\DeclareRobustCommand{\eoutdot}{\dot{\phi}\protect{\vphantom{\phi}}^{\rm out}}
\newcommand{\qin}{q^{\rm in}}
\newcommand{\qout}{q^{\rm out}}
\newcommand{\nablaS}{\nabla_{\!\Sph}}
\DeclareMathOperator{\Span}{span}
\newcommand{\Div}{\operatorname{div}}
\newcommand{\avg}[1]{\left\langle #1\right\rangle}
\newcommand{\jump}[1]{\left\llbracket #1\right\rrbracket}
\newcommand{\proj}{\mathsf P}
\newcommand{\perpProj}{\pi_{\{e_3\}^{\perp}}}
\newcommand{\eps}{\varepsilon}
\newcommand{\so}{\operatorname{SO}(2)}

\renewcommand{\epsilon}{\varepsilon}

\DeclareRobustCommand{\Hdot}{\dot{\H}\protect{\vphantom{H}}}

\newcommand\set[1]{\left\{#1\right\}}
\newcommand{\Oin}{\cD^{\rm in}}
\newcommand{\Oout}{\cD^{\rm out}}
\newcommand{\Oinout}{\cD^{\rm in /\rm out}}
\newcommand{\Omegain}{\Omega^{\rm in}}
\newcommand{\Omegainout}{\Omega^{\rm in / \rm out}}
\newcommand{\Omegaout}{\Omega^{\rm out}}
\newcommand{\uin}{u^{\rm in}}
\newcommand{\uout}{u^{\rm out}}
\newcommand{\uinout}{u^{\rm in / \rm out}}
\newcommand{\pin}{p^{\rm in}}
\newcommand{\pout}{p^{\rm out}}
\newcommand{\pinout}{p^{\rm in / \rm out}}
\newcommand{\tphiin}{\tilde{\phi}^{\rm in}}
\newcommand{\tphiout}{\tilde{\phi}^{\rm out}}

\newcommand{\phiin}{\phi^{\rm in}}
\newcommand{\phiout}{\phi^{\rm out}}
\newcommand{\phinout}{\phi^{\rm in / \rm out}}
\DeclareRobustCommand{\phinoutdot}{\dot{\phi}\protect{\vphantom{\phi}}^{\rm in / \rm out}}

\DeclareMathOperator{\divv}{div}
\DeclareMathOperator{\curl}{curl}

\newcommand{\We}{\operatorname{We}}

\title{On Leonardo da Vinci's paradox}

\author{Bj\"orn Gebhard}

\address[Bj\"orn Gebhard]{Institut f\"ur Analysis und Numerik, Universit\"at M\"unster Orl\'eans-Ring 10, 48149 M\"unster, Germany.}
\email{bjoern.gebhard@uni-muenster.de}

\author{Lukas Niebel}

\address[Lukas Niebel]{ETH Z\"urich, Department of Mathematics, R\"amistrasse 101, 8092 Z\"urich, Switzerland.}
\email{lukas.niebel@math.ethz.ch}

\author{Christian Seis}

\address[Christian Seis]{Institut f\"ur Analysis und Numerik, Universit\"at M\"unster Orl\'eans-Ring 10, 48149 M\"unster, Germany.}
\email{seis@uni-muenster.de}

\date{\today}



\begin{document}
\allowdisplaybreaks

\begin{abstract}
	Buoyancy points straight upward, yet in the absence of external forces, a rising air bubble may trace a spiral through the water. This phenomenon was already observed by Leonardo da Vinci. We show that the gravity-free, three-dimensional two-phase incompressible Euler equations with surface tension already admit such a motion. For every nonnegative inner-to-outer density ratio and every sufficiently small positive axial Weber number, we construct smooth near-spherical bubbles and drops with compact interfaces and phasewise irrotational flow. They are stationary in a frame translating along and rotating about a fixed axis, while their centroids lie off that axis and therefore trace genuine circular helices in laboratory coordinates. These helical relative equilibria form a symmetry-breaking branch that issues from the axisymmetric family of rectilinearly translating solutions and is generated by the spherical harmonic mode of degree two and azimuthal number one.

	The mathematical construction exploits an overdetermined free-boundary formulation, its variational structure, and an $\mathrm{SO}(2)$-equivariant Lyapunov--Schmidt reduction.
	Particular difficulties in the form of resonances arise at a discrete set of density ratios, including the frequently studied vacuum case with vanishing inner density.

	Complementary to our existence result, in the absence of surface tension and if the inner density is not larger than the outer density we prove that the model admits neither nontrivial translating relative equilibria nor relative equilibria with helical centroid trajectories.
	As a byproduct, we show that every regular finite-energy stationary
	vortex sheet of spherical topology for the homogeneous three-dimensional Euler equations, with irrotational flow on both sides, is trivial.
\end{abstract}

\maketitle
\vspace{-2\baselineskip}
\tableofcontents

\section{Introduction}

``The air which is submerged together with the water \dots{}
returns to the air, penetrating the water in sinuous movement.''
With these words, Leonardo da Vinci described air rising through water in the
\emph{Codex Leicester} \cite[vol.~I, p.~112]{LeonardoMacCurdy1938}. The upward
motion is explained by buoyancy. The phrase ``sinuous movement'' records the
less immediate part of the observation: although buoyancy singles out the
vertical direction, the bubble does not need to follow it. A drawing in the same
notebook shows a bubble following a spiral path towards the surface
\cite[p.~136]{christies1994leonardo}. Long after da Vinci's observations the physics literature refers to ``the failure
of bubbles to follow a straight path [as] Leonardo's paradox''
\cite[p.~1861]{MR2060292}.

The upward motion itself is not paradoxical. What requires explanation is the loss of symmetry in the absence of external forces.
As long as the bubble and the surrounding flow are axisymmetric, no horizontal
direction is distinguished. Nevertheless, the bubble may acquire a lateral
velocity and then follow a zigzag or spiral path.

A freely moving bubble is not a rigid body. Its interface responds to the
surrounding pressure, its curvature is constrained by surface tension, and
each change of shape affects the fluid motion and vice versa. In a viscous liquid, the wake
adds another part to this interaction. Early studies considered
terminal speed, bubble shape, and the onset of non-rectilinear paths together
\cite{Saffman_1956,MR88963,Moore_1965,MR627227,MR687006}.

During the 1980s, a complementary line of work isolated inviscid aspects of
the problem. Tsamopoulos and Brown studied nonlinear oscillations of inviscid
drops and bubbles \cite{Tsamopoulos_Brown_1983}. Benjamin gave a Hamiltonian
formulation for capillary bubbles in potential flow \cite{MR905816}. Meiron
then investigated the stability of such bubbles numerically. He was careful
not to identify the inviscid model with the laboratory experiment, but
suggested that it might possess equilibria representing steady spiralling
\cite[p.~113]{MR982188}; see also \cite{Pozrikidis_1989}.

More recent work has returned to the path instability of bubbles rising in
water. Herrada and Eggers emphasise the coupling between pressure and
time-dependent deformation \cite{MR4548739}. Bonnefis, Fabre and Magnaudet
find that a fixed-shape model already captures the qualitative instability,
while deformation improves the predicted threshold \cite{MR4581313}. Both
analyses concern the dynamical loss of a straight path in a gravity-driven,
viscous problem.

We take the inviscid question on its own terms. We remove gravity and
viscosity, retain the two-phase incompressible Euler equations and surface
tension, and ask only for existence of a steady-spiralling bubble solution
as observed by Leonardo da Vinci. We neither address the stability of the
resulting motions nor investigate the instability mechanism at the moment of their emergence.

The motions considered here have a simple description. The interface is
stationary in a frame which translates and rotates about the same axis. In
laboratory coordinates it undergoes a screw motion. For the spiralling solutions
constructed below, the centroid is displaced from the axis and traces a circular helix.
We call such a solution a \emph{relative equilibrium}.
The interface itself has a near-spherical but asymmetric shape.

Rectilinearly translating, near-spherical bubbles and drops were recently
constructed in the same two-phase Euler model with surface tension \cite{MR4972964}. The present
paper finds a non-axisymmetric branch emerging from those solutions. It thereby
gives a rigorous near-spherical realisation of the steady-spiralling scenario
proposed in the inviscid literature and observed already by Leonardo da Vinci.

We call the inner phase a \emph{bubble} when it is lighter than the outer phase and a
\emph{drop} when it is heavier. The axial Weber number compares the inertial pressure
associated with translation with the capillary pressure, see \eqref{C1}. Small Weber number is
therefore the regime in which surface tension keeps the interface close to a
sphere.

We first state a non-specific version of our main results.

\begin{theorem}
	\label{thm:main-informal}
	Let $\din \ge 0$ and $\dout>0$ denote the mass densities of the inner and outer fluids,
	respectively. Furthermore, let $ \sigma>0$ be the surface-tension coefficient, and
	$ m>0$ be a prescribed volume.
	Then, for every sufficiently small and positive
	(axial) Weber number (defined in \eqref{C1} below),
	\[
		0<\We\ll1,
	\]
	the two-phase capillary Euler equations admit a relative-equilibrium solution describing a
	near-spherical bubble/drop of volume $m$ whose centroid moves along a true helical trajectory.
	For each sufficiently small fixed Weber number, this non-axisymmetric solution branch is locally unique
	near the $(2,1)$ Rayleigh--Lamb frequency,
	up to translations along and rotations about the screw axis.
\end{theorem}

The displacement of the centroid is essential. An axisymmetric translating
solution may be described in a rotating frame with any angular velocity, but
this does not make its path helical: its centroid stays on the axis.
Our solution branch has a nonzero horizontal displacement, and hence describes a genuine helix.

The Rayleigh--Lamb frequency already occurs as the natural oscillation frequency in the study  of droplets and bubbles oscillating about a spherical form; see \cite[Chapter~IX, Art.~275, (10)]{MR1317348}.

We now introduce the equations and state the main results in full detail.

\subsection*{The two-phase Euler equations}
Let $\din\ge0$ and $\dout>0$ be the constant densities of the inner and
outer phases. At time $t$, the inner phase occupies a bounded domain
$\Omegain(t)\subset\R^3$, and the outer phase occupies
\[
	\Omegaout(t):=\R^3\setminus\overline{\Omegain(t)}.
\]
We assume that the interface has the topology of a sphere: $\Omegain(t)$ and
$\Omegaout(t)$ are simply connected, and
$\partial\Omegain(t)=\partial\Omegaout(t)$ is connected and smooth. Set
\[
	\Oinout:=\set{(t,x):x\in\Omegainout(t)},
	\qquad
	\cS:=\partial\Oin\cap\partial\Oout.
\]

The velocity and pressure are denoted by
\[
	u=\uin\mathbbm 1_{\Oin}+\uout\mathbbm 1_{\Oout},
	\qquad
	p=\pin\mathbbm 1_{\Oin}+\pout\mathbbm 1_{\Oout},
\]
and $\rho=\din\mathbbm 1_{\Oin}+\dout\mathbbm 1_{\Oout}$. On the
interface, $n$ is the unit normal pointing out of the inner phase,
$V_{\cS}$ is the normal velocity of the interface, and
$H=\divv_{\partial\Omegain(t)}n$ is its mean curvature. We use the jump
convention
\[
	\jump{f}:=f^{\rm in}-f^{\rm out}.
\]
Let $\sigma>0$ be the surface-tension coefficient. With these notations, the two-phase capillary Euler
equations are
\begin{equation}
	\begin{aligned}\label{eq:original_euler}
		\hspace{3em}\rho(\partial_t u+u\cdot\nabla u)+\nabla p & =0        &  & \text{in }\Oin\cup\Oout, \\
		\divv u                                                & =0        &  & \text{in }\Oin\cup\Oout, \\
		\uin\cdot n=\uout\cdot n                               & =V_{\cS}  &  & \text{on }\cS,           \\
		\jump p                                                & =\sigma H &  & \text{on }\cS.
	\end{aligned}
\end{equation}
The first two equations are the incompressible Euler equations in each phase.
The third says that the fluid transports the interface. The last is the
Young--Laplace law. Our sign convention gives $H=2/R$ on a sphere of radius
$R$.

We consider regular solutions of finite kinetic energy. Since the inner region
is bounded, this amounts to
\begin{equation}\label{eq:finite_energy_condition}
	\uout(t,\cdot)\in\L^2(\Omegaout(t);\R^3).
\end{equation}
When $\din=0$, the inner phase is interpreted as a massless gas: the inner
velocity field and its equations are omitted, and the kinematic condition in
\eqref{eq:original_euler} is read as $\uout\cdot n=V_{\cS}$, and the interior pressure $p^{\rm in}$ on the surface is constant.

The relative equilibria constructed here translate with speed $V$ along and rotate
with angular velocity $\alpha$ about the $x_3$-axis. Write
\begin{equation}\label{eq:rotationmatrix}
	\cR_\theta:=
	\begin{pmatrix}
		\cos\theta & -\sin\theta & 0 \\
		\sin\theta & \cos\theta  & 0 \\
		0          & 0           & 1
	\end{pmatrix},
	\qquad
	e_3:=
	\begin{pmatrix}0\\0\\1\end{pmatrix}.
\end{equation}
The domains, density and pressure are stationary in the moving coordinates,
while velocity transforms as a vector:
\begin{gather}\begin{gathered}\label{eq:rigid_motion_Omega_in}
		\Omegainout(t)=\cR_{\alpha t}\Omegainout(0)+tVe_3,\\
		\rho(t,\cR_{\alpha t}y+tVe_3)=\rho(0,y),\qquad
		p(t,\cR_{\alpha t}y+tVe_3)=p(0,y),\\
		u(t,\cR_{\alpha t}y+tVe_3)=\cR_{\alpha t}u(0,y).
	\end{gathered}\end{gather}
We also impose phasewise irrotationality,
\begin{equation}\label{eq:irrotational_condition}
	\curl u=0\qquad\text{in }\Oin\cup\Oout.
\end{equation}
Since both phases are simply connected, their velocities have single-valued
potentials, $\uinout=\nabla\phinout$. The velocity may jump tangentially
across the interface and form a vortex sheet. The vorticity of the full flow is concentrated there.

\subsection*{The overdetermined free boundary problem}

The following explanation is only formal; the complete derivation is
given in Appendix~\ref{sec:derivation}.

We next pass from the evolving Euler system to a stationary problem for the
shape. Irrotationality and incompressibility
give harmonic velocity potentials in the two phases. Once the moving frame is
prescribed, the kinematic condition fixes their normal derivatives on the
interface. Thus, for a chosen shape, two harmonic Neumann problems determine
the flow.

The shape is not arbitrary. Euler's equation in the moving frame gives a
Bernoulli law in each phase, while the Young--Laplace law forces the pressure
jump to be proportional to mean curvature. The resulting expression must be constant along the
same interface that was used to solve the Neumann problems. This second demand
is an additional scalar equation for the unknown shape and is what makes the
free boundary problem below overdetermined.

For $\alpha,V\in\R$, set
\begin{equation}\label{eq:def-WalphaV}
	W_{\alpha,V}(y):=\alpha e_3\times y+Ve_3.
\end{equation}
Let $\Omega\subset\R^3$ be the inner phase at time zero, let
$\Gamma=\partial\Omega$, and let $n$ be the outer unit normal of $\Omega$.
We assume that $\Omega$ is bounded, that $\Gamma$ is connected and smooth, and
that both $\Omega$ and $\R^3\setminus\overline\Omega$ are simply connected.
The volume is prescribed as
\[
	\abs{\Omega}=\frac{4\pi}{3}R^3;
\]
thus $R$ is the radius of the sphere having the same volume.

We seek $\Omega$, harmonic potentials
$\ein\colon\Omega\to\R$ and
$\eout\colon\R^3\setminus\overline\Omega\to\R$, and a Bernoulli constant
$c\in\R$. With $\phi=\ein$ and $\rho=\din$ in $\Omega$, and
$\phi=\eout$ and $\rho=\dout$ outside, the stationary problem is
\begin{align}\label{eq:main-problem}
	\Delta\ein                     & =0
	                               &                                                 & \text{in }\Omega,                             \nonumber \\
	\Delta\eout                    & =0
	                               &                                                 & \text{in }\R^3\setminus\overline\Omega,\nonumber        \\
	\partial_n\ein=\partial_n\eout & =W_{\alpha,V}\cdot n
	                               &                                                 & \text{on }\Gamma,\nonumber                              \\
	\jump{\frac{\rho}{2}\abs{\nabla\phi-Ve_3}^2
		+\rho\alpha\,y\cdot(e_3\times\nabla\phi)}
	+\sigma H                      & =c
	                               &                                                 & \text{on }\Gamma,                                       \\
	\eout(y)                       & \longrightarrow0
	                               &                                                 & \text{as }\abs y\to\infty,\nonumber                     \\
	\nabla\eout                    & \in\L^2(\R^3\setminus\overline\Omega),\nonumber                                                           \\
	\abs{\Omega}                   & =\frac{4\pi}{3}R^3.\nonumber
\end{align}
When $\din=0$, all interior-potential terms and the interior Neumann problem
are omitted.
Conversely, every sufficiently regular solution of \eqref{eq:main-problem}
gives a relative equilibrium of \eqref{eq:original_euler}. In the following, we will frequently refer to the jump equation as the \emph{Bernoulli equation}.

\subsection*{Parameters and main results}

Two dimensionless quantities organise the statement:
\begin{equation}
	\label{C1}
	\We:=\frac{\dout RV^2}{\sigma},
	\qquad
	\tau:=\frac{\din}{\dout}.
\end{equation}
Here, $\dout V^2$ is the inertial pressure scale and $\sigma/R$ is the
capillary pressure scale. The reference angular velocity is
\[
	\alpha_*:=
	\left(\frac{24\sigma}{R^3(3\din+2\dout)}\right)^{1/2}>0.
\]
The positive number $\alpha_*$ is the Rayleigh--Lamb frequency
for spherical harmonics of degree two and azimuthal number one
\cite[Chapter~IX, Art.~275, (10)]{MR1317348}.

The unit sphere will be written as
\[
	\Sph=\set{\omega=(X,Y,Z)\in\R^3:X^2+Y^2+Z^2=1}.
\]
We write $e_1,e_2,e_3$ for the standard basis of $\R^3$ and
$\perpProj$ for the orthogonal projection onto the horizontal plane
$\{e_3\}^{\perp}$.

Our investigation shows that at the sphere, the linearised shape equation is diagonal in spherical
harmonics. A small translation first excites the axisymmetric degree-two mode
$3Z^2-1$, making the sphere slightly oblate. The symmetry-breaking directions
used here are the degree-two azimuthal-number-one modes $XZ$ and $YZ$.
At $\alpha=\alpha_*$, these modes lie in the kernel of the linearised shape
operator and serve as bifurcation directions for a branch of relative
equilibria.

We give now a precise version of our main result, which we loosely announced in Theorem \ref{thm:main-informal}.

\begin{theorem}
	\label{thm:din-pos-dimensional}
	Assume that
	\[
		\din\ge 0,\qquad \dout>0,\qquad \sigma>0
	\]
	are given inner and outer mass densities and a surface tension coefficient, respectively.
	There are constants
	$\We_0>0$ and $\eta>0$, depending only on the density ratio $\tau$, with the following
	property: For any axial velocity $V \ge 0$ and any radius  $R>0$ satisfying
	\[
		0\le\We<\We_0
	\]
	and any $\abs{s}\le \eta\sqrt{\We}$, there exists a smooth solution ${\Omega=}\Omega(s,\We)$ of \eqref{eq:main-problem}, with angular velocity 	$\alpha=\alpha(s,\We)>0$, whose boundary is the radial graph
	\[
		\Gamma(s,\We)=
		\set{R(1+h(s,\We)(\omega))\omega:\omega\in\Sph}.
	\]
	The shape function $h$ is smooth for every fixed $s$ and $\We$. After a rotation about the $e_3$-axis, the parameter $s$ may be chosen so
	that the leading non-axisymmetric graph-height term is $sXZ$.
	The leading-order asymptotics of the   shape function are then given by
	\begin{align}
		\label{eq:h-expansion}
		h(s,\We)
		={} & -\frac{3}{16}\We\,\frac{3Z^2-1}{2}+sXZ \\
		    & -\sqrt{\We}\,sY\sqrt{6(3\tau+2)}
		\left[
			\frac{3}{40(2\tau+1)}
			+\frac{5}{22\tau+14}\left(Z^2-\frac15\right)
			\right]\nonumber \\
		    & +O_{\C^4(\Sph)}
		\left(s^2+ \We^2\right).\nonumber
	\end{align}
	The angular velocity satisfies
	\begin{equation}\label{eq:alpha-expansion}
		\frac{\alpha(s,\We)}{\alpha_*}
		=1+\frac{\gamma_{21}(\tau)}8\We
		+O\left(s^2+\We^{3/2}\right),
	\end{equation}
	where
	\[
		\gamma_{21}(\tau)
		:=-\frac{
			3(1188\tau^3+6605\tau^2+5601\tau+1156)}
		{280(2\tau+1)(3\tau+2)(11\tau+7)}.
	\]
	Moreover, there is the estimate
	\begin{align}\label{eq:centroid-expansion}
		\frac1R\perpProj\left(
		\frac1{\abs{\Omega(s,{\We})}}\int_{\Omega(s,{\We})}y\dd y
		\right)
		={} & -
		\frac{3\sqrt{6(3\tau+2)}}{40(2\tau+1)}
		s\sqrt{\We}\,e_2                               \\
		    & +O\left(s^2+\abs{s}\We\right). \nonumber
	\end{align}
	In particular, the horizontal component of the centroid is nonzero for $s \neq 0$.
	Consequently, in laboratory coordinates the centroid follows a genuine
	circular helix for $sV \neq 0$. For $s=0$, the branch is axisymmetric, its centroid
	remains on the rotation axis and follows a
	rectilinear trajectory.
	For each fixed $V$ and $\alpha$ near the $(2,1)$ Rayleigh--Lamb frequency $\alpha_*$, the non-axisymmetric branch is locally unique modulo
	rotations about the $e_3$-axis in the fixed-volume, vertically centred
	radial-graph class. The axisymmetric shape is locally unique, but $\alpha$
	remains free.
\end{theorem}

Our analysis actually reveals that apart from the vacuum case $\tau=0$ and for density ratios $\tau $ outside of a certain resonance set $\cO$ defined in \eqref{C2},  the bifurcation branch extends to a fixed small interval of
$s$ independent of the axial Weber number $\We$, i.e.,
\[
	|s|\le \eta;
\]
see Section \ref{C3} below for details. Moreover, the asymptotic formula in \eqref{eq:h-expansion} can be improved to any fixed regularity order at the expense of possibly decreasing the values of $\We_0$ and $\eta$.

The inner vacuum case $\tau=0$
was considered earlier in Benjamin's variational model \cite{MR905816} and in Meiron's numerical
study \cite{MR982188}. In Section~\ref{sec:endpoint-benjamin-comparison} we relate Benjamin's
eccentricity and inclination parameters to $(\We,s)$. The resulting helix
radius, pitch, and leading angular frequency agree with his formulae in their
common small-eccentricity, small-inclination regime. Thus, in the vacuum case $\tau=0$,
Theorem~\ref{thm:din-pos-dimensional} gives a rigorous near-spherical
realisation of the steady-spiralling scenario conjectured by Meiron and recovers
Benjamin's leading asymptotics which have been based on simplified shape approximations.

\begin{remark}\label{rem:really_a_helix}
	Let
	\[
		\bar y_{s,{\We}}:=
		\frac1{\abs{\Omega(s,{\We})}}\int_{\Omega(s,{\We})}y\dd y
	\]
	be the centroid in the moving frame. In laboratory coordinates it is
	\[
		\bar x_{s,{\We}}(t)
		=\cR_{\alpha(s,{\We})t}\bar y_{s,{\We}}+Vte_3.
	\]
	If the horizontal part of $\bar y_{s,{\We}}$ has length $r>0$ and polar angle
	$\theta_0$, then
	\[
		\bar x_{s,{\We}}(t)
		=r\bigl(\cos(\alpha(s,{\We}) t+\theta_0)e_1
		+\sin(\alpha(s,{\We}) t+\theta_0)e_2\bigr)
		+(\bar y_{s,{\We}}\cdot e_3+Vt)e_3.
	\]
	Thus the trajectory is a circular helix of
	radius $r$ and pitch $2\pi V/\alpha(s,{\We})$.
	For example, choosing $\abs{s}$ of order $\sqrt{\We}$ gives a radius
	of order $R\We$ and a
	pitch of order $R\sqrt{\We}$.
\end{remark}

\subsection*{Idea of the construction}

We first nondimensionalise the problem and write the interface as a radial
graph over the unit sphere. For a prescribed graph and fixed motion parameters,
the requirement that the fluid and the interface have the same normal velocity
determines Neumann problems for the harmonic velocity potentials. Solving these
problems and substituting the potentials into the Bernoulli condition reduces
the free-boundary problem to an equation for the graph, subject to the
fixed-volume and an additionally imposed vertical-centring constraint. At zero axial speed, the round
sphere solves this equation. For small axial speed, the sphere extends to an
axisymmetric translating solution. This branch was also constructed in
\cite{MR4972964} and provides the base branch for the subsequent bifurcation
argument.

We then linearise the Bernoulli equation at the sphere, with angular velocity
equal to the $(2,1)$ Rayleigh--Lamb frequency. The desired symmetry-breaking
kernel is the plane generated by $XZ$ and $YZ$. The reduced equation is a
gradient equation and is $\mathrm{SO}(2)$-equivariant under rotations about the
vertical axis. Whether the weight-one plane is the full kernel depends on the
density ratio. Besides the endpoint $\tau=0$, additional resonances occur
precisely at the positive density ratios
\begin{equation}
	\label{C2}
	\cO:=
	\set{
		\frac{2}{77},
		\frac{2}{7},
		\frac{23}{60},
		\frac{10}{9},
		\frac{46}{27}
	}.
\end{equation}
Accordingly, we distinguish three cases.

In the nonresonant case $\tau>0$ and $\tau\notin\cO$, the weight-one plane is
the full kernel. Using the gradient structure and the
$\mathrm{SO}(2)$-equivariance, a Lyapunov--Schmidt reduction followed by the
implicit-function theorem yields the non-axisymmetric branch and selects its
angular velocity. The branch is locally unique modulo rotations about the
vertical axis and exists for amplitudes in a fixed neighbourhood of zero.

At the endpoint $\tau=0$, the linearised operator at the sphere has an
additional resonant plane of weight six. Instead of bifurcating from
the sphere, where both planes belong to the kernel, we first move away from the
sphere along the axisymmetric translating branch. Along this branch, the
additional resonant directions are lifted: their eigenvalue becomes nonzero,
of quadratic order in the axial speed, while the angular
velocity is chosen so that the weight-one plane remains in the kernel. This is
the main mechanism by which the translation separates the desired
bifurcation directions from the additional resonance. The inverse on the
weight-six plane is consequently not uniform as the axial speed tends to zero.
However, the $\so$-equivariance compensates for this loss, since the
weight-one mode can force the weight-six mode only at sixth order. After
rescaling the bifurcation amplitude with the axial speed, the reduced system
extends regularly to zero speed and the implicit-function theorem in an equivariant flavour applies.

For a positive resonant density ratio $\tau\in\cO$, the construction follows as
in the endpoint case, with the weight-six plane replaced by the corresponding
resonant spherical-harmonic plane. Moving along the axisymmetric branch again
lifts the additional resonant directions, and the same rescaling and
weight-selection argument yields the non-axisymmetric branch.

A priori, the solutions obtained in this way are merely relative equilibria
under a screw motion. The construction alone does not show that they spiral,
since their centroids might still lie on the rotation axis. It therefore
remains to derive the asymptotics of the graph, the angular velocity and the
centroid. These asymptotics show that every non-axisymmetric solution on the
constructed branch has a nonzero horizontal centroid component. Finally, we return to dimensional
variables and deduce Theorem \ref{thm:din-pos-dimensional}.

\begin{remark}
	The degree-two, azimuthal-number-one branch constructed here is one
	instance of a more general bifurcation mechanism. For $\ell\ge2$, the
	Rayleigh--Lamb frequency of a degree-$\ell$ spherical harmonic is
	\[
		\omega_\ell^{\rm RL}
		=
		\left(
		\frac{\sigma\ell(\ell-1)(\ell+1)(\ell+2)}
		{R^3\bigl((\ell+1)\din+\ell\dout\bigr)}
		\right)^{1/2}.
	\]
	A mode with azimuthal number $1\le m\le\ell$ enters the kernel of the
	rotating-frame linearisation at the angular velocity
	\[
		\alpha_{\ell,m}
		=
		\frac{\omega_\ell^{\rm RL}}{m}.
	\]
	Away from additional resonances, the same equivariant bifurcation scheme
	can be directly adapted to construct analogous branches associated
	with other pairs $(\ell,m)$ than $(2,1)$. At the resonant density ratios a more detailed analysis will be needed. It would be interesting to determine the
	geometry of these higher-mode interfaces and the corresponding
	laboratory-frame trajectories.
\end{remark}

\subsection*{A complementary rigidity statement}

For comparison, we record what happens in the same class when surface tension
is removed. This result covers the full bubble regime, including the massless
endpoint and the equal-density case.

\begin{theorem}
	\label{thm:no-bubbles}
	Let
	\[
		0\le\din\le\dout,\qquad \dout>0,\qquad
		R>0,\qquad \alpha,V\in\R.
	\]
	Let $\Omega\subset\R^3$ be a bounded domain with connected boundary
	$\Gamma=\partial\Omega$ of class $\C^2$, and assume that both $\Omega$ and
	$\R^3\setminus\overline\Omega$ are simply connected.
	Suppose that $\eout$, and also $\ein$ when $\din>0$, have $\C^2$
	one-sided extensions up to $\Gamma$ and, together with a constant $c\in\R$,
	solve \eqref{eq:main-problem} classically with $\sigma=0$.

	Then
	\[
		V=0,\qquad
		\nabla\eout\equiv0.
	\]
	If $\din>0$, then also $\nabla\ein\equiv0$. If $\alpha\neq0$, then
	\[
		\cR_\theta\Omega=\Omega
		\qquad\text{for every }\theta\in\R.
	\]
\end{theorem}

Thus there is no nontrivial translating relative equilibrium, and no relative
equilibrium with a helical centroid trajectory, in the full irrotational,
sphere-topology bubble regime without surface tension. The theorem does not
exclude static interfaces with trivial fluid velocity: when $V=\alpha=0$ and all potentials present in the
problem are constant, every admissible shape solves the zero-surface-tension
problem.

The equal-density case gives the following consequence for regular stationary
vortex sheets.

\begin{corollary}
	\label{cor:sheets}
	Let $\Gamma=\partial\Omega\subset\R^3$ be a closed embedded $\C^2$ surface
	diffeomorphic to $\Sph$, with $\Omega$ the bounded component of its
	complement. Let $(u,p)$ be a stationary distributional solution of the
	homogeneous Euler equations which is piecewise $\C^1$ up to $\Gamma$.
	Write $\uinout$ and $\pinout$ for its restrictions to the two phases.
	Assume
	\[
		u\in\L^2(\R^3;\R^3),\qquad
		\curl\uin=\curl\uout=0,
	\]
	and that $\Gamma$ is a stationary
	vortex sheet with
	\[
		\uin\cdot n=\uout\cdot n=0
		\qquad\text{on }\Gamma.
	\]
	Then $u\equiv0$.
\end{corollary}

\subsection*{Mathematical context}

The initial-value problem for free-boundary Euler flows with surface tension,
including vortex sheets with surface tension, has a substantial well-posedness
theory \cite{MR2001473,MR2334849,MR2291920,MR2456184,MR2763036}. Almost global
well-posedness near a circle is known in two dimensions \cite{MR5064622}.
The construction of coherent capillary solutions includes two-dimensional
hollow vortices and capillary vortex sheets
\cite{crowdy_circulation_1999,MR1794849,MR5037134,
	murgante2024steadyvortexsheetspresence}, and three-dimensional rectilinear
bubbles and drops, rotating drops, and ring-shaped vortex sheets
\cite{MR4972964,MR5001433,baldi2025bifurcationmultipleeigenvaluesrotating,
	meyer2025steadyringshapedvortexsheets}. Recent rigidity and nonexistence
results for two-dimensional bubbles, thick bubble rings and liquid drops appear in
\cite{niebel2026globalrigiditytwodimensionalbubbles,
	han2026nonexistencethickbubblerings,baldi2026rigiditycapillaryliquiddrops,baldi2026rigidityresult3dcapillary};
see also \cite{niebel2026sphericalrigidityexterioroverdetermined}.

The combination of symmetry with a variational or gradient structure is a
standard theme in equivariant bifurcation theory
\cite{MR950168,MR2004250}. In capillary free-boundary problems it has also
proved effective in the construction of special solutions
\cite{MR1781220,MR5001433,
	baldi2025bifurcationmultipleeigenvaluesrotating}.

There is also a large literature on coherent Euler structures without surface
tension. It includes steady vortex rings
\cite{MR422916,MR628444,MR1011553}, helical filaments, patches, and vortex
sheets
\cite{Caflisch01081992,MR4417384,MR4834945,MR5018456,
	guerra2025nearlyparallelhelicalvortex,MR5036920,cluster_helices}. The word
\emph{helical} has a different meaning in much of that literature: it describes
the spatial structure or symmetry of an unbounded vortex. In the present work the interface
is compact and need not have a helical shape. It undergoes a screw motion, and
it is the trajectory of its centroid that is a helix.

\subsection*{Acknowledgements}
The authors are funded by the Deutsche Forschungsgemeinschaft (DFG, German Research Foundation) under Germany's Excellence Strategy EXC 2044/2 -- 390685587, Mathematics Münster: Dynamics--Geometry--Structure and Excellence and grant no.~531098047.
Bj\"orn Gebhard acknowledges support by the Generalitat Valenciana through project CIGE/2024/115.
Lukas Niebel is also funded by the SNSF Starting Grant TMSGI2\textunderscore226018.

\subsection*{Declaration of AI Use}
OpenAI's GPT-5.5 and GPT-5.6 Sol were used for proof-checking.

\section{Functional-analytic and geometric setup}

\subsection{Nondimensionalisation}
\label{sec:nondimensionalisation}

First, we record the nondimensional form of our free-boundary problem \eqref{eq:main-problem}. For a near-spherical bubble, the natural
length scale is the prescribed equivalent radius
\[
	R = \left(\frac{3}{4\pi}|\Omega|\right)^{1/3}.
\]
Measuring length in units of this radius, we consider
\[
	y = R\widehat y.
\]

In particular, if the dimensionless and dimensional interfaces are written as
radial graphs
\[
	\widehat\Gamma
	=
	\set{(1+\widehat h(\omega))\omega:\omega\in\Sph},
	\qquad
	\Gamma
	=
	\set{R(1+h(\omega))\omega:\omega\in\Sph},
\]
then $\Gamma=R\widehat\Gamma$ is equivalent to
\begin{equation}
	\label{eq:radial-height-rescaling}
	h=\widehat h.
\end{equation}

The capillary velocity $\sqrt{\sigma/({\dout R})}$ balances the dynamic pressure with
the capillary pressure. We rescale all velocity variables in our model by this characteristic velocity
\[
	V=\left(\frac{\sigma}{\dout R}\right)^{1/2}\widehat V,
	\qquad W_{\alpha, V}(y) = \left(\frac{\sigma}{\dout R}\right)^{1/2} W_{\widehat \alpha, \widehat V}(\widehat y).
\]
Here, $\widehat \alpha$ is the nondimensionalised angular velocity
\[
	\alpha=\left(\frac{\sigma}{\dout R^3}\right)^{1/2}\widehat\alpha.
\]
The rescaled velocity potentials are given by
\[
	\ein(y)=\left(\frac{\sigma R}{\dout}\right)^{1/2}\widehat\phi^{\rm in}(\widehat y),
	\qquad
	\eout(y)=\left(\frac{\sigma R}{\dout}\right)^{1/2}\widehat\phi^{\rm out}(\widehat y).
\]

After dropping hats, the nondimensional problem is the following. Find a
bounded smooth simply connected domain $\Omega\subset\R^3$ with connected
boundary $\Gamma=\partial\Omega$, whose exterior
$\R^3\setminus\overline\Omega$ is also simply connected, potentials
\[
	\ein\colon\Omega\to\R,
	\qquad
	\eout\colon\R^3\setminus\overline\Omega\to\R,
\]
and a constant $c\in\R$ such that
\begin{equation}\label{eq:main-problem-nondimensional}
	\begin{aligned}
		\Delta \ein  & =0
		             &      & \text{in }\Omega,                       \\
		\Delta \eout & =0
		             &      & \text{in }\R^3\setminus\overline\Omega, \\
		\partial_n\ein=\partial_n\eout
		             & =
		W_{\alpha,V}\cdot n
		             &      & \text{on }\Gamma,                       \\
		\jump{
			\frac{\rho}{2}\abs{\nabla\phi-Ve_3}^2
			+\rho\alpha\, y\cdot(e_3\times\nabla\phi)
		}
		+H
		             & =c
		             &      & \text{on }\Gamma,                       \\
		\eout(y)     & \to0
		             &      & \text{as }\abs{y}\to\infty,             \\
		\nabla\eout  & \in
		\L^2(\R^3\setminus\overline\Omega),                           \\
		\abs{\Omega} & =
		\frac{4\pi}{3}.
	\end{aligned}
\end{equation}
Here, $\rho$ is the non-dimensional density functional
\[
	\rho
	=
	\tau\,\mathbbm{1}_{\Omega}
	+
	\mathbbm{1}_{\R^3\setminus\overline\Omega},
	\qquad
	\tau=\frac{\din}{\dout}.
\]
The original dimensional variables are recovered by undoing the above
scaling.

\subsection{Functional analysis}
We consider functions on the two-dimensional sphere $\Sph\subset\R^3$. We
write $\omega=(\omega_1,\omega_2,\omega_3)=(X,Y,Z)\in\Sph$. Throughout,
$\N=\{1,2,\ldots\}$ and	$\N_0=\{0,1,2,\ldots\}$.

All Sobolev spaces below are real Sobolev spaces. For $j\in\N$ we use
$\H^j(\Sph)$ with the standard Sobolev norm.
Since $\dim\Sph=2$, the spaces $\H^j(\Sph)$ are Banach algebras and
embed continuously into $\C^0(\Sph)$ for $j>1$. Consequently, the smooth
pointwise expressions used below, including rational expressions with
denominators bounded away from zero, define smooth maps on the relevant open
subsets; see
\cite[Lemma~6.2.8 and Remark~6.2.9(i)]{PruessSimonett2016} for the underlying
multiplication estimates.

For consistency with the later analysis, we take $j=k+2$, where $k\ge4$ is
an integer. In particular, functions in $\H^{k+2}(\Sph)$ are at least
$\C^4$. For $\delta>0$, we introduce the open ball
\[
	U^{k+2}_\delta
	:=
	\set{
		h\in \H^{k+2}(\Sph):
		\norm{h}_{\H^{k+2}}<\delta
	}.
\]
By Sobolev embedding, after choosing $\delta$ sufficiently small, every
$h\in U^{k+2}_{\delta}$ satisfies $1+h>\frac12$. For any such $h$ we write
\begin{equation}\label{eq:radial-graph-chart}
	\Gamma_h=\{(1+h(\omega))\omega:\omega\in\Sph\},
	\qquad
	\Psi_h(\omega)=(1+h(\omega))\omega ,
\end{equation}
and we denote by $\Omega_h$ the bounded component enclosed by $\Gamma_h$.
After decreasing $\delta$ if necessary, $\Gamma_h$ is a regular embedded
sphere, $\Omega_h$ is diffeomorphic to $B_1$, and
$\R^3\setminus\overline{\Omega_h}$ is diffeomorphic to
$\R^3\setminus\overline{B_1}$. In particular, both phases are simply
connected for every $h\in U^{k+2}_\delta$.

The outer unit normal and the surface Jacobian are
\begin{equation}\label{eq:radial-normal-jacobian}
	n_h\circ \Psi_h
	=
	\frac{(1+h)\omega-\nablaS h}
	{\sqrt{(1+h)^2+\abs{\nablaS h}^2}},
	\qquad
	J_h
	=
	(1+h)\sqrt{(1+h)^2+\abs{\nablaS h}^2}.
\end{equation}
These are the radial-graph specialisations of the normal-graph formulas in
\cite[pp.~55--57, equations~(2.42) and~(2.44)--(2.45)]{PruessSimonett2016}.
These depend smoothly on $h$ as maps
\[
	U^{k+2}_\delta\to \H^{k+1}(\Sph;\R^3),
	\qquad
	U^{k+2}_\delta\to \H^{k+1}(\Sph).
\]

We introduce two constraints: the first fixes the
volume of the enclosed domain to be that of the unit ball,
\[
	\mathcal{V}(h) = 0,\quad\mbox{where }\mathcal{V}(h) = |\Omega_h|-|B_1|,
\]
while the second fixes the vertical component of
the centroid at the origin,
\[
	\mathcal{C}_3(h) = 0,\quad\mbox{where }\mathcal{C}_3(h) = \int_{\Omega_h} y_3\dd y.
\]
Because in the radial graph
parametrisation $y=\Psi_h(\omega)=(1+h(\omega))\omega$ one has
\[
	\bigl((y\cdot n_h)\circ\Psi_h\bigr)J_h=(1+h)^3,
	\qquad
	y_3\circ\Psi_h=(1+h)Z,
\]
we find via the divergence theorem that
\begin{equation}\label{eq:vertical-centring_V}
	\cV(h)
	=
	\frac13\int_{\Gamma_h} y\cdot n_h\dd S
	-\frac{4\pi}{3}
	=
	\frac13\int_{\Sph}(1+h)^3\dd S-\frac{4\pi}{3},
\end{equation}
and
\begin{equation} \label{eq:vertical-centring_C}
	\cC_3(h)
	=
	\frac14\int_{\Gamma_h} y_3\, y\cdot n_h\dd S
	=
	\frac14\int_{\Sph}(1+h)^4 Z\dd S .
\end{equation}

The maps $\cV,\cC_3\colon U^{k+2}_\delta\to\R$ are smooth and
\[
	D\cV(0)\eta=\int_{\Sph}\eta\dd S,
	\qquad
	D\cC_3(0)\eta=\int_{\Sph}\eta Z\dd S .
\]
The map
\[
	h\mapsto \left(\cV(h),\cC_3(h)\right)
\]
has surjective differential at $h=0$, because its restriction to
$\Span\{1,Z\}$ is represented by the diagonal matrix
\[
	\begin{pmatrix}
		4\pi & 0              \\
		0    & \frac{4\pi}{3}
	\end{pmatrix}.
\]
Hence, after decreasing $\delta$ if necessary, the set
\[
	\cM^{k+2}
	=
	\{h\in U^{k+2}_\delta:\cV(h)=0,\ \cC_3(h)=0\}
\]
is a smooth codimension-two submanifold of $\H^{k+2}(\Sph)$.
Its tangent space at the round sphere is
\[
	T_0\cM^{k+2}
	=
	\set{\eta\in \H^{k+2}(\Sph):
		\int_{\Sph}\eta\dd S=0,\quad
		\int_{\Sph}\eta Z\dd S=0
	}.
\]

We shall use the following linear slice for the fixed-volume, vertically centred
constraint. For $j\in \N$ set
\begin{equation}\label{eq:linear-slice}
	\H^j_{\rm lin}
	:=
	\set{
		\xi\in \H^j(\Sph):
		\int_{\Sph}\xi\dd S=0,\quad
		\int_{\Sph}\xi Z\dd S=0
	}.
\end{equation}
For $j=k+2$, this space is $T_0\cM^{k+2}$.
We denote by
$\proj_{\rm lin}$ the $\L^2(\Sph)$-orthogonal projection of $\H^k$ onto
$\H^k_{\rm lin}$.

The manifold $\cM^{k+2}$ can be locally parametrised over its tangent space at
the origin. Indeed, for $\xi\in\H^{k+2}_{\rm lin}$ small, we solve
\begin{equation}\label{eq:constraint-correction-equations}
	\cV(\xi+c+dZ)=0,
	\qquad
	\cC_3(\xi+c+dZ)=0
\end{equation}
for $c,d\in\R$. Since
\[
	D_{(c,d)}
	\begin{pmatrix}
		\cV(\xi+c+dZ) \\
		\cC_3(\xi+c+dZ)
	\end{pmatrix}_{(0,0,0)}
	=
	\begin{pmatrix}
		4\pi & 0              \\
		0    & \frac{4\pi}{3}
	\end{pmatrix},
\]
the implicit-function theorem gives unique smooth functions
$c=c(\xi)$ and $d=d(\xi)$ with
\[
	c(0)=d(0)=0,
	\qquad
	Dc(0)=Dd(0)=0,
\]
and these functions satisfy \eqref{eq:constraint-correction-equations}.
Thus
\begin{equation}\label{eq:Theta-chart}
	\Theta(\xi):=\xi+c(\xi)+d(\xi)Z
\end{equation}
is a smooth local chart from $\H^{k+2}_{\rm lin}$ to $\cM^{k+2}$ near $0$.
Moreover,
\begin{equation}\label{eq:Theta-expansion}
	\Theta(0)=0,\qquad
	D\Theta(0)=\id_{\H^{k+2}_{\rm lin}},
	\qquad
	\Theta(\xi)-\xi=O\left(\norm{\xi}_{\H^{k+2}}^2\right).
\end{equation}

For later use we also record the tangent space at a general point of
$\cM^{k+2}$. Since
\begin{equation}\label{eq:radial-volume-weight}
	D\cV(h)[\eta]
	=
	\int_{\Sph}M_h\eta\dd S,
	\qquad
	M_h:=\bigl(\omega\cdot(n_h\circ\Psi_h)\bigr)J_h=(1+h)^2 ,
\end{equation}
and
\begin{equation}\label{eq:radial-C3-variation}
	D\cC_3(h)[\eta]
	=
	\int_{\Sph}(1+h)^3Z\eta\dd S
	=
	\int_{\Sph}(y_3\circ\Psi_h)M_h\eta\dd S ,
\end{equation}
we have
\begin{equation}\label{eq:radial-tangent-space}
	T_h\cM^{k+2}
	=
	\set{
		\eta\in \H^{k+2}(\Sph):
		\int_{\Sph}M_h\eta\dd S=0,\quad
		\int_{\Sph}(y_3\circ\Psi_h)M_h\eta\dd S=0
	}.
\end{equation}

Let us now turn to the differentiability of the elliptic problems with respect to $h$.

\begin{proposition}
	\label{prop:neumann-sobolev}
	After possibly decreasing $\delta>0$, the following holds.
	For every $h\in U^{k+2}_\delta$ and every $\alpha,V\in\R$, there are
	unique functions
	\[
		\ein_h=\ein(h,\alpha,V),
		\qquad
		\eout_h=\eout(h,\alpha,V),
	\]
	solving
	\[
		\Delta \ein_h=0
		\quad\text{in }\Omega_h,
		\qquad
		\partial_{n_h}\ein_h
		=
		W_{\alpha,V}\cdot n_h
		\quad\text{on }\Gamma_h,
	\]
	with the normalisation
	\[
		\int_{\Omega_h}\ein_h\dd y=0,
	\]
	and
	\[
		\Delta \eout_h=0
		\quad\text{in }\R^3\setminus\overline{\Omega_h},
		\qquad
		\partial_{n_h}\eout_h
		=
		W_{\alpha,V}\cdot n_h
		\quad\text{on }\Gamma_h,
	\]
	with
	\[
		\nabla\eout_h\in \L^2(\R^3\setminus\overline{\Omega_h}),
		\qquad
		\eout_h(y)\to0
		\quad\text{as }\abs y\to\infty .
	\]
	In fact,
	\[
		\eout_h(y)=O(\abs y^{-2}),
		\qquad
		\nabla\eout_h(y)=O(\abs y^{-3})
		\quad\text{as }\abs y\to\infty .
	\]
	Moreover, the pulled-back boundary gradients
	\[
		\Qin(h,\alpha,V)
		:=
		(\nabla\ein_h)\circ\Psi_h,
		\qquad
		\Qout(h,\alpha,V)
		:=
		(\nabla\eout_h)\circ\Psi_h
	\]
	define smooth maps
	\[
		\Qin,\Qout
		\colon
		U^{k+2}_\delta\times\R^2
		\to
		\H^{k+1}(\Sph;\R^3).
	\]
\end{proposition}

We set
\[
	\Hdot^1(D)
	:=
	\left\{
	u\in \H^1_{\mathrm{loc}}(D)\cap \L^6(D)
	:\nabla u\in \L^2(D;\mathbb R^3)
	\right\},
	\qquad
	\|u\|_{\Hdot^1(D)}:=\|\nabla u\|_{L^2(D)}.
\]
The \(\L^6\)-condition is the natural Sobolev normalisation: on an
exterior Lipschitz domain in \(\mathbb R^3\), every finite-energy class
modulo constants has a unique \(\L^6\)-representative.

\begin{proof}
	Set
	\[
		D_{\rm in}:=B_1,
		\qquad
		D_{\rm out}:=\R^3\setminus\overline{B_1}.
	\]
	Fix a linear extension operator $\mathsf{E}$ which, for every $r\ge k+2$,
	is bounded as a map
	\[
		\mathsf{E}:
		\H^r(\Sph;\R^3)\longrightarrow
		\H^{r+1/2}(\R^3;\R^3),
		\qquad
		(\mathsf{E}q)\big|_{\Sph}=q,
	\]
	and has support in a fixed annulus around $\Sph$. Such a single extension
	operator, bounded throughout this Sobolev scale, is obtained by localising the
	half-space Poisson extension in a finite smooth atlas and multiplying by a
	fixed cut-off which equals one near $\Sph$. We set
	\[
		\Phi_h(x):=x+\mathsf{E}(h\omega)(x).
	\]
	Then
	\[
		\Phi_h(\omega)=(1+h(\omega))\omega=\Psi_h(\omega)
		, \qquad \omega \in \Sph.
	\]
	Since $k+\frac52>\frac32+1$, Sobolev embedding and boundedness of $\mathsf{E}$
	show that, after decreasing $\delta$, one has
	\[
		\sup_{h\in U^{k+2}_\delta}
		\norm{D\Phi_h-I}_{\L^\infty}<1.
	\]
	The resulting lower Lipschitz bound makes $\Phi_h$ injective. Moreover,
	$\Phi_h$ equals the identity outside the fixed annulus, and hence
	$\abs{\Phi_h(x)}\to\infty$ as $\abs{x}\to\infty$. Thus inverse images of
	compact sets are compact. The global inverse theorem now shows that $\Phi_h$
	is an orientation-preserving global diffeomorphism mapping $B_1$ onto
	$\Omega_h$. Moreover,
	\[
		h\longmapsto\Phi_h-\id
	\]
	is smooth from $U^{k+2}_\delta$ to $\H^{k+\frac52}(\R^3;\R^3)$.

	Write
	\[
		\mathfrak j_h:=\det D\Phi_h,
		\qquad
		A_h:=
		\mathfrak j_hD\Phi_h^{-1}D\Phi_h^{-T}.
	\]
	The preceding choice of $\delta$ ensures $\mathfrak j_h>0$. The Sobolev
	multiplication theorem in dimension three gives smooth maps
	\[
		h\longmapsto \mathfrak j_h-1,
		\quad
		h\longmapsto A_h-I
		\quad\text{from}\quad
		U^{k+2}_\delta
		\quad\text{to}\quad
		\H^{k+\frac32}(\R^3),
	\]
	with the matrix-valued space for $A_h-I$. Both differences have
	support in one fixed annulus. In particular, $A_h$ is uniformly elliptic
	when $\delta$ is sufficiently small.

	Set
	\[
		G(h,\alpha,V)
		:=
		\left(
		W_{\alpha,V}(\Psi_h)\cdot(n_h\circ\Psi_h)
		\right)J_h.
	\]
	The radial-graph formulas and the Sobolev multiplication theorem on $\Sph$
	imply that
	\[
		G:
		U^{k+2}_\delta\times\R^2
		\longrightarrow
		\H^{k+1}(\Sph)
	\]
	is smooth. The divergence theorem gives
	\begin{equation}
		\label{eq:compatibility_condition_neumann}
		\int_{\Sph}G(h,\alpha,V)\dd S
		=
		\int_{\Gamma_h}W_{\alpha,V}\cdot n_h\dd S
		=
		\int_{\Omega_h}\Div W_{\alpha,V}\dd y
		=0.
	\end{equation}
	The change of variables	$u=\phi\circ\Phi_h$ shows
	\begin{equation} \label{eq:smoothness-change-of-variables}
		\int_{\Phi_h(D)}\nabla\phi\cdot\nabla(\zeta\circ\Phi_h^{-1})\dd y
		=
		\int_D (A_h\nabla u)\cdot\nabla\zeta\dd x
	\end{equation}
	for all test functions $\zeta\in \C^\infty_c(\mathbb R^3)$.
	Together with the surface change of variables defining $G$, this identifies
	the fixed equations below exactly with the physical Neumann problems. In the
	exterior problem the transformed normal has the opposite sign on the inner boundary.

	We first treat the interior problem. Introduce the fixed Banach space
	\[
		\mathscr Y_{\rm in}
		:=
		\left\{
		(f,g,c)\in
		\H^{k+\frac12}(B_1)
		\times\H^{k+1}(\Sph)
		\times\R:
		\int_{B_1}f\dd x+
		\int_{\Sph}g\dd S=0
		\right\}
	\]
	and define
	\[
		\mathscr L_h^{\rm in}:
		\H^{k+\frac52}(B_1)\longrightarrow\mathscr Y_{\rm in}
	\]
	by
	\[
		\mathscr L_h^{\rm in}u
		:=
		\left(
		-\Div(A_h\nabla u),
		(A_h\nabla u)\cdot\omega\big|_{\Sph},
		\int_{B_1}\mathfrak j_hu\dd x
		\right).
	\]
	The divergence theorem shows that the first two components satisfy the
	compatibility condition defining $\mathscr Y_{\rm in}$.

	At $h=0$, the operator $\mathscr L_0^{\rm in}$ is the usual Neumann
	operator on the ball, supplemented by the mean-value functional that fixes the additive constant.
	Lax--Milgram on $\H^1(B_1)/\R$, followed by the standard Neumann regularity estimate on the
	smooth ball, gives the isomorphism
	\[
		\mathscr L_0^{\rm in}:
		\H^{k+\frac52}(B_1)\longrightarrow \mathscr Y_{\rm in}.
	\]
	The multiplication and trace estimates give
	\[
		\norm{\mathscr L_h^{\rm in}-\mathscr L_0^{\rm in}}_
		{\mathcal L(\H^{k+\frac52}(B_1),\mathscr Y_{\rm in})}
		\le
		C\left(
		\norm{A_h-I}_{\H^{k+\frac32}}
		+
		\norm{\mathfrak j_h-1}_{\H^{k+\frac32}}
		\right).
	\]
	Thus, after decreasing $\delta$, $\mathscr L_h^{\rm in}$ is an
	isomorphism for every $h\in U^{k+2}_\delta$. Since
	$h\mapsto\mathscr L_h^{\rm in}$ is smooth in bounded-operator norm and
	inversion is smooth on the open set of Banach-space isomorphisms,
	\[
		h\longmapsto(\mathscr L_h^{\rm in})^{-1}
	\]
	is smooth. We may therefore define
	\[
		U_h^{\rm in}
		:=
			(\mathscr L_h^{\rm in})^{-1}
		\bigl(0,G(h,\alpha,V),0\bigr).
	\]
	Then $U_h^{\rm in}\in\H^{k+\frac52}(B_1)$ depends smoothly on
	$(h,\alpha,V)$ and satisfies
	\[
		\int_{B_1}\mathfrak j_hU_h^{\rm in}\dd x=0.
	\]
	After changing variables, this is $\int_{\Omega_h}\ein_h\dd y=0$ for $\ein_h=U_h^{\rm in}\circ \Phi_h^{-1}$.

	For the exterior problem choose $R>1$ so large that $\Phi_h=\id$ and
	$A_h=I$ in a neighbourhood of $\partial B_R$, for every
	$h\in U^{k+2}_\delta$, and put
	\[
		D_R:=B_R\setminus\overline{B_1}.
	\]
	Let $\Lambda_R^{\rm ext}$ be the exterior Dirichlet-to-Neumann operator on
	$\partial B_R$: for a spherical harmonics expansion, if
	\[
		v(R\omega)=\sum_{\ell,m}v_{\ell m}Y_\ell^m(\omega),
	\]
	then
	\[
		\Lambda_R^{\rm ext}v
		=
		-\sum_{\ell,m}\frac{\ell+1}{R}
		v_{\ell m}Y_\ell^m.
	\]
	It is bounded as a map
	\[
		\Lambda_R^{\rm ext}:
		\H^{k+2}(\partial B_R)
		\longrightarrow
		\H^{k+1}(\partial B_R).
	\]
	Define the space
	\[
		\mathscr X_{\rm out}
		:=
		\left\{
		u\in\H^{k+\frac52}(D_R):
		\partial_ru
		=
		\Lambda_R^{\rm ext}(u\big|_{\partial B_R})
		\text{ on }\partial B_R
		\right\},
	\]
	equipped with the $\H^{k+\frac52}(D_R)$ norm, and define
	\[
		\mathscr L_h^{\rm out}:
		\mathscr X_{\rm out}
		\longrightarrow
		\H^{k+\frac12}(D_R)\times\H^{k+1}(\Sph)
	\]
	by
	\[
		\mathscr L_h^{\rm out}u
		:=
		\left(
		-\Div(A_h\nabla u),
		(A_h\nabla u)\cdot(-\omega)\big|_{\Sph}
		\right).
	\]
	Here $-\omega$ is the outer normal of the annulus on its inner boundary.

	The operator $\mathscr L_0^{\rm out}$ is an isomorphism. Indeed, the
	condition at $\partial B_R$ says precisely that the radial derivative of
	$u$ agrees with that of the unique harmonic extension of its trace to
	$\R^3\setminus\overline{B_R}$ which decays at infinity. At the energy level,
	existence follows from Lax--Milgram applied to
	\[
		\int_{D_R}\nabla u\cdot\nabla v\dd x
		-
		\left\langle
		\Lambda_R^{\rm ext}(u\big|_{\partial B_R}),
		v\big|_{\partial B_R}
		\right\rangle.
	\]
	Since $-\Lambda_R^{\rm ext}$ has the strictly positive eigenvalues
	$(\ell+1)/R$, this form is coercive. Uniqueness follows from the same
	identity. Elliptic regularity for the annular problem with this exterior
	Dirichlet-to-Neumann condition, gives
	\[
		\norm{u}_{\H^{k+\frac52}(D_R)}
		\le
		C\left(
		\norm{-\Delta u}_{\H^{k+\frac12}(D_R)}
		+
		\norm{-\partial_ru\big|_{\Sph}}_{\H^{k+1}(\Sph)}
		\right),
	\]
	and hence
	\[
		\mathscr L_0^{\rm out}:
		\mathscr X_{\rm out}
		\longrightarrow
		\H^{k+\frac12}(D_R)\times\H^{k+1}(\Sph).
	\]

	Because $A_h=I$ near $\partial B_R$, the domain
	$\mathscr X_{\rm out}$ is independent of $h$. Moreover,
	\[
		\norm{\mathscr L_h^{\rm out}-\mathscr L_0^{\rm out}}_
		{\mathcal L(
			\mathscr X_{\rm out},
			\H^{k+\frac12}(D_R)\times\H^{k+1}(\Sph))}
		\le C\norm{A_h-I}_{\H^{k+\frac32}}.
	\]
	Thus $\mathscr L_h^{\rm out}$ is an isomorphism for all sufficiently
	small $h$, and its inverse depends smoothly on $h$. Set
	\[
		U_{h,R}^{\rm out}
		:=
			(\mathscr L_h^{\rm out})^{-1}
		\bigl(0,-G(h,\alpha,V)\bigr).
	\]
	The minus sign is due to the inner-boundary normal $-\omega$ of $D_R$.
	Extend $U_{h,R}^{\rm out}$ to $\R^3\setminus\overline{B_R}$ by the
	decaying exterior harmonic extension of its trace on $\partial B_R$.
	The annular solution and its exterior harmonic extension have the same
	trace and, by the boundary condition, the same radial derivative on
	$\partial B_R$; hence they define a global weak solution $U_h^{\rm out}$.
	The exterior Poisson operator is
	bounded from $\H^{k+2}(\partial B_R)$ into $\Hdot^1$ and into
	$\H^{k+\frac52}$
	on every bounded exterior annulus. Consequently,
	\[
		(h,\alpha,V)\longmapsto U_h^{\rm out}
	\]
	is smooth into
	\[
		\Hdot^1(D_{\rm out})
		\cap
		\H^{k+\frac52}(D_{\rm out}\cap B_{R_1})
	\]
	for every fixed $R_1>1$.

	The flux through $\partial B_R$ is
	\[
		\int_{\partial B_R}\partial_rU_h^{\rm out}\dd S
		=
		\int_{\Sph}G(h,\alpha,V)\dd S
		=0.
	\]
	In the exterior spherical-harmonic expansion
	\[
		U_h^{\rm out}(r,\omega)
		=
		\sum_{\ell,m}
		a_{\ell m}
		\left(\frac Rr\right)^{\ell+1}
		Y_\ell^m(\omega),
		\qquad r\ge R,
	\]
	the preceding identity is exactly the vanishing of the $\ell=0$
	coefficient. Hence the expansion starts at $\ell=1$, and therefore
	\[
		U_h^{\rm out}(x)=O(\abs{x}^{-2}),
		\qquad
		\nabla U_h^{\rm out}(x)=O(\abs{x}^{-3}).
	\]
	The energy identity proves uniqueness in $\Hdot^1(D_{\rm out})$.

	Finally, define
	\[
		\ein_h:=U_h^{\rm in}\circ\Phi_h^{-1},
		\qquad
		\eout_h:=U_h^{\rm out}\circ\Phi_h^{-1}.
	\]
	The identity
	\eqref{eq:smoothness-change-of-variables}, together with the definition of $G$,
	shows that $\ein_h$ and $\eout_h$ satisfy the harmonic equations and their
	Neumann boundary conditions. Since $\Phi_h$ is bi-Lipschitz and is the identity outside the
	fixed annulus, the exterior finite-energy condition and the displayed decay
	pass unchanged to $\eout_h$.
	Denoting the boundary trace on either side of $\Sph$ by $\gamma$, the chain rule
	on $\Sph$ gives
	\[
		(\nabla\ein_h)\circ\Psi_h
			=
			(D\Phi_h\big|_{\Sph})^{-T}
		\,\gamma(\nabla U_h^{\rm in}),
	\]
	and the identical formula holds for the exterior solution. Since
	\[
		\gamma(\nabla U_h^{\rm in}),
		\ \gamma(\nabla U_h^{\rm out})
		\in
		\H^{k+1}(\Sph;\R^3),
	\]
	and $\H^{k+1}(\Sph)$ is a Banach algebra, multiplication by
	$(D\Phi_h\big|_{\Sph})^{-T}$ preserves this space and depends smoothly on
	$h$. This proves the claimed smoothness of $\Qin$ and $\Qout$.
\end{proof}

The following lemma will be used in the regularity bootstrap.

\begin{lemma}
	\label{lem:neumann-sobolev-scale}
	Let $\delta>0$ be the fixed $\H^{k+2}$-radius in
	Proposition~\ref{prop:neumann-sobolev}. For every integer $q\ge k$, if
	\[
		h\in U_\delta^{k+2}\cap\H^{q+2}(\Sph),
	\]
	then the same energy solutions constructed in that proposition satisfy
	\[
		U_h^{\rm in}\in\H^{q+5/2}(B_1),
		\qquad
		U_h^{\rm out}\in\Hdot^1(\R^3\setminus\overline{B_1})
		\cap
		\H^{q+5/2}
		\bigl((\R^3\setminus\overline{B_1})\cap B_R\bigr)
	\]
	for every fixed $R>1$, and
	\[
		\Qin(h,\alpha,V),\ \Qout(h,\alpha,V)
		\in\H^{q+1}(\Sph;\R^3).
	\]
	Moreover,
	\[
		\Qin,\Qout:
		\bigl(U_\delta^{k+2}\cap\H^{q+2}(\Sph)\bigr)
		\times\R^2
		\longrightarrow
		\H^{q+1}(\Sph;\R^3)
	\]
	are smooth when the intersection is given its $\H^{q+2}$ topology. No
	smallness of $\norm h_{\H^{q+2}}$ is required.
\end{lemma}

\begin{proof}
	We use the extension operator $\mathsf{E}$ and the
	fixed-domain formulation from the proof of
	Proposition~\ref{prop:neumann-sobolev}. The condition
	$h\in U_\delta^{k+2}$ ensures that $\Phi_h$ is a diffeomorphism and that
	$A_h$ is uniformly elliptic. If, in addition,
	$h\in\H^{q+2}(\Sph)$, then
	\[
		A_h-I,\ \mathfrak j_h-1\in\H^{q+3/2}(\R^3),
		\qquad
		G(h,\alpha,V)\in\H^{q+1}(\Sph),
	\]
	with smooth dependence in the $\H^{q+2}$ topology.

	Choose a truncation radius $R_0>1$ as in the preceding proof. For each
	fixed such $h$, the same Lax--Milgram arguments give
	existence and uniqueness for the interior and exterior problems with
	general data, subject to the interior compatibility condition. Standard
	variable-coefficient Neumann regularity on $B_1$ and elliptic regularity
	on $D_{R_0}$ with the fixed exterior Dirichlet-to-Neumann condition then
	promote these solutions to $\H^{q+5/2}$. These estimates follow by boundary flattening and tangential difference
	quotients, using the equation and the Neumann boundary condition to control
	the remaining normal derivatives. Their constants may depend
	locally on $\norm h_{\H^{q+2}}$, but this norm need not be small.
	Consequently, the corresponding operators
	$\mathscr L_h^{\rm in}$ and $\mathscr L_h^{\rm out}$ between the
	Sobolev spaces with solution regularity $\H^{q+5/2}$ are isomorphisms.
	The domain of the exterior operator remains fixed because $A_h=I$ near
	$\partial B_{R_0}$.

	The coefficient maps and the datum $G$ depend smoothly on
	$(h,\alpha,V)$ in the indicated Sobolev spaces. Hence the two operators
	depend smoothly on $h$ in bounded-operator norm, and smoothness of
	inversion gives the asserted smooth dependence of the fixed-domain
	solutions. The fixed exterior harmonic continuation gives the required
	regularity on every bounded exterior annulus. Finally,
	\[
		(\nabla\ein_h)\circ\Psi_h
		=
		\left.D\Phi_h^{-T}\nabla U_h^{\rm in}\right|_{\Sph},
		\qquad
		(\nabla\eout_h)\circ\Psi_h
		=
		\left.D\Phi_h^{-T}\nabla U_h^{\rm out}\right|_{\Sph}.
	\]
	The trace theorem and Sobolev multiplication yield the claimed
	$\H^{q+1}$ regularity and smooth dependence of $\Qin$ and $\Qout$.
\end{proof}

\subsection{Variational structure of the Bernoulli equation}

In this section we set up the Bernoulli equation, i.e.\ the fourth equation of \eqref{eq:main-problem-nondimensional}, as a function of the radial graph $h$, and record its underlying variational structure.
The shape-derivative identities are first derived for
$\C^\infty$ radial graphs and $\C^\infty$ variations. By
Proposition~\ref{prop:neumann-sobolev} and the smoothness of the maps
constructed below, both sides of these identities extend continuously to
$h\in\cM^{k+2}$ and $\eta\in T_h\cM^{k+2}$. Thus the formulas
obtained by the smooth calculation are valid in the Sobolev scale.

For a velocity $q$ define
\begin{equation*}
	F_{\alpha,V}(y,q)
	=\frac12\abs{q-Ve_3}^2+\alpha y\cdot(e_3\times q).
\end{equation*}
We call the first term the kinetic part and the second term will be referred to as the rotational part.
On the boundary, the normal component of $q=\nabla\phi$ is not an independent
unknown: by the Neumann condition, $q\cdot n=W_{\alpha,V}\cdot n$. Thus, when
one decomposes $q=q_t+(q\cdot n)n$, the only freely determined part of the
boundary gradient is its tangential component $q_t$.

Using the identity
$y\cdot(e_3\times q)=-(e_3\times y)\cdot q$, this may also be written as
\begin{equation*}
	F_{\alpha,V}(y,q)
	=\frac12\abs{q}^2-W_{\alpha,V}(y)\cdot q+\frac{V^2}{2}.
\end{equation*}
For each radial graph $\Gamma_h$, we define the Bernoulli density
\begin{equation*}
	B_h(y)
	:=
	\tau F_{\alpha,V}(y,\nabla\ein_h(y))
	-F_{\alpha,V}(y,\nabla\eout_h(y))
	+H_h(y),
	\qquad y\in\Gamma_h,
\end{equation*}
Here $H_h$ is the mean curvature on $\Gamma_h$; see
\eqref{eq:curvature-radial-graph}.
The constant $c$ in the Bernoulli condition removes the average of this
quantity. Recall
$M_h(\omega)=\bigl(\omega\cdot(n_h\circ\Psi_h)(\omega)\bigr)J_h(\omega)$,
where $J_h\dd S$ is the surface measure on $\Gamma_h$ and $n_h$ is the
outer unit normal of $\Omega_h$. Then
$D\cV(h)[\eta]=\int_{\Sph}M_h\eta\dd S$ by
\eqref{eq:radial-volume-weight}.
Define
\begin{equation} \label{eq:def-Pi-h}
	\avg{f}_{M_h}
	=\frac{\int_{\Sph}fM_h\dd S}{\int_{\Sph}M_h\dd S},
	\qquad
	\Pi_h f=f-\avg{f}_{M_h}.
\end{equation}
The Bernoulli system \eqref{eq:main-problem-nondimensional}, including the
volume and vertical-centring normalisations, is equivalent to the existence
of $h\in\cM^{k+2}$ and $\alpha,V\in\R$ satisfying
\begin{equation}\label{eq:Bernoulli-map}
	\cB(h,\alpha,V):=\Pi_h\left(B_h\circ\Psi_h\right)=0,
	\qquad \Psi_h(\omega)=(1+h(\omega))\omega,\quad \omega\in\Sph.
\end{equation}
If \eqref{eq:Bernoulli-map} holds, the constant $c$ is the corresponding
omitted weighted average.

\begin{proposition}
	\label{prop:smooth-bernoulli-map}
	For $k\ge4$ and $\delta>0$ sufficiently small, the map
	\[
		(h,\alpha,V)
		\mapsto
		B_h\circ\Psi_h
	\]
	is smooth from
	\[
		U^{k+2}_\delta\times\R^2
		\quad\text{to}\quad
		\H^k(\Sph).
	\]
	Consequently
	\[
		\cB
		\colon
		U^{k+2}_\delta\times\R^2
		\to
		\H^k(\Sph),
		\qquad
		\cB(h,\alpha,V)
		=
		\Pi_h(B_h\circ\Psi_h),
	\]
	is smooth.
\end{proposition}

\begin{proof}
	By Proposition~\ref{prop:neumann-sobolev},
	\[
		\Qin(h,\alpha,V)
		=
		(\nabla\ein_h)\circ\Psi_h,
		\qquad
		\Qout(h,\alpha,V)
		=
		(\nabla\eout_h)\circ\Psi_h
	\]
	depend smoothly on $(h,\alpha,V)$ as $\H^{k+1}(\Sph;\R^3)$-valued maps.
	Also
	\[
		h\mapsto \Psi_h,\qquad
		h\mapsto n_h\circ\Psi_h,\qquad
		h\mapsto J_h
	\]
	are smooth in the Sobolev spaces stated above.

	The mean curvature of the radial graph can be written locally as a smooth
	rational expression in
	\[
		1+h,\qquad \nablaS h,\qquad \nablaS^2 h,
	\]
	with denominator bounded away from zero for $h\in U^{k+2}_\delta$; see
	\eqref{eq:curvature-radial-graph}.
	Up to the opposite sign convention for mean curvature, the corresponding
	general normal-graph formula and ellipticity statement are given in
	\cite[pp.~60--61, equations~(2.48)--(2.49)]{PruessSimonett2016}.
	Therefore
	\[
		h\mapsto H_h\circ\Psi_h
	\]
	is smooth from $U^{k+2}_\delta$ to $\H^k(\Sph)$.

	Since $\H^k(\Sph)$ is a Banach algebra and
	$\Qin,\Qout\in\H^{k+1}$, the functions
	\[
		F_{\alpha,V}(\Psi_h,\Qin),
		\qquad
		F_{\alpha,V}(\Psi_h,\Qout)
	\]
	depend smoothly on $(h,\alpha,V)$ as $\H^{k+1}$-valued, hence also as
	$\H^k$-valued, functions. 	Thus
	\[
		B_h\circ\Psi_h
		=
		\tau F_{\alpha,V}(\Psi_h,\Qin)
		-
		F_{\alpha,V}(\Psi_h,\Qout)
		+
		H_h\circ\Psi_h
	\]
	is smooth into $\H^k(\Sph)$.

	It remains only to check the projection. Since $M_h=(1+h)^2$,
	the map $h\mapsto M_h$ is smooth into $\H^{k+2}(\Sph)$, and
	$\int_{\Sph}M_h\dd S$ stays bounded away from zero. Hence
	\[
		\Pi_h f
		=
		f-
		\frac{\int_{\Sph}fM_h\dd S}{\int_{\Sph}M_h\dd S}
	\]
	depends smoothly on $(h,f)\in U^{k+2}_\delta\times\H^k(\Sph)$.
	Thus $\cB$ is smooth.
\end{proof}

Let us now turn to the variational structure. We define
\begin{equation*}
	\scrL_{\alpha,V}(y,p)
	=\frac12 \abs{p}^2-W_{\alpha,V}(y)\cdot p,
\end{equation*}
so that
\begin{equation}\label{eq:F-L-relation}
	F_{\alpha,V}(y,p)=\scrL_{\alpha,V}(y,p)+\frac{V^2}{2}.
\end{equation}

Let $D\in\set{\Omega,\R^3\setminus\overline{\Omega}}$. We denote
by $\nu_D$ the unit normal pointing out of $D$. In the exterior case,
$\partial D=\Gamma$ denotes the compact inner boundary, oriented by
$\nu_D=-n$. No boundary integral at infinity is included. For fixed $(\alpha,V)$ and a suitably smooth function $u$, define the one-phase action
\begin{equation*}
	\scrA_D(u,\alpha,V)
	:=\frac12\int_D \abs{\nabla u}^2\dd y
	-\int_{\partial D}u\,W_{\alpha,V}\cdot\nu_D\dd S .
\end{equation*}
The Neumann compatibility identity
\[
	\int_{\partial D}W_{\alpha,V}\cdot\nu_D\dd S=0
\]
holds for both phases by \eqref{eq:compatibility_condition_neumann}. Hence adding a constant to $u$
does not change the value of $\scrA_D(\cdot,\alpha,V)$.

The first variation with respect to $u$ is
\[
	D_u\scrA_D(u,\alpha,V)[v]
	=\int_D \nabla u\cdot\nabla v\dd y
	-\int_{\partial D}v\,W_{\alpha,V}\cdot\nu_D\dd S .
\]
After integration by parts, critical points satisfy
\begin{equation*}
	\Delta u=0\quad\text{in }D,
	\qquad
	\partial_{\nu_D}u=W_{\alpha,V}\cdot\nu_D
	\quad\text{on }\partial D .
\end{equation*}
For both phases the latter is
\[
	\partial_n u=W_{\alpha,V}\cdot n
	\quad\text{on }\Gamma,
\]
because $\nu_D=n$ or $\nu_D=-n$.

Let $u_D^{\alpha,V}$ denote the corresponding solution, with decay at infinity
in the exterior case, and define the reduced one-phase value
\[
	\scrN_D(\alpha,V)
	:=
	\scrA_D(u_D^{\alpha,V},\alpha,V).
\]
When $(\alpha,V)$ is fixed, we write simply $u_D$ for
$u_D^{\alpha,V}$.
The two-phase reduced action is
\begin{equation}\label{eq:two_phase_reduced_action_definition}
	\scrE(\Omega,\alpha,V)
	:=
	\abs{\partial\Omega}
	+\tau\scrN_{\Omega}(\alpha,V)
	+\scrN_{\R^3\setminus\overline\Omega}(\alpha,V).
\end{equation}
\begin{proposition}
	\label{prop:shape-derivative}
	Fix $(\alpha,V)$. Let $(\Omega_t)_t$ be a family of smooth deformations of
	$\Omega=\Omega_0$ with boundaries $(\Gamma_t)_t$ moving at $t=0$ with
	normal velocity $\zeta n$, where $n$ is the outer normal of $\Omega$ at
	$\Gamma=\Gamma_0$. Then
	\begin{equation}\label{eq:shape-derivative-full}
		\frac{\ddd}{\ddd t}\scrE(\Omega_t,\alpha,V)\Big|_{t=0}
		=
		\int_\Gamma
		\left[
			\tau \scrL_{\alpha,V}(y,\nabla\ein)
			-\scrL_{\alpha,V}(y,\nabla\eout)
			+H
			\right]\zeta\dd S .
	\end{equation}
\end{proposition}

\begin{proof}
	We first compute the shape derivative of one phase. Let $D_t$ be a smooth
	deformation of $D$ with normal velocity $\zeta_D\nu_D$. Since $u_D$
	is a critical point of $\scrA_D(\cdot,\alpha,V)$, differentiating the
	reduced value $\scrN_{D_t}(\alpha,V)$ is the same as differentiating
	$\scrA_{D_t}(u_D,\alpha,V)$ with $u_D$ frozen and suitably extended to a
	neighbourhood of $\overline D$.

	By the Reynolds transport theorem,
	\[
		\frac{\ddd}{\ddd t}
		\frac12\int_{D_t}\abs{\nabla u_D}^2\dd y\Big|_{t=0}
		=
		\int_{\partial D}
		\frac12\abs{\nabla u_D}^2\zeta_D\dd S .
	\]
	For the boundary term, the flux form of the Reynolds transport theorem also implies
	\begin{equation}\label{eq:flux-transport}
		\left.\frac{\ddd}{\ddd t}\right|_{t=0}
		\int_{\partial D_t}A\cdot\nu_{D_t}\dd S
		=
		\int_{\partial D}\Div A\,\zeta_D\dd S
	\end{equation}
	for every sufficiently regular vector field $A$. In order to see this in the exterior case,
	choose, without changing the prescribed normal velocity
	on the compact interface, a deformation field supported in a fixed collar of
	$\Gamma$, and fix $\scrR$ so large that this collar lies in $B_{\scrR}$.
	Apply the flux transport identity \eqref{eq:flux-transport} to
	$D_t\cap B_{\scrR}$. Its boundary is the
	moving inner boundary together with the fixed sphere $S_{\scrR}$. Since $u_D$
	is frozen and neither $S_{\scrR}$ nor its normal moves,
	\[
		\left.\frac{\ddd}{\ddd t}\right|_{t=0}
		\int_{S_{\scrR}}u_DW_{\alpha,V}\cdot\nu\dd S=0.
	\]
	Subtracting this zero derivative gives the flux identity on the
	moving compact boundary and hence \eqref{eq:one-phase-shape} for the exterior
	phase.

	In particular, with
	$A=u_DW_{\alpha,V}$ and $\Div W_{\alpha,V}=0$, this becomes
	\[
		\frac{\ddd}{\ddd t}
		\int_{\partial D_t}u_DW_{\alpha,V}\cdot\nu_{D_t}\dd S\Big|_{t=0}
		=
		\int_{\partial D}
		W_{\alpha,V}\cdot\nabla u_D\,\zeta_D\dd S .
	\]
	Therefore
	\begin{align}\label{eq:one-phase-shape}
		\frac{\ddd}{\ddd t}\scrN_{D_t}(\alpha,V)\Big|_{t=0}
		 & =
		\int_{\partial D}
		\left(
		\frac12\abs{\nabla u_D}^2
		-W_{\alpha,V}\cdot\nabla u_D
		\right)\zeta_D\dd S \\
		 & =
		\int_{\partial D}
		\scrL_{\alpha,V}(y,\nabla u_D)\zeta_D\dd S . \nonumber
	\end{align}

	We now apply \eqref{eq:one-phase-shape} to the two phases. For the interior
	phase, $\zeta_D=\zeta$. For the exterior phase, the outer normal of the
	exterior domain is $\nu_D=-n$, hence the same motion of $\Gamma$ has
	normal velocity $\zeta_D=-\zeta$ relative to the exterior domain. Thus
	the two one-phase contributions combine as
	\[
		\int_\Gamma \tau \scrL_{\alpha,V}(y,\nabla\ein)\zeta\dd S
		-
		\int_\Gamma \scrL_{\alpha,V}(y,\nabla\eout)\zeta\dd S .
	\]

	Finally, with our convention $H=\Div_\Gamma n$, the standard first-variation
	formula for the perimeter is
	\[
		\frac{\ddd}{\ddd t}\abs{\Gamma_t}\Big|_{t=0}
		=
		\int_\Gamma H\zeta\dd S .
	\]

	This proves \eqref{eq:shape-derivative-full}.
\end{proof}

We now rewrite \eqref{eq:shape-derivative-full} in terms of the radial graphs
\eqref{eq:radial-graph-chart}. In a slight abuse of notation, define
$\scrE:U^{k+2}_\delta\times\R^2\to\R$ by
\begin{align}\label{eq:graph_parametrized_two_phase_action}
	\scrE(h,\alpha,V):=\scrE(\Omega_h,\alpha,V)=\abs{\partial\Omega_h}
	+\tau\scrN_{\Omega_h}(\alpha,V)
	+\scrN_{\R^3\setminus\overline\Omega_h}(\alpha,V).
\end{align}
Recalling the Bernoulli map \eqref{eq:Bernoulli-map}, the following lemma
characterises solutions of \eqref{eq:main-problem-nondimensional} as critical
points.
\begin{lemma}\label{lem:variational_structure}
	The functional $\scrE$ is smooth with derivative given by
	\begin{align*}
		D_h\scrE(h,\alpha,V)[\eta]=\int_{\Sph} \cB(h,\alpha,V)M_h \eta\dd S
	\end{align*}
	for all $h\in U_\delta^{k+2}$, $\alpha,V\in\R$, and
	$\eta\in \H^{k+2}(\Sph)$ satisfying
	$\int_{\Sph} M_h\eta\dd S=0$. Moreover, a function
	$h\in\cM^{k+2}$ is a critical point of
	$\scrE(\cdot,\alpha,V)\big|_{\cM^{k+2}}$ if and only if
	$\cB(h,\alpha,V)=0$.
\end{lemma}
\begin{proof}
	Define the reduced Bernoulli density associated with the action by
	\begin{equation}
		E_h
		:=
		\tau \scrL_{\alpha,V}(y,\nabla\ein_{h})
		-\scrL_{\alpha,V}(y,\nabla\eout_{h})
		+H_h ,
	\end{equation}
	We first establish smoothness of the reduced action. For either
	phase $D$, testing the weak Neumann equation with its solution $u_D$ gives
	\[
		\int_D\abs{\nabla u_D}^2\dd y
		=
		\int_{\partial D}u_DW_{\alpha,V}\cdot\nu_D\dd S .
	\]
	Here the exterior orientation $\nu_D=-n$ is already included, and
	$u_D\in\Hdot^1(D)$ is an admissible test function in the exterior problem.
	It follows from the definition of the one-phase action that
	\begin{equation}\label{eq:reduced-one-phase-dirichlet-energy}
		\scrN_D(\alpha,V)
		=
		-\frac12\int_D\abs{\nabla u_D}^2\dd y .
	\end{equation}

	Let $\Phi_h$, $A_h$, $U_h^{\rm in}=\ein_h\circ\Phi_h$, and
	$U_h^{\rm out}=\eout_h\circ\Phi_h$ be the fixed-domain objects constructed
	in the proof of Proposition~\ref{prop:neumann-sobolev}. By
	\eqref{eq:smoothness-change-of-variables} and
	\eqref{eq:reduced-one-phase-dirichlet-energy},
	\begin{align}\label{eq:fixed-domain-reduced-action}
		\scrE(h,\alpha,V)
		={} & \int_{\Sph}J_h\dd S
		-\frac{\tau}{2}\int_{B_1}
		A_h\nabla U_h^{\rm in}\cdot\nabla U_h^{\rm in}\dd x \nonumber \\
		    & -\frac12\int_{\R^3\setminus\overline{B_1}}
		A_h\nabla U_h^{\rm out}\cdot\nabla U_h^{\rm out}\dd x .
	\end{align}
	The fixed-domain construction in the proof of that proposition shows that
	$U_h^{\rm in}$ depends smoothly on $(h,\alpha,V)$ in
	$\H^{k+5/2}(B_1)$ and that $U_h^{\rm out}$ depends smoothly in
	$\Hdot^1(\R^3\setminus\overline{B_1})$. Moreover,
	$h\mapsto A_h-I$ is smooth into
	$\H^{k+3/2}(\R^3)\hookrightarrow\L^\infty(\R^3)$ and has support in a
	fixed annulus. Writing $A_h=I+(A_h-I)$, smoothness of the two energy terms
	follows from the continuous bilinear form
	\[
		(u,v)\longmapsto\int\nabla u\cdot\nabla v\dd x
	\]
	and the continuous trilinear form
	\[
		(K,u,v)\longmapsto\int K\nabla u\cdot\nabla v\dd x
	\]
	on $\L^\infty\times\Hdot^1\times\Hdot^1$ (and on the corresponding
	bounded-domain spaces). Since $h\mapsto J_h$ is smooth as well,
	\eqref{eq:fixed-domain-reduced-action} proves that
	\[
		\scrE:U_\delta^{k+2}\times\R^2\longrightarrow\R
	\]
	is smooth.

	We next identify its derivative. First suppose that $h$ and the
	radial variation are smooth. Let
	$t\mapsto h+\eta_t\in U_\delta^{k+2}$ be a smooth path with
	$\eta_0=0$ and tangent vector
	\[
		\eta=\left.\frac{\ddd}{\ddd t}\right|_{t=0}\eta_t.
	\]
	Viewed as a deformation of $\Gamma_h$, this path generates the normal
	velocity at $t=0$
	\[
		\zeta=\eta\bigl(\omega\cdot(n_h\circ\Psi_h)\bigr).
	\]
	Recalling
	$M_h=\bigl(\omega\cdot(n_h\circ\Psi_h)\bigr)J_h=(1+h)^2$ from
	\eqref{eq:radial-volume-weight}, Proposition~\ref{prop:shape-derivative}
	therefore gives
	\begin{equation}\label{eq:reduced-action-first-variation}
		D_h\scrE(h,\alpha,V)[\eta]
		=
		\int_{\Sph}
		(E_h\circ\Psi_h)M_h\eta\dd S .
	\end{equation}
	The density $B_h$ and $E_h$ differ only by the constant
	$(1-\tau)V^2/2$; see \eqref{eq:F-L-relation}. Proposition~\ref{prop:smooth-bernoulli-map}
	therefore shows that
	$(h,\alpha,V)\mapsto M_h(E_h\circ\Psi_h)$ is smooth into $\H^k(\Sph)$.
	Consequently, both sides of \eqref{eq:reduced-action-first-variation} depend
	continuously on
	\[
		(h,\alpha,V,\eta)
		\in U_\delta^{k+2}\times\R^2\times\H^{k+2}(\Sph).
	\]
	Density of $\C^\infty(\Sph)$ in $\H^{k+2}(\Sph)$ extends the identity from
	smooth graphs and variations to the stated Sobolev class. Thus
	$M_h(E_h\circ\Psi_h)$ is the $\L^2(\Sph)$-representer of the unconstrained
	differential of $\scrE(\cdot,\alpha,V)$.

	Let us now consider variations with $\int_{\Sph}M_h\eta\dd S=0$.
	With the projection $\Pi_h$ defined in \eqref{eq:def-Pi-h}, this means
	$\Pi_h\eta=\eta$. Since $B_h$ and $E_h$ differ only by a constant, it follows
	that
	\begin{align*}
		D_h\scrE(h,\alpha,V)[\eta]
		 & =
		\int_{\Sph}
		(B_h\circ\Psi_h)M_h\eta\dd S=\int_{\Sph}
		(B_h\circ\Psi_h)M_h\Pi_h\eta\dd S \\
		 & =\int_{\Sph}
		\Pi_h[B_h\circ\Psi_h]M_h\eta\dd S=\int_{\Sph}
		\cB(h,\alpha,V)M_h\eta\dd S
	\end{align*}
	for such $\eta$. Equivalently, there exists $\lambda_h\in\R$ such that the
	$\L^2(\Sph)$-representer of the differential of
	$\scrE(\cdot,\alpha,V)$ can be written as
	\begin{align}\label{eq:L2_gradient_scrE}
		\nabla^{\L^2}_h\scrE(h,\alpha,V)=(E_h\circ\Psi_h)M_h=\cB(h,\alpha,V)M_h+\lambda_h M_h.
	\end{align}

	Now let $h\in\cM^{k+2}$ be a critical point of
	$\scrE(\cdot,\alpha,V)\big|_{\cM^{k+2}}$. By
	\eqref{eq:radial-tangent-space} and the preceding expression, this is
	equivalent to the existence of Lagrange multipliers $\lambda,\mu\in\R$ such
	that
	\begin{equation}\label{eq:lagrange_multipliers_E}
		E_h=\lambda+\mu y_3
		\qquad
		\text{on }\Gamma_h .
	\end{equation}

	Substitution of \eqref{eq:lagrange_multipliers_E} into
	\eqref{eq:L2_gradient_scrE}, followed by integration over $\Sph$, shows that
	the equation
	$\cB(h,\alpha,V)=\Pi_h(B_h\circ\Psi_h)=0$ is equivalent to $\mu=0$. We now
	show that the multiplier $\mu$ associated with the vertical-centring
	constraint vanishes, using invariance under vertical translations. Consider
	$T_t(y)=y+te_3$ in Proposition~\ref{prop:shape-derivative}, that is,
	$\Omega_t=\Omega_h+te_3$ in \eqref{eq:shape-derivative-full}.
	Since
	\[
		W_{\alpha,V}(y+te_3)=W_{\alpha,V}(y),
	\]
	vertical translation leaves the reduced action invariant: the volume and area
	are unchanged, and the translated Neumann potentials are the potentials of the
	translated domain. Thus the derivative of the reduced action in this direction
	vanishes, and \eqref{eq:shape-derivative-full} gives
	\begin{equation}\label{eq:tranlational_EEE}
		0=\int_{\Gamma_h}E_h e_3\cdot n_h\dd S.
	\end{equation}
	Substituting $E_h=\lambda+\mu y_3$, we obtain
	\[
		0
		=
		\lambda\int_{\Gamma_h}e_3\cdot n_h\dd S
		+
		\mu\int_{\Gamma_h}y_3\,e_3\cdot n_h\dd S .
	\]
	The first integral vanishes by the divergence theorem:
	\[
		\int_{\Gamma_h}e_3\cdot n_h\dd S
		=
		\int_{\Omega_h}\Div e_3\dd y
		=0 .
	\]
	The second integral equals the enclosed volume:
	\[
		\int_{\Gamma_h}y_3\,e_3\cdot n_h\dd S
		=
		\int_{\Omega_h}\partial_3 y_3\dd y
		=
		\abs{\Omega_h}.
	\]
	Thus $0=\mu \abs{\Omega_h}$. Since $\abs{\Omega_h}>0$, we have $\mu=0$,
	as desired.
\end{proof}

\section{Shape calculus}
\label{sec:shape-calculus}

In this section we collect the shape-calculus expansions at the round
sphere that are used repeatedly in the rest of the manuscript. Throughout, the reference sphere has
unit radius, points on $\Sph$ are denoted by
$\omega=(X,Y,Z)$, and
\[
	\Psi_h(\omega)=(1+h(\omega))\omega .
\]
For a one-parameter deformation we write $h=t\eta$. All remainders are meant in
the Sobolev spaces dictated by the expression under consideration, locally
uniformly for $\eta$ in bounded subsets of $\H^{k+2}(\Sph)$ and for bounded
$\alpha,V$.
For general background on normal-graph geometry, normal variations,
and transport on moving hypersurfaces, see
\cite[pp.~55--85]{PruessSimonett2016}. The book uses the opposite sign
convention for mean curvature.

We use the following notation to extract a coefficient in a Taylor expansion. Let
$\eps_1,\ldots,\eps_m$ denote the small parameters which are
being expanded at the given point of the argument. If
\[
	F(\eps_1,\ldots,\eps_m)
	=
	\sum_{\beta\in\N_0^m}
	F_\beta\,
	\eps_1^{\beta_1}\cdots \eps_m^{\beta_m},
\]
then, for the monomial
$\frB=\eps_1^{\beta_1}\cdots \eps_m^{\beta_m}$, we write
\[
	[F]_{\frB}:=F_\beta .
\]
For example, if
\[
	F=F_0+tF_t+sF_s+t s F_{ts}+t^2F_{t^2}+\cdots ,
\]
then
\[
	[F]_t=F_t,
	\qquad
	[F]_{ts}=F_{ts}
	=
	\partial_t\partial_sF\big|_{t=s=0},
	\qquad
	[F]_{t^2}=F_{t^2}
	=
	\frac12\partial_t^2F\big|_{t=0}.
\]
All other small parameters are set equal to zero unless they are explicitly
displayed in the subscript.

\subsection{Some identities}
We first record the spherical calculus identities used below. The angular
vector field around the $e_3$-axis is
\begin{equation}
	\label{eq:dphi-def}
	\partial_\varphi=(e_3\times\omega)\cdot\nablaS,
	\qquad
	\partial_\varphi X=-Y,
	\qquad
	\partial_\varphi Y=X,
	\qquad
	\partial_\varphi Z=0 .
\end{equation}
Moreover,
\begin{equation}
	\nablaS \omega_i=e_i-\omega_i\omega,
	\qquad
	e_i=\omega_i\omega+\nablaS\omega_i,
	\qquad
	\nablaS\omega_i\cdot\nablaS\omega_j=\delta_{ij}-\omega_i\omega_j .
\end{equation}
In particular,
\begin{equation}
	\label{eq:low-mode-products}
	e_3=Z\omega+\nablaS Z,
	\qquad
	\abs{\nablaS Z}^2=1-Z^2,
	\qquad
	\nablaS X\cdot\nablaS Z=-XZ,
\end{equation}
\begin{equation}
	\label{eq:XZ-YZ-products}
	\nablaS(XZ)\cdot\nablaS Z=X(1-2Z^2),
	\qquad
	\nablaS(YZ)\cdot\nablaS Z=Y(1-2Z^2).
\end{equation}
We use the second Legendre polynomial
\begin{equation}
	P_2(Z)=\frac{3Z^2-1}{2}.
\end{equation}
The functions $X$, $Y$ are degree-one spherical harmonics, while
\begin{equation}
	\label{eq:degree-three-real}
	X\left(Z^2-\frac15\right),
	\qquad
	Y\left(Z^2-\frac15\right)
\end{equation}
are degree-three spherical harmonics with azimuthal number one.
For a real or complex spherical harmonic $Y_\ell^m$ of azimuthal
number $m$,
\begin{equation}
	\label{eq:spherical-harmonic-eigen}
	\Delta_{\Sph}Y_\ell^m=-\ell(\ell+1)Y_\ell^m,
	\qquad
	\partial_\varphi^2Y_\ell^m=-m^2Y_\ell^m.
\end{equation}
Moreover,
\begin{equation}
	\label{eq:centroid-integrals}
	\int_{\Sph}X\omega\dd S=\frac{4\pi}{3}e_1,
	\qquad
	\int_{\Sph}Y\omega\dd S=\frac{4\pi}{3}e_2,
	\qquad
	\int_{\Sph}Z\omega\dd S=\frac{4\pi}{3}e_3 .
\end{equation}

Finally, for integers $\ell\ge0$ and $\abs{m}\le\ell$, we set
\begin{equation}
	\label{eq:a-coefficients}
	a_{\ell,+}^m=
	\left(\frac{(\ell+1)^2-m^2}{(2\ell+1)(2\ell+3)}\right)^{1/2},
	\qquad
	a_{\ell,-}^m=
	\left(\frac{\ell^2-m^2}{(2\ell-1)(2\ell+1)}\right)^{1/2}.
\end{equation}
For these indices, let $Y_\ell^m$
denote the complex $\L^2(\Sph)$-normalised spherical harmonics, with
phases chosen according to the standard
Condon--Shortley convention, so that
$\partial_\varphi Y_\ell^m=imY_\ell^m$.
In \eqref{eq:Z-recursion}--\eqref{eq:gradZ-recursion} and in every subsequent
formula obtained from them, any term involving $Y_{\ell-1}^m$ or one of its
derivatives is omitted whenever $\ell-1<\abs{m}$. This
includes every mode with $\abs{m}=\ell$ and the constant mode
$(\ell,m)=(0,0)$. With this convention,
\begin{equation}
	\label{eq:Z-recursion}
	ZY_\ell^m=a_{\ell,+}^mY_{\ell+1}^m+a_{\ell,-}^mY_{\ell-1}^m,
\end{equation}
and
\begin{equation}
	\label{eq:gradZ-recursion}
	\nablaS Z\cdot\nablaS Y_\ell^m
	=
	-\ell a_{\ell,+}^mY_{\ell+1}^m
	+(\ell+1)a_{\ell,-}^mY_{\ell-1}^m .
\end{equation}
Relation \eqref{eq:Z-recursion} can, for example, be found in
\cite[Sec.~5.7, Eq.~(2)]{MR1022665}. Relation
\eqref{eq:gradZ-recursion} follows by applying $\Delta_{\Sph}$ to
\eqref{eq:Z-recursion} and using \eqref{eq:spherical-harmonic-eigen}.

\subsection{The geometric quantities}
We first recall from \eqref{eq:radial-normal-jacobian} that, for
$h\in U^{k+2}_\delta$, the outer unit normal, the surface Jacobian, and
the radial volume weight are given by
\begin{equation}
	\label{eq:normal-jacobian}
	n_h\circ\Psi_h
	=
	\frac{(1+h)\omega-\nablaS h}
	{\sqrt{(1+h)^2+\abs{\nablaS h}^2}},
	\qquad
	J_h=(1+h)\sqrt{(1+h)^2+\abs{\nablaS h}^2},
\end{equation}
and
\begin{equation}
	\label{eq:Mh}
	M_h=\bigl(\omega\cdot(n_h\circ\Psi_h)\bigr)J_h=(1+h)^2.
\end{equation}
We collect some further identities in the following lemma.
\begin{lemma}
	\label{lem:radial-geometry}
	Let $h\in U^{k+2}_\delta$.
	The normal component of the rotational-translational field $W_{\alpha,V}$ is
	\begin{equation}
		\label{eq:W-normal-exact}
		\bigl(W_{\alpha,V}\cdot n_h\bigr)\circ\Psi_h
		=
		\frac{-\alpha(1+h)\partial_\varphi h
			+V\left((1+h)Z-\nablaS h\cdot\nablaS Z\right)}
		{\sqrt{(1+h)^2+\abs{\nablaS h}^2}},
	\end{equation}
	and,
	\begin{equation}
		\label{eq:translation-normal-exact}
		\bigl(Ve_3\cdot n_h\bigr)\circ\Psi_h
		=
		V\frac{(1+h)Z-\nablaS h\cdot\nablaS Z}
		{\sqrt{(1+h)^2+\abs{\nablaS h}^2}} .
	\end{equation}
	For $h=t\eta$, as $t\to0$,
	\begin{equation}
		\label{eq:normal-first-expansion}
		n_{t\eta}\circ\Psi_{t\eta}
		=
		\omega-t\nablaS\eta+O(t^2),
	\end{equation}
	\begin{equation}
		\label{eq:W-normal-first-expansion}
		\bigl(W_{\alpha,V}\cdot n_{t\eta}\bigr)\circ\Psi_{t\eta}
		=
		VZ-t\alpha\partial_\varphi\eta
		-tV\nablaS\eta\cdot\nablaS Z+O(t^2),
	\end{equation}
	and the purely translational normal velocity has the second-order expansion
	\begin{align}
		\label{eq:translation-normal-second-expansion}
		\bigl(Ve_3\cdot n_{t\eta}\bigr)\circ\Psi_{t\eta}
		 & =
		VZ-tV\nablaS\eta\cdot\nablaS Z
		+t^2V
		\left(
		\eta\,\nablaS\eta\cdot\nablaS Z
		-\frac Z2\abs{\nablaS\eta}^2
		\right)                          \\
		 & \phantom{=}+O(t^3). \nonumber
	\end{align}
	The mean curvature, with the convention $H=2$ on the round sphere, satisfies
	\begin{equation}
		\label{eq:curvature-expansion-one-parameter}
		H_{t\eta}\circ\Psi_{t\eta}
		=
		2
		-t(\Delta_{\Sph}+2)\eta
		+O(t^2).
	\end{equation}
\end{lemma}

\begin{proof}
	Since
	$e_3=Z\omega+\nablaS Z$ and
	$(e_3\times\omega)\cdot\nablaS h=\partial_\varphi h$, the normal
	components \eqref{eq:W-normal-exact} and
	\eqref{eq:translation-normal-exact} follow by taking the scalar product of
	$W_{\alpha,V}(\Psi_h)=\alpha(1+h)e_3\times\omega+Ve_3$ with
	\eqref{eq:normal-jacobian}. Expanding these identities at
	$h=t\eta$ gives \eqref{eq:normal-first-expansion},
	\eqref{eq:W-normal-first-expansion}, and
	\eqref{eq:translation-normal-second-expansion}.

	For the curvature, we recall the radial-graph formula
	\begin{align}
		 & H_h\circ\Psi_h \label{eq:curvature-radial-graph} \\
		 & =
		\frac{
			2(1+h)^3
			+
			3(1+h)\abs{\nablaS h}^2
			-
			\big((1+h)^2+\abs{\nablaS h}^2\big)\Delta_{\Sph} h
			+
			\nablaS^2 h(\nablaS h,\nablaS h)
		}{
			(1+h)\big((1+h)^2+\abs{\nablaS h}^2\big)^{3/2}
		} \nonumber
	\end{align}
	found, with the opposite
	curvature sign convention, in
	\cite[pp.~60--61, equation~(2.49)]{PruessSimonett2016}.
	Taylor expansion gives \eqref{eq:curvature-expansion-one-parameter}.
\end{proof}

The exact curvature formula above, together with the Sobolev-scale Neumann
regularity of Lemma~\ref{lem:neumann-sobolev-scale}, yields the following
regularity bootstrap.
\begin{lemma}
	\label{lem:small-sobolev-solution-spatial-regularity}
	Let $\delta$ be as in Lemma~\ref{lem:neumann-sobolev-scale}. Suppose that
	$h\in U_\delta^{k+2}$ satisfies
	\[
		\tau F_{\alpha,V}(\Psi_h,\Qin)
		-
		F_{\alpha,V}(\Psi_h,\Qout)
		+
		H_h\circ\Psi_h
		=c
		\quad\text{on }\Sph.
	\]
	Then $h\in\C^\infty(\Sph)$. The associated interior and exterior
	potentials are smooth up to the free boundary, and the exterior potential
	has the decay stated in Proposition~\ref{prop:neumann-sobolev}.
\end{lemma}

\begin{proof}
	We prove, for every integer $r\ge k$, that
	\[
		h\in\H^{r+2}(\Sph)
		\quad\Longrightarrow\quad
		h\in\H^{r+3}(\Sph).
	\]
	Lemma~\ref{lem:neumann-sobolev-scale} gives
	$\Qin,\Qout\in\H^{r+1}(\Sph;\R^3)$. Since
	$\H^{r+1}(\Sph)$ is a Banach algebra, the non-curvature part of the
	Bernoulli equation belongs to $\H^{r+1}(\Sph)$, and hence
	\[
		H_h\circ\Psi_h\in\H^{r+1}(\Sph).
	\]

	To see the ellipticity, put $\rho:=1+h$, $p:=\nablaS h$, and let
	$g^{ij}$ be the inverse round metric. Formula
	\eqref{eq:curvature-radial-graph} has the form
	\[
		H_h\circ\Psi_h
		=
		b(h,p)-a_h^{ij}\nabla_i\nabla_jh,
	\]
	where $b(h,p)$ contains no second derivative of $h$ and
	\[
		a_h^{ij}
		=
		\frac{
			(\rho^2+\abs{p}^2)g^{ij}-p^ip^j
		}{
			\rho(\rho^2+\abs{p}^2)^{3/2}
		}.
	\]
	For every $\omega\in\Sph$ and every cotangent vector $\zeta\in T_\omega^*\Sph$,
	\[
		a_h^{ij}\zeta_i\zeta_j
		\ge
		\frac{\rho}{(\rho^2+\abs{p}^2)^{3/2}}
		\abs{\zeta}^2.
	\]
	The fixed low-norm bound makes $\rho$ positive and bounds $p$ in
	$\C^0$, so the operator is uniformly elliptic. This uses no smallness of
	$\norm h_{\H^{r+2}}$.

	Under the induction hypothesis,
	$a_h^{ij},b(h,p)\in\H^{r+1}(\Sph)$. In a finite atlas, differentiate the
	last equation by difference quotients. Each first derivative of $h$ then
	satisfies a uniformly elliptic linear equation whose right-hand side is in
	$\H^r$: the commutator terms have the form
	$(\nabla a_h)\nabla^2h$ and belong to $\H^r$ by Sobolev multiplication on
	the two-dimensional sphere. The linear elliptic $\H^{r+2}$ estimate gives
	$\nabla h\in\H^{r+2}$ and therefore $h\in\H^{r+3}$.
	Starting at $r=k$ and iterating proves that $h$ belongs to every Sobolev
	space. Lemma~\ref{lem:neumann-sobolev-scale}, applied at every level,
	then gives smoothness of the potentials up to the boundary.
\end{proof}

For the moving-boundary calculations below, we use a compactly supported fixed-domain identification of the same type as that used in the proof of Proposition~\ref{prop:neumann-sobolev}.
The maps $T_t^\eta$ identify the moving and fixed domains. Choose
$\vartheta\in \C_c^\infty((0,\infty))$ equal to one near $r=1$ and, for small
$t$, write $\Gamma_t:=\Psi_{t\eta}(\Sph)$ with $\Psi$ as in \eqref{eq:radial-graph-chart}, and set
\[
	T_t^\eta(r\omega)
	=
	\bigl(r+t\vartheta(r)\eta(\omega)\bigr)\omega,
	\qquad
	\mathscr X_\eta
	:=
	\left.\partial_tT_t^\eta\right|_{t=0}.
\]
Thus $T_t^\eta$ maps $B_1$ and its exterior onto the corresponding domains
bounded by $\Gamma_t$, is the identity away from a fixed annulus, and
$\mathscr X_\eta(\omega)=\eta(\omega)\omega$ on $\Sph$. Let
$n_t:=n_{t\eta}$ be the bounded-phase outer normal from
\eqref{eq:normal-jacobian}.

\begin{lemma}
	\label{lem:eulerian-boundary-differentiation}
	Let $\eta\in\C^\infty(\Sph)$, and let $\phi_t$
	be defined on $T_t^\eta(D_0)$, where $D_0$ is either $B_1$ or
	$\R^3\setminus\overline{B_1}$. Assume that, for every fixed $R>1$,
	\[
		t\longmapsto\widehat\phi_t:=\phi_t\circ T_t^\eta
	\]
	is twice continuously differentiable with values in
	$\H^{k+5/2}(D_0\cap B_R)$. In the exterior case assume in addition that this
	map is continuously differentiable with values in $\Hdot^1(D_0)$. Define
	\[
		\dot\phi
		:=
		\left.\partial_t\right|_{t=0}\widehat\phi_t,
		\qquad
		\phi'
		:=
		\dot\phi-\mathscr X_\eta\cdot\nabla\phi_0.
	\]
	Then, on $\Sph$,
	\begin{align}
		\phi_t\circ\Psi_{t\eta}
		 & =
		\phi_0
		+t\bigl(\phi'+\eta\partial_r\phi_0\bigr)
		+O(t^2),
		\label{eq:moving-trace-eulerian}    \\
		(\nabla\phi_t)\circ\Psi_{t\eta}
		 & =
		\nabla\phi_0
		+t\bigl(\nabla\phi'+\eta\partial_r\nabla\phi_0\bigr)
		+O(t^2),
		\label{eq:moving-gradient-eulerian} \\
		(\nabla\phi_t\cdot n_t)\circ\Psi_{t\eta}
		 & =
		\partial_r\phi_0
		+t\left(
		  \partial_r\phi'
		  +\eta\partial_{rr}\phi_0
		  -\nablaS\eta\cdot(\nabla\phi_0)_{\tan}
		\right)
		+O(t^2).
		\label{eq:moving-normal-eulerian}
	\end{align}
	If every $\phi_t$ is harmonic, then $\phi'$ is harmonic on $D_0$.
	The material derivative $\dot\phi$ is not harmonic in general.
	Moreover, if $\phi_t$ is given by an exterior Neumann family $\phi^{\rm out}_{t\eta}$ of Proposition~\ref{prop:neumann-sobolev}, then $\phi'$ has finite Dirichlet energy and satisfies
	\[
		\phi'(x)=O(\abs{x}^{-2}),
		\qquad
		\nabla\phi'(x)=O(\abs{x}^{-3}).
	\]
\end{lemma}

\begin{proof}
	Write $\widehat\phi_t:=\phi_t\circ T_t^\eta$. Smooth material dependence
	gives $\widehat\phi_t=\phi_0+t\dot\phi+O(t^2)$. Restriction to $\Sph$ and
	the identity $\dot\phi=\phi'+\mathscr X_\eta\cdot\nabla\phi_0$, as well as $\mathscr X_\eta=\eta\omega$ on $\Sph$ give
	\eqref{eq:moving-trace-eulerian}. For the gradient, the chain rule yields
	\[
		(\nabla\phi_t)\circ T_t^\eta
			=
			(DT_t^\eta)^{-T}\nabla\widehat\phi_t.
	\]
	Differentiating this identity, and using
	\[
		DT_t^\eta=\id+tD\mathscr{X}_\eta=\id+t\left(\omega\otimes\nabla_{\Sph}\eta+\eta\left(\id-\omega\otimes\omega\right)\right)\quad \text{on } \Sph
	\]
	gives
	\[
		\left.\partial_t\right|_{t=0}
		\bigl((\nabla\phi_t)\circ\Psi_{t\eta}\bigr)
		=
		\nabla\phi'+\eta D^2\phi_0\,\omega
		=
		\nabla\phi'+\eta\partial_r\nabla\phi_0,
	\]
	which is \eqref{eq:moving-gradient-eulerian}. Taking its scalar product with
	$n_t\circ\Psi_{t\eta}=\omega-t\nablaS\eta+O(t^2)$ proves
	\eqref{eq:moving-normal-eulerian}.

	If $\Delta\phi_t=0$, differentiation at fixed spatial points on
	each compact subset of $D_0$ gives $\Delta\phi'=0$. In the exterior
	Neumann case, $T_t^\eta$ is the identity beyond the fixed collar, while
	the proof of Proposition~\ref{prop:neumann-sobolev} represents $\phi_t$
	there by a smoothly parameter-dependent exterior Poisson series whose
	$\ell=0$ coefficient vanishes identically. Differentiating that series
	therefore preserves both finite energy and the stated dipole decay.
\end{proof}
For a Neumann solution family from
Proposition~\ref{prop:neumann-sobolev}, the first variations in the three
boundary   identities \eqref{eq:moving-trace-eulerian}--\eqref{eq:moving-normal-eulerian} depend continuously on the direction
$\eta\in\H^{k+2}(\Sph)$, with values in $\H^k(\Sph)$. This follows from
the smooth fixed-domain solution maps proved there, the trace theorem, and
the multiplication estimates on $\Sph$. Since $\C^\infty(\Sph)$ is dense
in $\H^{k+2}(\Sph)$, identities
\eqref{eq:moving-trace-eulerian}--\eqref{eq:moving-normal-eulerian}
therefore extend to every $\eta\in\H^{k+2}(\Sph)$ in the sense of first
variations.

\subsection{The elliptic Neumann problems}

The formulas in Lemma~\ref{lem:neumann-to-dirichlet} define the
\emph{round-sphere Neumann-to-Dirichlet maps} $g\mapsto\ein_g\big|_{r=1}$ and
$g\mapsto\eout_g\big|_{r=1}$
and the \emph{round-sphere Neumann-to-gradient maps}
$g\mapsto\nabla\ein_g\big|_{r=1}$ and $g\mapsto\nabla\eout_g\big|_{r=1}$.
We denote by $\partial_r=\omega\cdot\nabla$ differentiation in the normal
direction of the bounded phase.

\begin{lemma}
	\label{lem:neumann-to-dirichlet}
	Let $g\in\H^j(\Sph)$ have zero average and write its spherical-harmonic
	decomposition as
	\[
		g=\sum_{\ell\ge1}g_\ell,
	\]
	where $g_\ell$ is the degree-$\ell$ part.
	Let $\ein_g$ denote the unique solution of the interior Neumann problem
	\begin{equation}
		\label{eq:NtD-interior-problem}
		\begin{cases}
			\Delta \ein_g=0    & \text{in } B_1, \\
			\partial_r\ein_g=g & \text{on } r=1, \\
			\displaystyle\int_{\Sph}\ein_g(1,\omega)\dd S=0.
		\end{cases}
	\end{equation}
	and $\eout_g$ the unique decaying finite-energy solution of the exterior
	Neumann problem
	\begin{equation}
		\label{eq:NtD-exterior-problem}
		\begin{cases}
			\Delta \eout_g=0     &
			\text{in } \R^3\setminus\overline{B_1},            \\
			\partial_r\eout_g=g  & \text{on } r=1,             \\
			\nabla\eout_g\in\L^2(\R^3\setminus\overline{B_1}),
			\quad \eout_g(y)\to0 & \text{as } \abs y\to\infty.
		\end{cases}
	\end{equation}
	The two solutions have the decompositions
	\begin{equation}
		\label{eq:NtD-potentials}
		\ein_g(r,\omega)
		=
		\sum_{\ell\ge1}\frac{1}{\ell}r^\ell g_\ell(\omega),
		\qquad
		\eout_g(r,\omega)
		=
		-\sum_{\ell\ge1}\frac{1}{\ell+1}r^{-\ell-1}g_\ell(\omega).
	\end{equation}
	Thus
	\begin{equation}
		\label{eq:NtD-traces}
		\ein_g\big|_{r=1}
		=
		\sum_{\ell\ge1}\frac1\ell g_\ell,
		\qquad
		\eout_g\big|_{r=1}
		=
		-\sum_{\ell\ge1}\frac1{\ell+1}g_\ell.
	\end{equation}
	The corresponding boundary gradients are
	\begin{align}
		\qin_g
		 & :=
		\nabla\ein_g\big|_{r=1}
		=
		\sum_{\ell\ge1}\left(g_\ell\omega+\frac1\ell\nablaS g_\ell\right), \label{eq:NtG-gradient-in} \\
		\qout_g
		 & :=
		\nabla\eout_g\big|_{r=1}
		=
		\sum_{\ell\ge1}\left(g_\ell\omega-\frac1{\ell+1}\nablaS g_\ell\right). \label{eq:NtG-gradient-out}
	\end{align}
	In particular, the tangential parts are
	\begin{equation}
		\label{eq:NtG-tangential}
		(\qin_g)_{\tan}
		=
		\sum_{\ell\ge1}\frac1\ell\nablaS g_\ell,
		\qquad
		(\qout_g)_{\tan}
		=
		-\sum_{\ell\ge1}\frac1{\ell+1}\nablaS g_\ell .
	\end{equation}
	For a pure degree-$\ell$ datum $g_\ell$, these formulas reduce to
	\begin{equation}
		\label{eq:NtG-pure-mode}
		\qin_{g_\ell}=g_\ell\omega+\frac1\ell\nablaS g_\ell,
		\qquad
		\qout_{g_\ell}=g_\ell\omega-\frac1{\ell+1}\nablaS g_\ell .
	\end{equation}
\end{lemma}

\begin{proof}
	For a pure degree-$\ell$ datum, the regular interior harmonic and the decaying
	exterior harmonic are multiples of $r^\ell g_\ell(\omega)$ and
	$r^{-\ell-1}g_\ell(\omega)$, respectively. The constants in
	\eqref{eq:NtD-potentials} are chosen so that the radial derivative at
	$r=1$ is $g_\ell$ in the sense of
	\eqref{eq:NtD-interior-problem}--\eqref{eq:NtD-exterior-problem}. The
	zero-average assumption implies the Neumann compatibility condition, and the trace
	normalisation fixes the interior additive constant.
	Summing over $\ell$ gives the general case. The gradient decomposition on
	$r=1$,
	\[
		\nabla\phi=\partial_r\phi\,\omega+\nablaS(\phi\big|_{r=1}),
	\]
	then gives \eqref{eq:NtG-gradient-in}, \eqref{eq:NtG-gradient-out} and
	\eqref{eq:NtG-tangential}.
\end{proof}

Our goal is to obtain expansions of the inner and outer potentials under shape
deformations of the round sphere. In view of the expansion of the Neumann datum
in \eqref{eq:W-normal-first-expansion}, we begin with the zeroth-order datum
$g=VZ$ on the round sphere. The resulting solutions are called the
\emph{round-sphere translational solutions}, their
potentials are denoted by $\ein_V$ and $\eout_V$, and the boundary gradients by
$\qin_V$ and $\qout_V$.

\begin{lemma}
	\label{lem:round-translational-potentials}
	On the round sphere we consider the translational datum $VZ$. The associated
	potentials, obtained from \eqref{eq:NtD-interior-problem} and
	\eqref{eq:NtD-exterior-problem}, are
	\begin{equation}
		\label{eq:round-translational-potentials}
		\ein_V(y)=Vy_3=VrZ,
		\qquad
		\eout_V(r,\omega) = -\frac{V}2\frac{y_3}{|y|^3}
		=-\frac{V}{2}r^{-2}Z .
	\end{equation}
	Their boundary gradients on $r=1$ are
	\begin{equation}
		\label{eq:round-translational-gradients}
		\qin_V=Ve_3,
		\qquad
		\qout_V
		=V\left(Z\omega-\frac12\nablaS Z\right),
		\qquad
		\qout_V-Ve_3=-\frac{3V}{2}\nablaS Z .
	\end{equation}
	Moreover,
	\begin{equation}
		\label{eq:round-translational-F}
		F_{\alpha,V}(\omega,\qin_V)=0,
		\qquad
		F_{\alpha,V}(\omega,\qout_V)
		=\frac{9V^2}{8}(1-Z^2).
	\end{equation}
\end{lemma}

\begin{proof}
	The formulas for $\ein_V$ and $\eout_V$ are the
	degree-one case of Lemma~\ref{lem:neumann-to-dirichlet}. Their gradients
	give \eqref{eq:round-translational-gradients}. Since
	$\qin_V=Ve_3$, the interior value of $F_{\alpha,V}$ is zero. For the
	exterior field, the rotational contribution vanishes because
	$\qout_V$ is axisymmetric:
	\[
		(e_3\times\omega)\cdot \qout_V
		=-\frac V2\partial_\varphi Z=0.
	\]
	Thus
	\[
		F_{\alpha,V}(\omega,\qout_V)
		=\frac12\abs{\qout_V-Ve_3}^2
		=\frac{9V^2}{8}\abs{\nablaS Z}^2
		=\frac{9V^2}{8}(1-Z^2).
	\]
\end{proof}

Next, we turn to the first-order terms associated with the radial perturbation
\[
	\Gamma_t=\bigl\{(1+t\eta(\omega))\omega:\omega\in\Sph\bigr\}.
\]
We split the first-order correction of the potentials into two pieces. The
\emph{rotational correction} is caused by the fact that the rotational vector field $\alpha e_3\times y$, although tangent to the round sphere, acquires a normal component when the
normal vector is varied. Its first-order Neumann datum is
$-\alpha\,\partial_\varphi \eta$; see \eqref{eq:W-normal-first-expansion}.
The \emph{translational shape correction} is the derivative at $t=0$ of the
translational Neumann solution corresponding to the field $V e_3$ on the
deformed domain. The boundary-gradient corrections are obtained by taking the
gradients of these first-order potentials at $r=1$.

\begin{lemma}
	\label{lem:first-order-neumann-expansion}
	Fix $\alpha,V\in\R$ and let $\eta\in\H^{k+2}(\Sph)$. Set
	\[
		\Gamma_t=\Psi_{t\eta}(\Sph),\qquad
		\Omegain_t=\set{r<1+t\eta(\omega)},\qquad
		\Omegaout_t=\R^3\setminus\overline{\Omegain_t}.
	\]
	Let $n_t$ be the normal in \eqref{eq:normal-jacobian} with $h=t\eta$. The
	moving-boundary Neumann problems are
	\begin{equation}
		\begin{cases}
			\Delta\ein_{t,V}=0 & \text{in } \Omegain_t, \\
			\nabla\ein_{t,V}\cdot n_t=W_{\alpha,V}\cdot n_t
			                   & \text{on } \Gamma_t,   \\
			\displaystyle \int_{\Sph}\ein_{t,V}((1+t\eta)\omega)\dd S=0,
		\end{cases}
	\end{equation}
	and
	\begin{equation}
		\begin{cases}
			\Delta\eout_{t,V}=0 & \text{in } \Omegaout_t,      \\
			\nabla\eout_{t,V}\cdot n_t=W_{\alpha,V}\cdot n_t
			                    & \text{on } \Gamma_t,         \\
			\nabla\eout_{t,V}\in\L^2(\Omegaout_t),
			\quad
			\eout_{t,V}(y)\to0  & \text{as } \abs y\to\infty .
		\end{cases}
	\end{equation}
	Let $\phinoutdot$ and $(\phinout)'$ denote the material and
	Eulerian derivatives associated with $T_t^\eta$ as in
	Lemma~\ref{lem:eulerian-boundary-differentiation}. The Eulerian derivatives
	are harmonic and decompose as
	$(\phinout)'=\phinout_{{\rm rot},\eta}+\uinout_{V,\eta}$. Consequently, up
	to the irrelevant interior additive constant, the fixed-domain pullback
	expansions are
	\begin{equation}
		\label{eq:potential-expansion-in-out}
		\begin{aligned}
			\ein_{t,V}\circ T_t^\eta
			 & =
			\ein_V+t\left(
			\mathscr X_\eta\cdot\nabla\ein_V
			+\ein_{{\rm rot},\eta}+\uin_{V,\eta}
			\right)+O(t^2), \\
			\eout_{t,V}\circ T_t^\eta
			 & =
			\eout_V+t\left(
			\mathscr X_\eta\cdot\nabla\eout_V
			+\eout_{{\rm rot},\eta}+\uout_{V,\eta}
			\right)+O(t^2).
		\end{aligned}
	\end{equation}
	The round-sphere potentials $\phinout_V$ are given in
	Lemma~\ref{lem:round-translational-potentials}. The Eulerian rotational
	potentials solve
	\begin{equation}
		\label{eq:rotational-potentials-def}
		\begin{cases}
			\Delta\phinout_{{\rm rot},\eta}
			 & = 0 \\[0.2em]
			\partial_r\phinout_{{\rm rot},\eta}\big|_{r=1}
			 & =
			g_{{\rm rot},\eta}:=-\alpha\partial_\varphi\eta,
		\end{cases}
	\end{equation}
	with exterior decay for the outer problem and the same interior normalisation
	as above. The translational shape part is trivial in the interior, in the
	sense that
	\begin{equation}
		\nabla\uin_{V,\eta}=0,
	\end{equation}
	The moving trace normalisation fixes the resulting constant in
	$\uin_{V,\eta}$. None of the subsequent boundary-gradient formulas depends
	on that constant.
	The exterior part solves
	\begin{equation}
		\label{eq:exterior-translation-shape-datum}
		\begin{cases}
			\Delta\uout_{V,\eta}=0
			 & \text{in } \R^3\setminus\overline{B_1}, \\
			\displaystyle
			\partial_r\uout_{V,\eta}\big|_{r=1}
			=
			g_{V,\eta}^{\rm out}
			:=
			3V\eta Z
			-
			\frac{3V}{2}\nablaS\eta\cdot\nablaS Z,
			 & \text{on } r=1,                         \\
			\uout_{V,\eta}(y)\to0
			 & \text{as } \abs y\to\infty .
		\end{cases}
	\end{equation}
	If $g_{{\rm rot},\eta}=\sum_{\ell\ge1}g_\ell$ and
	$g_{V,\eta}^{\rm out}=\sum_{n\ge1}k_n$ are their spherical-harmonic
	decompositions, then on $r=1$
	\begin{align}
		\qin_{{\rm rot},\eta} \label{eq:rotational-boundary-gradient-in}
		 & :=\nabla\ein_{{\rm rot},\eta}\big|_{r=1}
		=
		\sum_{\ell\ge1}\left(g_\ell\omega+\frac1\ell\nablaS g_\ell\right), \\
		\qout_{{\rm rot},\eta} \label{eq:rotational-boundary-gradient-out}
		 & :=\nabla\eout_{{\rm rot},\eta}\big|_{r=1}
		=
		\sum_{\ell\ge1}\left(g_\ell\omega-\frac1{\ell+1}\nablaS g_\ell\right),
	\end{align}
	and
	\begin{equation}
		\label{eq:translational-shape-boundary-gradient}
		\nabla\uout_{V,\eta}\big|_{r=1}
		=
		\sum_{n\ge1}\left(k_n\omega-\frac1{n+1}\nablaS k_n\right).
	\end{equation}
	Moreover, for the pulled-back boundary gradients
	\[
		\qin_{t,V}(\omega)
		:=
		\nabla\ein_{t,V}((1+t\eta(\omega))\omega),
		\qquad
		\qout_{t,V}(\omega)
		:=
		\nabla\eout_{t,V}((1+t\eta(\omega))\omega),
	\]
	we have
	\begin{equation}
		\label{eq:boundary-gradient-expansion-inner}
		\qin_{t,V}
		=
		Ve_3+t\qin_{{\rm rot},\eta}+O(t^2),
	\end{equation}
	and
	\begin{align}
		\label{eq:boundary-gradient-expansion-outer}
		\qout_{t,V}
		 & =
		V\left(Z\omega-\frac12\nablaS Z\right)
		+t\qout_{{\rm rot},\eta}
		+t\left[
			  V\eta\left(-3Z\omega+\frac32\nablaS Z\right)
			  +\nabla\uout_{V,\eta}\big|_{r=1}
			  \right] \\
		 & \hphantom{=}+O(t^2). \nonumber
	\end{align}
	Therefore the tangential part of the exterior boundary gradient, with respect
	to the tangent space $T_\omega\Sph$, is
	\begin{equation}
		\label{eq:boundary-gradient-expansion-outer-tangential}
		(\qout_{t,V})_{\tan}
		=
		-\frac V2\nablaS Z
		+t(\qout_{{\rm rot},\eta})_{\tan}
		+t\left[
			\frac{3V\eta}{2}\nablaS Z
			+\left(\nabla\uout_{V,\eta}\big|_{r=1}\right)_{\tan}
			\right]
		+O(t^2),
	\end{equation}
	where the two tangential pieces are obtained from
	\eqref{eq:rotational-boundary-gradient-out} and
	\eqref{eq:translational-shape-boundary-gradient} by dropping their radial
	parts.
\end{lemma}

\begin{proof}
	The pullbacks of the moving Neumann solutions by $T_t^\eta$ depend
	smoothly on $t$. Let $\chi^{\rm in}$ and $\chi^{\rm out}$ denote their
	Eulerian derivatives, as defined in
	Lemma~\ref{lem:eulerian-boundary-differentiation}, i.e. $\chi^{\rm in/out}=\left(\phi_{t,V}^{\rm in /out}\right)'$.
	By said Lemma $\chi^{\rm in}$ is harmonic in
	$B_1$ and $\chi^{\rm out}$ is harmonic in
	$\R^3\setminus\overline{B_1}$ with
	$\chi^{\rm out}=O(r^{-2})$ and $\nabla\chi^{\rm out}=O(r^{-3})$. It remains
	to identify their Neumann data on $r=1$.

	Next, we linearise the boundary condition. To treat both phases simultaneously,
	let $(\phi_{t,V},\phi_V,\chi)$ denote either
	$(\phiin_{t,V},\phiin_V,\chi^{\rm in})$ or
	$(\phiout_{t,V},\phiout_V,\chi^{\rm out})$. Thus
	\[
		\phi_V=\phi_{0,V},
		\qquad
		\chi
		=
		\left.\partial_t\right|_{t=0}
		(\phi_{t,V}\circ T_t^\eta)
		-\mathscr X_\eta\cdot\nabla\phi_V,
	\]
	as in
	Lemma~\ref{lem:eulerian-boundary-differentiation}.
	The third boundary identity of
	Lemma~\ref{lem:eulerian-boundary-differentiation} gives
	\begin{align}
		\label{eq:linearised-neumann-condition-template}
		 & (\nabla\phi_{t,V}\cdot n_t)\circ\Psi_{t\eta} \\
		 & =
		\partial_r\phi_V
		+t\left(
		  \partial_r\chi
		  +\eta\partial_{rr}\phi_V
		  -\nablaS\eta\cdot(\nabla\phi_V)_{\tan}
		\right)_{r=1}
		+O(t^2). \nonumber
	\end{align}
	The three order-$t$ terms are, respectively, the Neumann trace of the first
	correction, the displacement of the evaluation point, and the first variation
	of the normal.

	On the other hand, \eqref{eq:W-normal-first-expansion} gives the expansion
	of the prescribed datum
	\begin{equation}
		\label{eq:linearised-moving-datum}
		(W_{\alpha,V}\cdot n_t)\circ\Psi_{t\eta}
		=
		VZ+t\left(
		-\alpha\partial_\varphi\eta
		-V\nablaS\eta\cdot\nablaS Z
		\right)+O(t^2).
	\end{equation}
	Comparing \eqref{eq:linearised-neumann-condition-template} with
	\eqref{eq:linearised-moving-datum}, we find that the
	zeroth-order potentials are the translational potentials
	$\ein_V$ and $\eout_V$.

	Turning to the first-order terms and recalling
	\eqref{eq:round-translational-potentials}, we observe that the Neumann datum
	for $\chi$ is
	\begin{align}\label{eq:first-order-neumann-datum}
		\partial_r\chi=
		-\eta\partial_{rr}\phi_V
		+\nablaS\eta\cdot(\nabla\phi_V)_{\tan}-\alpha\partial_\varphi\eta
		-V\nablaS\eta\cdot\nablaS Z,
	\end{align}
	which is linear in $(\alpha,V)$. We may therefore consider separately the
	purely rotational part, obtained by setting $V=0$, and the purely
	translational part, obtained by setting $\alpha=0$.

	Let $V=0$. By Lemma~\ref{lem:round-translational-potentials},
	$\nabla\phi_{V=0}=0$, so only
	\begin{align*}
		\partial_r\chi=-\alpha\partial_\varphi\eta
	\end{align*}
	remains. Denoting the corresponding terms by
	$\phinout_{{\rm rot},\eta}$ therefore gives harmonic first variations with
	Neumann datum
	\[
		\partial_r\ein_{{\rm rot},\eta}\big|_{r=1}
		=
		\partial_r\eout_{{\rm rot},\eta}\big|_{r=1}
		=
		g_{{\rm rot},\eta}
		:=
		-\alpha\partial_\varphi\eta .
	\]
	The datum has zero average on $\Sph$, so the compatibility conditions
	for the interior and exterior Neumann
	problems are satisfied. This proves
	\eqref{eq:rotational-potentials-def}.

	Now set $\alpha=0$ to identify the translational shape part, and denote the
	corresponding variation by $u^{\rm in /\rm out}_{V,\eta}$. For the interior
	base solution $\ein_V=Vy_3$, one has, on $r=1$,
	\[
		\partial_r\ein_V=VZ,
		\qquad
		\partial_{rr}\ein_V=0,
		\qquad
		(\nabla\ein_V)_{\tan}=V\nablaS Z .
	\]
	Substitution into
	\eqref{eq:first-order-neumann-datum}
	yields
	\[
		\partial_r\uin_{V,\eta}\big|_{r=1}
		=
		-V\nablaS\eta\cdot\nablaS Z
		+V\nablaS\eta\cdot\nablaS Z
		=
		0.
	\]
	Thus $\uin_{V,\eta}$ is harmonic with homogeneous Neumann datum, i.e.\ $\nabla\uin_{V,\eta}=0$.

	For the exterior base solution
	\[
		\eout_V(r,\omega)
		=-\frac{V}{2}r^{-2}Z,
	\]
	we compute, on $r=1$,
	\begin{equation}
		\label{eq:exterior-base-radial-data-proof}
		\partial_r\eout_V=VZ,
		\qquad
		\partial_{rr}\eout_V=-3VZ,
		\qquad
		(\nabla\eout_V)_{\tan}=-\frac V2\nablaS Z .
	\end{equation}
	Substitution in \eqref{eq:first-order-neumann-datum} gives
	\[
		\partial_r\uout_{V,\eta}\big|_{r=1}
		=
		3V\eta Z
		-
		\frac{3V}{2}\nablaS\eta\cdot\nablaS Z,
	\]
	which is the boundary condition in
	\eqref{eq:exterior-translation-shape-datum}. The compatibility condition
	for this exterior Neumann problem follows from
	\[
		\int_{\Sph}\nablaS\eta\cdot\nablaS Z\dd S
		=
		-\int_{\Sph}\eta\Delta_{\Sph}Z\dd S
		=
		2\int_{\Sph}\eta Z\dd S,
	\]
	because $\Delta_{\Sph}Z=-2Z$.

	By linearity, the first variation is the sum of the rotational part and the
	translational shape part, which gives
	\eqref{eq:potential-expansion-in-out}. Applying the
	Neumann-to-gradient formulas of Lemma~\ref{lem:neumann-to-dirichlet} to
	$g_{{\rm rot},\eta}$ and $g_{V,\eta}^{\rm out}$ gives
	\eqref{eq:rotational-boundary-gradient-in}, \eqref{eq:rotational-boundary-gradient-out} and
	\eqref{eq:translational-shape-boundary-gradient}.

	It remains to expand the gradient at the moving boundary. For either
	moving potential family, write $\Phi_t=\ein_{t,V}$ or
	$\Phi_t=\eout_{t,V}$, let $\Phi_0$ denote the corresponding base field,
	and let $\chi:=\Phi'$ be its Eulerian derivative in the sense of
	Lemma~\ref{lem:eulerian-boundary-differentiation}. Then
	\eqref{eq:moving-gradient-eulerian} gives
	\begin{equation}
		\label{eq:linearised-boundary-gradient-template}
		(\nabla\Phi_t)\circ\Psi_{t\eta}
		=
		\nabla\Phi_0\big|_{r=1}
		+t\left(
		\nabla\chi\big|_{r=1}
		+\eta\partial_r\nabla\Phi_0\big|_{r=1}
		\right)+O(t^2).
	\end{equation}
	For the interior base field $\Phi_0=\phiin_V=Vy_3$,
	$\partial_r\nabla(Vy_3)=0$, and hence
	\eqref{eq:boundary-gradient-expansion-inner} follows. For the exterior
	base field $\Phi_0=\phiout_V=-\frac{V}{2}r^{-2}Z$, differentiating
	\[
		\nabla\eout_V
		=
		\partial_r\eout_V\,\omega
		+\frac1r\nablaS\eout_V
	\]
	and using \eqref{eq:exterior-base-radial-data-proof} gives
	\[
		\partial_r\nabla\eout_V\big|_{r=1}
		=
		V\left(-3Z\omega+\frac32\nablaS Z\right).
	\]
	Inserting this into
	\eqref{eq:linearised-boundary-gradient-template} gives
	\eqref{eq:boundary-gradient-expansion-outer}. Projection onto
	$T_\omega\Sph$ gives
	\eqref{eq:boundary-gradient-expansion-outer-tangential}.
\end{proof}

\subsection{Asymptotics for the $(2,1)$ kernel direction}
A key role is played by the deformation in the direction $XZ$ arising in the
kernel of the unperturbed equation.
\begin{lemma}
	Let $\eta=sXZ$, where $s\in\R$ is sufficiently small. Then the exterior
	translational shape datum from
	\eqref{eq:exterior-translation-shape-datum} is
	\begin{equation}
		\label{eq:gV-XZ-splitting}
		g_{V,\eta}^{\rm out}
		=
		Vs\left[
			-\frac3{10}X
			+6X\left(Z^2-\frac15\right)
			\right].
	\end{equation}
	Consequently,
	\begin{equation}
		\label{eq:uV-XZ-rotational-contribution}
		-\alpha(e_3\times\omega)\cdot\nabla\uout_{V,\eta}\big|_{r=1}
		=
		\alpha Vs
		\left[
			\frac3{20}Y
			-\frac32Y\left(Z^2-\frac15\right)
			\right].
	\end{equation}
\end{lemma}

\begin{proof}
	Using \eqref{eq:XZ-YZ-products},
	\[
		\eta Z=sXZ^2,
		\qquad
		\nablaS\eta\cdot\nablaS Z=sX(1-2Z^2).
	\]
	Substitution into \eqref{eq:exterior-translation-shape-datum} gives
	\[
		g_{V,\eta}^{\rm out}
		=6VsXZ^2-\frac32VsX
		=Vs\left[-\frac3{10}X+6X\left(Z^2-\frac15\right)\right],
	\]
	which is \eqref{eq:gV-XZ-splitting}. The first term is a degree-one
	spherical harmonic and the second is a degree-three spherical harmonic. By
	the exterior Neumann-to-gradient formula
	\eqref{eq:NtG-pure-mode}, only the tangential part contributes to
	$(e_3\times\omega)\cdot\nabla\uout_{V,\eta}$. Hence
	\[
		(e_3\times\omega)\cdot\nabla\uout_{V,\eta}
		=
		-\frac12\partial_\varphi\left(-\frac3{10}VsX\right)
		-\frac14\partial_\varphi\left(6VsX\left(Z^2-\frac15\right)\right),
	\]
	and \eqref{eq:dphi-def} as well as multiplication by $-\alpha$
	gives \eqref{eq:uV-XZ-rotational-contribution}.
\end{proof}

We use the preceding lemma to extract the $sV$ coefficient of the
Bernoulli equation, with the notation introduced at the beginning of
Section~\ref{sec:shape-calculus}. In particular, $s$ and $V$ are independent
variables in the Taylor expansion, and $[\,\cdot\,]_{sV}$ denotes the mixed
second-order coefficient.

\begin{lemma}
	\label{lem:mixed-sV-residual-XZ}
	Let $h=sXZ$. The
	exterior contribution to the Bernoulli quantity is
	\begin{equation}
		\label{eq:Fout-sV-residual}
		\left[
			F_{\alpha,V}(\Psi_h,\qout_{h,V})
			\right]_{sV}
		=
		\alpha
		\left[
			\frac9{20}Y
			-\frac52Y\left(Z^2-\frac15\right)
			\right].
	\end{equation}
	The interior contribution is zero.
	The Bernoulli jump contributes
	\begin{align}
		\label{eq:B-sV-residual}
		 & \left[
			   \tau F_{\alpha,V}(\Psi_h,\qin_{h,V})
			   -
			   F_{\alpha,V}(\Psi_h,\qout_{h,V})
			   +H_h \circ \Psi_h
			   \right]_{sV} \\
		 & =
		-\alpha
		\left[
			\frac9{20}Y
			-\frac52Y\left(Z^2-\frac15\right)
			\right]. \nonumber
	\end{align}
\end{lemma}

\begin{proof}
	The curvature term has no factor $V$ and therefore gives no $sV$
	contribution. In the interior, \eqref{eq:boundary-gradient-expansion-inner}
	gives
	\begin{align*}
		q^{\rm in}_{s,V}=Ve_3+sq^{\rm in}_{{\rm rot},XZ}+O(s^2)
	\end{align*}
	with $q^{\rm in}_{{\rm rot},XZ}$ independent of $V$. Hence
	$[F_{\alpha,V}(\Psi_h,\qin_{h,V})]_{sV}=0$.

	For the exterior problem, \eqref{eq:round-translational-gradients} and
	\eqref{eq:boundary-gradient-expansion-outer} give
	\begin{align*}
		\qout_{s,V} & =\qout_V+s\qout_{{\rm rot},XZ}+s\left[VXZ\left(-3Z\omega+\frac{3}{2}\nablaS Z\right)+\nabla \uout_{V,XZ}\big|_{r=1}\right].
	\end{align*}
	By \eqref{eq:rotational-potentials-def}, the order-$s$ rotational datum is
	\[
		g_{{\rm rot},XZ}=-\alpha\partial_\varphi(XZ)=\alpha YZ .
	\]
	It is a degree-two datum, and hence \eqref{eq:NtG-pure-mode} implies
	\[
		\qout_{{\rm rot},XZ}
		=
		g_{{\rm rot},XZ}\omega-\frac13\nablaS g_{{\rm rot},XZ}.
	\]
	Using \eqref{eq:round-translational-gradients} and \eqref{eq:XZ-YZ-products} the kinetic cross term is
	\begin{align*}
		(\qout_V-Ve_3)\cdot (s\qout_{{\rm rot},XZ})
		 & =
		\left(-\frac{3V}{2}\nablaS Z\right)
		\cdot
		\left(-\frac13\nablaS(\alpha sYZ)\right) \\
		 & =
		\frac{\alpha Vs}{2}Y(1-2Z^2).
	\end{align*}
	Since $\nabla \uout_{V,XZ}$ is linear in $V$ by
	\eqref{eq:exterior-translation-shape-datum}, this is the only contribution
	to the $sV$ coefficient of the kinetic term
	$\frac{1}{2}\abs{\qout_{s,V}-Ve_3}^2$.

	For the rotational part
	$\alpha(1+sXZ)\omega\cdot(e_3\times\qout_{s,V})$, the displacement term
	is zero because
	$\qout_V$ is axisymmetric:
	\[
		\alpha h\,\omega\cdot(e_3\times \qout_V)=0.
	\]
	Similarly, the only remaining rotational term is the translational shape
	derivative of the
	exterior potential. By
	\eqref{eq:uV-XZ-rotational-contribution}, it equals
	\[
		\alpha Vs
		\left[
			\frac3{20}Y
			-\frac32Y\left(Z^2-\frac15\right)
			\right].
	\]
	Since
	\[
		\frac12Y(1-2Z^2)
		=\frac3{10}Y-Y\left(Z^2-\frac15\right),
	\]
	adding the two exterior terms gives \eqref{eq:Fout-sV-residual}. The
	Bernoulli jump is interior minus exterior, which gives
	\eqref{eq:B-sV-residual}.
\end{proof}

\subsection{Asymptotics for the endpoint case \texorpdfstring{$\tau=0$}{tau = 0}}

\begin{lemma}
	Let $\ell\ge1$ and let $\eta=Y_\ell^m$ be a complex
	$\L^2(\Sph)$-normalised spherical harmonic. The real shape derivatives for
	$\Re Y_\ell^m$ and $\Im Y_\ell^m$ are extended complex-linearly.
	For the exterior translational shape derivative in
	\eqref{eq:exterior-translation-shape-datum},
	\begin{equation}
		\label{eq:general-mode-translation-datum}
		\partial_r\uout_{V,Y_\ell^m}\big|_{r=1}
		= \frac{3V}{2} \left(
		(\ell+2)a_{\ell,+}^mY_{\ell+1}^m
		-
		(\ell-1)a_{\ell,-}^mY_{\ell-1}^m \right).
	\end{equation}
	Its tangential boundary gradient satisfies
	\begin{equation}
		\label{eq:general-mode-translation-tangential-gradient}
		\left(\nabla\uout_{V,Y_\ell^m}\right)_{\tan}\Big|_{r=1}
		=\frac{3V}{2} \left(
		-a_{\ell,+}^m\nablaS Y_{\ell+1}^m
		+
		\frac{\ell-1}{\ell}a_{\ell,-}^m\nablaS Y_{\ell-1}^m
		\right).
	\end{equation}
	In both formulas the degree-$(\ell-1)$ term is omitted whenever
	$a_{\ell,-}^m=0$.
\end{lemma}

\begin{proof}
	Insert $\eta=Y_\ell^m$ in
	\eqref{eq:exterior-translation-shape-datum} and use
	\eqref{eq:Z-recursion} and \eqref{eq:gradZ-recursion}. This gives the
	Neumann datum \eqref{eq:general-mode-translation-datum}. The two terms
	have degrees $\ell+1$ and $\ell-1$, so the exterior tangential
	Neumann-to-gradient formula \eqref{eq:NtG-tangential} gives
	\eqref{eq:general-mode-translation-tangential-gradient}.
\end{proof}

\begin{lemma}
	\label{lem:rotational-shape-axisymmetric-background}
	Let $V=0$, let $f\in \H^{k+2}(\Sph)$ be axisymmetric, i.e.\ $f=f(Z)$,
	let $\eta=Y_\ell^m$ with $m\ne0$, and set
	\begin{equation}
		g=-\alpha\partial_\varphi\eta .
	\end{equation}

	For small $\eps$ we define the axisymmetric background surface
	\[
		\Gamma_\eps
		=
		\set{(1+\eps f(\omega))\omega:\omega\in\Sph},
	\]
	and its exterior domain
	\[
		\Omega_\eps^{\rm out}
		=
		\R^3\setminus
		\overline{\set{r<1+\eps f(\omega)}} .
	\]
	We consider the two-parameter surface
	\[
		\Gamma_{\eps,t}
		=
		\set{(1+\eps f(\omega)+t\eta(\omega))\omega:\omega\in\Sph}
	\]
	with small parameter $t$ and define the $\eps$-dependent rotational Neumann datum by
	\begin{equation}
		g_\eps(\omega)
		:=
		\left.
		\frac{\ddd}{\ddd t}
		\right|_{t=0}
		\left[
			\bigl(W_{\alpha,0}\cdot n_{\eps f+t\eta}\bigr)
			\circ\Psi_{\eps f+t\eta}
			\right](\omega).
	\end{equation}
	Let $\eout_\eps$ be the exterior rotational potential
	generated by the perturbation direction $\eta$ on the background
	$r=1+\eps f$, i.e.\ the unique decaying exterior harmonic function
	solving
	\begin{equation}
		\begin{cases}
			\Delta \eout_\eps=0
			 & \text{in } \Omega_\eps^{\rm out}, \\[0.3em]
			\displaystyle
			\bigl(\nabla\eout_\eps\cdot n_{\eps f}\bigr)
			\circ\Psi_{\eps f}
			=
			g_\eps
			 & \text{on } \Sph,                  \\[0.6em]
			\eout_\eps(y)\to0
			 & \text{as } \abs y\to\infty .
		\end{cases}
	\end{equation}
	Let $T_\eps^f$ be the fixed diffeomorphism introduced before
	Lemma~\ref{lem:eulerian-boundary-differentiation}, and define
	\[
		\eoutdot
		:=
		\left.\partial_\eps\right|_{\eps=0}
		(\eout_\eps\circ T_\eps^f),
		\qquad
		v^{\rm out}
		:=
		\eoutdot-\mathscr X_f\cdot\nabla\eout_0.
	\]
	In the fixed-domain Sobolev spaces the pullback expansion is
	\begin{equation}
		\label{eq:rotational-eout-eps-expansion}
		\eout_\eps\circ T_\eps^f
		=
		\eout_0+\eps\left(
		v^{\rm out}+\mathscr X_f\cdot\nabla\eout_0
		\right)+O(\eps^2),
	\end{equation}
	where $\eout_0$ is the round-sphere exterior potential with Neumann datum
	$g$,
	\begin{equation}
		\eout_0(r,\omega)
		=
		-\frac1{\ell+1}
		r^{-\ell-1}
		g(\omega).
	\end{equation}
	The first correction $v^{\rm out}$ solves
	\begin{equation}
		\label{eq:rotational-shape-axisymmetric-outer-problem}
		\begin{cases}
			\Delta v^{\rm out}=0
			 & \text{in } \R^3\setminus\overline{B_1}, \\[0.4em]
			\displaystyle
			\partial_r v^{\rm out}\big|_{r=1}
			=
			(\ell+2)fg
			-
			\frac1{\ell+1}\nablaS f\cdot\nablaS g,
			 & \text{on } r=1,                         \\[0.8em]
			v^{\rm out}(y)\to0
			 & \text{as } \abs y\to\infty .
		\end{cases}
	\end{equation}
	Moreover, if
	\[
		\partial_r v^{\rm out}\big|_{r=1}
		=
		\sum_n k_n^{\rm out}
	\]
	is the spherical-harmonic decomposition of this Neumann datum, then the
	order-$\eps$ correction to the trace on the perturbed boundary is
	\begin{equation}
		\label{eq:rotational-shape-axisymmetric-trace-outer}
		\left[\eout_\eps\big|_{r=1+\eps f}\right]_\eps = \left.
		\frac{\ddd}{\ddd\eps}
		\right|_{\eps=0}
		\eout_\eps\bigl((1+\eps f(\omega))\omega\bigr)
		=
		fg
		-
		\sum_n\frac1{n+1}k_n^{\rm out}.
	\end{equation}
\end{lemma}

\begin{proof}
	The smooth fixed-domain inverse constructed in the proof of
	Proposition~\ref{prop:neumann-sobolev}, applied to $h=\eps f$ and the
	smoothly varying zero-mean datum $g_\eps J_{\eps f}$ proves that the pullbacks
	$\eout_\eps\circ T_\eps^f$ are $\C^2$ in the required Sobolev spaces.
	Hence the following differentiations are justified and, by density, it
	suffices to carry out the calculation for smooth $f$.

	We now compute the datum $g_\eps$. By the normal component
	formula \eqref{eq:W-normal-exact}, with
	\[
		h=\eps f+t\eta,
	\]
	we have
	\[
		\bigl(W_{\alpha,0}\cdot n_h\bigr)\circ\Psi_h
		=
		-\alpha
		\frac{(1+h)\partial_\varphi h}
		{\sqrt{(1+h)^2+\abs{\nablaS h}^2}} .
	\]
	Since $f$ is axisymmetric $\partial_\varphi f=0$.
	Therefore
	\[
		\partial_\varphi h
		=
		\partial_\varphi(\eps f+t\eta)
		=
		t\partial_\varphi\eta .
	\]
	At $t=0$ we have $h=\eps f$ and $\nablaS h=\eps\nablaS f$.
	Thus differentiating with respect to $t$ at $t=0$ gives
	\begin{equation}
		g_\eps
		=
		-\alpha
		\frac{(1+\eps f)\partial_\varphi\eta}
		{\sqrt{(1+\eps f)^2+\eps^2\abs{\nablaS f}^2}},
	\end{equation}
	which we rewrite as
	\[
		g_\eps
		=
		-\alpha\partial_\varphi\eta
		\left(
		1+
		\frac{\eps^2\abs{\nablaS f}^2}{(1+\eps f)^2}
		\right)^{-1/2}.
	\]
	Hence
	\begin{equation}
		\label{eq:geps-no-linear-term}
		g_\eps
		=
		g+O(\eps^2).
	\end{equation}
	In particular, there is no order-$\eps$ correction to the prescribed
	Neumann datum itself. The only order-$\eps$ terms come from moving the
	boundary and moving the normal.

	On the round sphere, \eqref{eq:NtD-potentials} shows that the exterior
	harmonic function with Neumann datum $g$
	of degree $\ell $ is
	\[
		\eout_0(r,\omega)
		=
		-\frac1{\ell+1}
		r^{-\ell-1}
		g(\omega).
	\]
	Moreover,
	\begin{equation}
		\label{eq:eout0-derivatives-axisymmetric}
		\partial_{rr}\eout_0\big|_{r=1}
		=
		-(\ell+2)g,
		\qquad
		(\nabla\eout_0)_{\tan}\big|_{r=1}
		=
		-\frac1{\ell+1}\nablaS g.
	\end{equation}

	The boundary condition for $\eout_\eps$ is
	\begin{equation}\label{eq:boundary_eout_epssss}
		\partial_{n_{\eps f}}\eout_\eps\circ\Psi_{\varepsilon f}
		=
		g_\eps
		\qquad
		\text{on }\Sph.
	\end{equation}
	Using \eqref{eq:geps-no-linear-term}, \eqref{eq:boundary_eout_epssss} and
	Lemma~\ref{lem:eulerian-boundary-differentiation}, together with the expansion \eqref{eq:rotational-eout-eps-expansion},
	we obtain
	\begin{align*}
		\partial_r\eout_0
		+
		\eps
		\left(
		\partial_r v^{\rm out}
		+
		f\partial_{rr}\eout_0
		-
		\nablaS f\cdot(\nabla\eout_0)_{\tan}
		\right)_{r=1}=g+O(\varepsilon^2).
	\end{align*}
	The zeroth-order terms cancel, and
	since the right-hand side has no
	order-$\eps$ term, the coefficient of $\eps$ must vanish:
	\[
		\partial_r v^{\rm out}
		+
		f\partial_{rr}\eout_0
		-
		\nablaS f\cdot(\nabla\eout_0)_{\tan}
		=
		0
		\qquad
		\text{on } r=1.
	\]
	Substituting \eqref{eq:eout0-derivatives-axisymmetric} gives
	\[
		\partial_r v^{\rm out}\big|_{r=1}
		=
		(\ell+2)fg
		-
		\frac1{\ell+1}\nablaS f\cdot\nablaS g.
	\]
	This datum has zero flux. Indeed,
	$\Delta_{\Sph}g=-\ell(\ell+1)g$ and hence
	\begin{align*}
		\int_{\Sph}\left((\ell+2)fg
		-
		\frac1{\ell+1}\nablaS f\cdot\nablaS g\right)\dd S
		=
		2\int_{\Sph}fg\dd S
		=0.
	\end{align*}
	The last equality follows because $f$ is axisymmetric whereas $g$
	has nonzero azimuthal number. This proves that $v^{\rm out}$ satisfies the
	elliptic problem
	\eqref{eq:rotational-shape-axisymmetric-outer-problem}.

	It remains to compute the trace on the perturbed boundary. Taylor expanding
	in the radial variable at $r=1$, with $\omega$ fixed, gives
	\begin{align*}
		\eout_\eps\bigl((1+\eps f(\omega))\omega\bigr)
		 & =
		\bigl(\eout_0+\eps v^{\rm out}+O(\eps^2)\bigr)
		\bigl((1+\eps f(\omega))\omega\bigr) \\
		 & =
		\eout_0\big|_{r=1}(\omega)
		+
		\eps f(\omega)\partial_r\eout_0\big|_{r=1}(\omega)
		+
		\eps v^{\rm out}\big|_{r=1}(\omega)
		+
		O(\eps^2).
	\end{align*}
	Hence
	\[
		\left[\eout_\eps\big|_{r=1+\eps f}\right]_\eps
		=
		f\partial_r\eout_0\big|_{r=1}
		+
		v^{\rm out}\big|_{r=1}
		=
		fg+v^{\rm out}\big|_{r=1}.
	\]
	If
	\[
		\partial_r v^{\rm out}\big|_{r=1}
		=
		\sum_n k_n^{\rm out},
	\]
	then the exterior Neumann-to-Dirichlet formula \eqref{eq:NtD-traces} gives
	\[
		v^{\rm out}\big|_{r=1}
		=
		-\sum_n\frac1{n+1}k_n^{\rm out},
	\]
	which yields \eqref{eq:rotational-shape-axisymmetric-trace-outer}
\end{proof}

\begin{lemma}
	\label{lem:curvature-cross-h02}
	Let
	\begin{equation}
		f=-\frac{3}{16}P_2(Z),\quad P_2(Z)=\frac{3Z^2-1}{2}.
	\end{equation}
	For any spherical harmonic $Y_\ell^m$, the coefficient of $\eps t$ in
	$H_{\eps f+tY_\ell^m}$ is
	\begin{equation}
		\label{eq:curvature-cross-h02-formula}
		\left[H_{\eps f+tY_\ell^m}\circ\Psi_{\eps f+tY_\ell^m}\right]_{\eps t}
		=
		\frac38\left(\ell(\ell+1)+4\right)
		P_2(Z)Y_\ell^m .
	\end{equation}
	For complex normalised harmonics, its diagonal matrix element is
	\begin{equation}
		\label{eq:curvature-cross-h02-diagonal}
		\avg{Y_\ell^m,
			\left[H_{\eps f+tY_\ell^m}\circ\Psi_{\eps f+tY_\ell^m}\right]_{\eps t}
		}_{\CC}
		=
		\frac38\left(\ell(\ell+1)+4\right)I_{\ell m},
	\end{equation}
	where
	\begin{equation}
		\label{eq:Ilm-def}
		I_{\ell m}:=\int_{\Sph}P_2\abs{Y_\ell^m}^2\dd S
		=
		\frac{3\left((a_{\ell,+}^m)^2+(a_{\ell,-}^m)^2\right)-1}{2}.
	\end{equation}
\end{lemma}

\begin{proof}
	Taylor expansion of the exact radial-graph formula \eqref{eq:curvature-radial-graph} gives the second-order expansion, for a general height function $h$,
	\[
		H_h\circ\Psi_h
		=
		2
		-(\Delta_{\Sph}+2)h
		+
		Q(h)
		+R_3(h),
		\qquad
		Q(h):=2\left(h^2+h\Delta_{\Sph}h\right),
	\]
	where, for instance as a map $\C^2(\Sph)\to \C^0(\Sph)$ near $0$,
	\[
		\norm{R_3(h)}_{\C^0}
		\le C\norm{h}_{\C^2}^3,
		\qquad
		R_3(0)=0,
		\qquad
		DR_3(0)=0,
		\qquad
		D^2R_3(0)=0.
	\]

	The constant and linear terms in the expansion of $H_h\circ\Psi_h$
	contain no mixed monomial $\eps t$, and the remainder $R_3$
	contributes none because $D^2R_3(0)=0$. Thus it suffices to compute
	the mixed coefficient of
	\[
		Q(h)=2(h^2+h\Delta_{\Sph}h)
	\]
	with $h=\eps f+tY_\ell^m$. Direct expansion gives
	\[
		[Q(\eps f+tY_\ell^m)]_{\eps t}
		=
		4fY_\ell^m
		+
		2f\Delta_{\Sph}Y_\ell^m
		+
		2Y_\ell^m\Delta_{\Sph}f.
	\]
	Using $f=-\frac{3}{16}P_2(Z)$, $\Delta_{\Sph}P_2=-6P_2$, and
	$\Delta_{\Sph}Y_\ell^m=-\ell(\ell+1)Y_\ell^m$, this becomes
	\[
		-\frac34P_2Y_\ell^m
		+
		\frac38\ell(\ell+1)P_2Y_\ell^m
		+
		\frac94P_2Y_\ell^m
		=
		\frac38\left(\ell(\ell+1)+4\right)P_2Y_\ell^m.
	\]
	Hence
	\[
		\left[
			H_{\eps f+tY_\ell^m}\circ
			\Psi_{\eps f+tY_\ell^m}
			\right]_{\eps t}
		=
		\frac38
		\left(\ell(\ell+1)+4\right)
		P_2(Z)Y_\ell^m,
	\]
	which proves \eqref{eq:curvature-cross-h02-formula}.

	For the diagonal matrix element, we take the complex inner product of
	\eqref{eq:curvature-cross-h02-formula} with the complex normalised spherical harmonic $Y_\ell^m$. This yields
	\[
		\avg{Y_\ell^m,
			\left[
				H_{\eps f+tY_\ell^m}\circ
				\Psi_{\eps f+tY_\ell^m}
				\right]_{\eps t}
		}_{\CC}
		=
		\frac38
		\left(\ell(\ell+1)+4\right)
		\int_{\Sph}P_2\abs{Y_\ell^m}^2\dd S .
	\]
	Thus \eqref{eq:curvature-cross-h02-diagonal} follows with
	\[
		I_{\ell m}
		=
		\int_{\Sph}P_2\abs{Y_\ell^m}^2\dd S .
	\]

	Using the recursion \eqref{eq:Z-recursion} and the orthonormality of
	the complex spherical harmonics, the cross terms vanish and therefore
	\[
		\int_{\Sph}Z^2\abs{Y_\ell^m}^2\dd S
		=
		(a_{\ell,+}^m)^2+(a_{\ell,-}^m)^2 .
	\]
	Substituting this into the expression for
	$\int_{\Sph}P_2\abs{Y_\ell^m}^2\dd S$ gives \eqref{eq:Ilm-def}.
\end{proof}
\section{Linear analysis at the round sphere}
If $V=0$, the round sphere is a solution of the Bernoulli system
\eqref{eq:main-problem-nondimensional}:
\begin{align}\label{eq:unperturbed_solution}
	\cB(0,\alpha,0)=0\text{ for all }\alpha\in\R.
\end{align}
We now turn to the linearisation along this family of solutions.
\begin{lemma}
	\label{lem:linearisation-round-sphere}
	Let $V=0$ and fix $\alpha>0$. For every real spherical harmonic
	$Y_\ell^m$ of degree $\ell\ge1$ and azimuthal number $0\le m\le\ell$,
	with arbitrary normalisation, the
	linearisation of the Bernoulli equation is
	diagonal:
	\begin{equation}\label{eq:lambda-lm}
		D_h\cB(0,\alpha,0)Y_\ell^m
		=
		\lambda_{\ell,m}(\alpha)Y_\ell^m,
	\end{equation}
	where
	\begin{equation}\label{eq:lambda-lm-formula}
		\lambda_{\ell,m}(\alpha)
		=
		\ell(\ell+1)-2
		-\alpha^2 m^2
		\left(\frac{\tau}{\ell}+\frac{1}{\ell+1}\right).
	\end{equation}
\end{lemma}

\begin{proof}
	Let $\eta=Y_\ell^m$, $h=t\eta$,	where $Y_\ell^m$ is real, of degree $\ell\ge 1$ and azimuthal number
	$m$. We use the shape calculus from Section~\ref{sec:shape-calculus}.
	By Lemma~\ref{lem:first-order-neumann-expansion}, and in particular by
	\eqref{eq:W-normal-first-expansion}, with $V=0$, the first variation
	of the moving Neumann datum is the rotational datum
	\begin{equation}
		\label{eq:lrs-rotational-datum}
		g_\eta:=g_{{\rm rot},\eta}=-\alpha\partial_\varphi\eta,
	\end{equation}
	which is the datum appearing in \eqref{eq:rotational-potentials-def}.
	Moreover, \eqref{eq:boundary-gradient-expansion-inner} and
	\eqref{eq:boundary-gradient-expansion-outer} from
	Lemma~\ref{lem:first-order-neumann-expansion} give, for the pulled-back
	boundary gradients,
	\begin{equation}
		\label{eq:lrs-gradient-expansion}
		\qin_{t,0}=t\qin_{{\rm rot},\eta}+O(t^2),
		\qquad
		\qout_{t,0}=t\qout_{{\rm rot},\eta}+O(t^2).
	\end{equation}
	Since \eqref{eq:spherical-harmonic-eigen} implies
	$\partial_\varphi^2\eta=-m^2\eta$, the datum in
	\eqref{eq:lrs-rotational-datum} lies in the same degree-$\ell$ spherical
	harmonic subspace. Therefore Lemma~\ref{lem:neumann-to-dirichlet}, specifically the pure-mode
	Neumann-to-gradient formula \eqref{eq:NtG-pure-mode}, gives the tangential
	parts
	\begin{equation}
		\label{eq:lrs-rotational-tangential-gradients}
		(\qin_{{\rm rot},\eta})_{\tan}
		=
		\frac1\ell\nablaS g_\eta,
		\qquad
		(\qout_{{\rm rot},\eta})_{\tan}
		=
		-\frac1{\ell+1}\nablaS g_\eta.
	\end{equation}

	At $V=0$ one has
	\begin{equation}
		\label{eq:lrs-F-alpha-zero}
		F_{\alpha,0}(y,q)
		=
		\frac12\abs{q}^2-\alpha(e_3\times y)\cdot q.
	\end{equation}
	Since the zeroth-order boundary gradients vanish by
	\eqref{eq:lrs-gradient-expansion}, \eqref{eq:lrs-F-alpha-zero} implies that,
	for any boundary-gradient expansion $q_t=tq_1+O(t^2)$,
	\begin{equation}
		\label{eq:lrs-F-first-variation}
		\left[F_{\alpha,0}(\Psi_{t\eta},q_t)\right]_t
		=
		-\alpha(e_3\times\omega)\cdot q_1.
	\end{equation}
	The displacement of the evaluation point gives no additional term in
	\eqref{eq:lrs-F-first-variation}, because it is multiplied by the base
	gradient $q_0=0$ in \eqref{eq:lrs-F-alpha-zero}. The normal components from
	\eqref{eq:NtG-pure-mode} do not contribute to
	\eqref{eq:lrs-F-first-variation}, because $e_3\times\omega$ is tangential by
	\eqref{eq:dphi-def}. Using \eqref{eq:lrs-F-first-variation},
	\eqref{eq:lrs-gradient-expansion},
	\eqref{eq:lrs-rotational-tangential-gradients}, and
	\eqref{eq:dphi-def}, we obtain
	\begin{equation}
		\label{eq:lrs-F-linear-in-out}
		\begin{aligned}
			\left[F_{\alpha,0}(\Psi_{t\eta},\qin_{t,0})\right]_t
			 & =
			-\frac{\alpha}{\ell}\partial_\varphi g_\eta
			=
			\frac{\alpha^2}{\ell}\partial_\varphi^2\eta, \\
			\left[F_{\alpha,0}(\Psi_{t\eta},\qout_{t,0})\right]_t
			 & =
			\frac{\alpha}{\ell+1}\partial_\varphi g_\eta
			=
			-\frac{\alpha^2}{\ell+1}\partial_\varphi^2\eta .
		\end{aligned}
	\end{equation}

	The curvature contribution is read directly from Lemma~\ref{lem:radial-geometry},
	specifically \eqref{eq:curvature-expansion-one-parameter}. Using also
	\eqref{eq:spherical-harmonic-eigen},
	\begin{equation}
		\label{eq:lrs-curvature-linearisation}
		\left[H_{t\eta}\circ\Psi_{t\eta}\right]_t
		=
		-(\Delta_{\Sph}+2)\eta
		=
		\bigl(\ell(\ell+1)-2\bigr)\eta .
	\end{equation}

	Along the variation $h=t\eta$, the pulled-back, unprojected Bernoulli
	density is
	\begin{equation}
		\label{eq:lrs-unprojected-density}
		B_{t\eta}\circ\Psi_{t\eta}
		=
		\tau F_{\alpha,0}(\Psi_{t\eta},\qin_{t,0})
		-
		F_{\alpha,0}(\Psi_{t\eta},\qout_{t,0})
		+
		H_{t\eta}\circ\Psi_{t\eta}.
	\end{equation}
	Combining \eqref{eq:lrs-F-linear-in-out},
	\eqref{eq:lrs-curvature-linearisation}, and
	\eqref{eq:spherical-harmonic-eigen} gives
	\begin{align}
		\label{eq:lrs-unprojected-linearisation}
		[B_{t\eta}\circ\Psi_{t\eta}]_t
		 & =
		\alpha^2\left(\frac{\tau}{\ell}+\frac{1}{\ell+1}\right)
		\partial_\varphi^2\eta
		+
		\bigl(\ell(\ell+1)-2\bigr)\eta \\
		 & =
		\left[
			\ell(\ell+1)-2
			-
			\alpha^2m^2\left(\frac{\tau}{\ell}+\frac{1}{\ell+1}\right)
			\right]\eta .\nonumber
	\end{align}

	It remains to pass from the unprojected density to the projected Bernoulli map.
	By \eqref{eq:Bernoulli-map},
	\begin{equation}
		\cB(t\eta,\alpha,0)
		=
		\Pi_{t\eta}\left(B_{t\eta}\circ\Psi_{t\eta}\right).
	\end{equation}
	At $t=0$, \eqref{eq:lrs-F-alpha-zero}, \eqref{eq:lrs-gradient-expansion},
	\eqref{eq:lrs-unprojected-density}, and
	\eqref{eq:curvature-expansion-one-parameter} give the constant value
	$B_0=2$. Since the projection in \eqref{eq:Bernoulli-map} removes
	constants the derivative of $\Pi_{t\eta}B_0$ contributes nothing. We deduce
	\begin{equation}
		\label{eq:lrs-projected-linearisation}
		D_h\cB(0,\alpha,0)[\eta]
		=
		\Pi_0[B_{t\eta}\circ\Psi_{t\eta}]_t.
	\end{equation}
	Finally, \eqref{eq:Mh} gives $M_0=1$, so $\Pi_0$ is the usual removal of
	the spherical average. Since $\eta=Y_\ell^m$ is nonconstant, $\Pi_0\eta=\eta$.
	Applying \eqref{eq:lrs-projected-linearisation} to
	\eqref{eq:lrs-unprojected-linearisation} yields
	\[
		D_h\cB(0,\alpha,0)Y_\ell^m
		=
		\lambda_{\ell,m}(\alpha)Y_\ell^m,
	\]
	with $\lambda_{\ell,m}(\alpha)$ as in \eqref{eq:lambda-lm-formula}.
\end{proof}

\medskip

We explicitly compute the eigenvalues for the low modes. The constant mode
$\ell=0$ is
excluded by the volume constraint. For $\ell=1$ one has
\[
	\lambda_{10}(\alpha)=0,
\]
corresponding to vertical translations of the sphere, since the first-order
radial displacement of a vertical translation is proportional to $Z$. This
mode is removed by the vertical-centring condition
\eqref{eq:vertical-centring_C}. The horizontal translation modes are not neutral:
\begin{equation}\label{eq:lambda_11}
	\lambda_{11}(\alpha)
	=
	-\alpha^2\left(\tau+\frac{1}{2}\right)<0.
\end{equation}
This is due to the horizontal translations not being symmetries of the
rotational field $y\mapsto \alpha e_3\times y$.

We use the following dimensionless critical angular frequency:
\begin{equation}\label{eq:dimensionless-critical-constants}
	\alpha_*=\sqrt{\frac{24}{3\tau+2}}.
\end{equation}

For $\ell=2$,
\begin{equation}\label{eq:lambda_20}
	\lambda_{20}(\alpha)=4,
\end{equation}
and
\begin{equation*}
	\lambda_{21}(\alpha)
	=
	4
	-\alpha^2
	\left(\frac{\tau}{2}+\frac{1}{3}\right).
\end{equation*}
Thus $\lambda_{21}(\alpha)=0$ when
\[
	\alpha^2
	=
	\frac{4}{\frac{\tau}{2}+\frac{1}{3}}
	=
	\frac{24}{3\tau+2},
\]
which is $\alpha=\alpha_*$ as defined in \eqref{eq:dimensionless-critical-constants}. At this value,
\[
	\lambda_{22}(\alpha_*)
	=
	4
	-4\cdot4
	=
	-12,
\]
so the degree-two kernel consists only of the $m=1$ modes.

The transversality derivative is
\begin{equation}\label{eq:transversality}
	\partial_\alpha\lambda_{21}(\alpha_*)
	=
	-2\alpha_*
	\left(\frac{\tau}{2}+\frac{1}{3}\right)
	=
	-\frac{\alpha_*}{3}(3\tau+2)
	=
	-\frac{8}{\alpha_*}
	\neq0.
\end{equation}

\medskip

The real degree-two $m=1$ harmonics are spanned by $XZ$ and $YZ$. The corresponding real kernel is
\begin{equation*}
	K
	=
	\Span\{XZ,YZ\}.
\end{equation*}

We also record the degree-three $m=1$ eigenvalue at the critical
angular velocity. The functions
\[
	X\left(Z^2-\frac15\right),
	\qquad
	Y\left(Z^2-\frac15\right)
\]
are degree-three spherical harmonics with azimuthal number one.
Moreover,
\[
	\lambda_{31}(\alpha)
	=
	10
	-\alpha^2
	\left(\frac{\tau}{3}+\frac{1}{4}\right).
\]
Substituting $\alpha = \alpha_*$ from \eqref{eq:dimensionless-critical-constants} gives
\begin{align}\label{eq:lambda_31}
	\lambda_{31}(\alpha_*)
	 & =
	10
	-
	\frac{24}{3\tau+2}
	\left(\frac{\tau}{3}+\frac{1}{4}\right)
	=
	\frac{22\tau+14}{3\tau+2}
	>0 .
\end{align}

\medskip

At $\alpha=\alpha_*$, the desired kernel is the two-dimensional
space $K$, provided no further spherical harmonic satisfies
$\lambda_{\ell m}(\alpha_*)=0$. The remaining possible resonances are ruled
out by the following lemma.

\begin{lemma}\label{lem:nonresonance}
	At $\alpha=\alpha_*$, the only zero eigenvalues of
	\eqref{eq:lambda-lm-formula}, apart from the vertical translation mode
	$(\ell,m)=(1,0)$ and the desired pair $(\ell,m)=(2,1)$,
	occur, for positive density ratio $\tau$, at precisely the following
	five values:
	\[
		\begin{array}{c|c|c}
			(\ell,m) & \tau  & \text{resonant harmonic} \\
			\hline
			(5,4)    & 10/9  & Y_5^4                    \\
			(6,5)    & 2/7   & Y_6^5                    \\
			(8,8)    & 46/27 & Y_8^8                    \\
			(9,9)    & 23/60 & Y_9^9                    \\
			(10,10)  & 2/77  & Y_{10}^{10}.
		\end{array}
	\]
	At the endpoint $\tau=0$, there is one additional resonance,
	$(\ell,m)=(7,6)$, with resonant harmonic $Y_7^6$.
\end{lemma}

\begin{proof}
	At $\alpha=\alpha_*$, the equation $\lambda_{\ell m}=0$ becomes
	\begin{equation}\label{eq:resonance_condition_tau}
		\ell(\ell+1)-2
		=\frac{24m^2}{3\tau+2}
		\left(\frac{\tau}{\ell}+\frac1{\ell+1}\right).
	\end{equation}
	We only need to consider $\ell>2$. For fixed $\ell>2$, the coefficient
	\[
		\frac{24}{3\tau+2}\left(\frac{\tau}{\ell}+\frac1{\ell+1}\right)
	\]
	of $m^2$ is strictly decreasing in $\tau>0$, with limits
	$12/(\ell+1)$ and $8/\ell$. Hence a positive-$\tau$ resonance requires
	\begin{equation*}
		\frac8\ell
		<\frac{\ell(\ell+1)-2}{m^2}
		<\frac{12}{\ell+1},
	\end{equation*}
	with endpoint equality corresponding to $\tau=\infty$ or $\tau=0$, respectively. Since
	$m\le\ell$, the left-hand side is at least
	$(\ell(\ell+1)-2)/\ell^2$. For $\ell\ge11$ this lower bound is larger than or
	equal to $12/(\ell+1)$, so no resonances occur. The finitely many cases
	$3\le\ell\le10$ are then checked directly from
	\begin{equation*}
		\tau=
		\frac{24m^2/(\ell+1)-2(\ell(\ell+1)-2)}
		{3(\ell(\ell+1)-2)-24m^2/\ell}.
	\end{equation*}
	The five positive values are those displayed in the table. At $\tau=0$ we obtain $(\ell,m)=(7,6)$.
\end{proof}

\begin{remark}
	The case of vanishing exterior density is outside the present nondimensionalisation.
	In the corresponding one-phase drop problem there is another resonance at
	$(\ell,m) = (4,3)$. This obstruction was overcome in \cite{MR5001433},
	where bifurcating branches of rigidly rotating drops, without additional
	translation, were constructed.
\end{remark}

\section{The axisymmetric translating branch}
\label{sec:axisymmetric-translating-branch}

All three bifurcation regimes considered below use the same axisymmetric
translating solution as their base point. This branch was constructed also in \cite{MR4972964}.
We construct that solution once and record the properties needed later. For $q\in\N$, let
\begin{equation}
	\label{eq:axisymmetric-linear-slice}
	\H^q_{\rm ax}
	:=
	\set{
		\xi\in\H^q_{\rm lin}:
		\xi\circ\cR_\theta=\xi\text{ in }\H^q(\Sph)
		\text{ for all }\theta\in\R
	}.
\end{equation}
\begin{lemma}\label{lem:axisymmetric_input}
	Let $k\geq 4$. For all $\alpha\in\R$, $V\in\R$, $\tau\geq 0$ and axisymmetric $h\in U_\delta^{k+2}$ there holds
	\begin{equation}
		\label{eq:axisymmetric-bernoulli-reduction}
		\cB(h,\alpha,V)
		=
		\Pi_h\left[
			H_h\circ\Psi_h
			-\frac{V^2}{2}
			\abs{\Qout(h,0,1)-e_3}^2
			\right].
	\end{equation}
	In particular, on the axisymmetric set, the Bernoulli map is
	independent of $\alpha$ and $\tau$ and even in $V$. Moreover, $\cB(h,\alpha,V)$ is axisymmetric.
\end{lemma}

\begin{proof}
	If $h$ is
	axisymmetric, then the rotational field $\alpha e_3\times y$ is tangent to
	$\Gamma_h$. Thus the Neumann datum is the translational datum
	$Ve_3\cdot n_h$. The inner Neumann problem, whenever present, consequently
	has $\nabla\phiin_h=Ve_3$. In particular,
	$F_{\alpha,V}(y,\nabla\phiin_h)=0$. The outer potential is axisymmetric,
	independent of $\alpha$ and $\tau$, and linear in $V$ for fixed $h$. Its gradient has
	no azimuthal component, so the rotational part of $F_{\alpha,V}$ vanishes as
	well. Since
	$\Qout(h,\alpha,V)=V\Qout(h,0,1)$, this proves
	\eqref{eq:axisymmetric-bernoulli-reduction}. The axisymmetry of $\cB(h,\alpha,V)$ follows from \eqref{eq:curvature-radial-graph} and axisymmetry of the induced outer potential $\phi^{\rm out}_h$ of Proposition \ref{prop:neumann-sobolev}.
\end{proof}

\begin{lemma}
	\label{lem:axisymmetric-translating-branch}
	Fix $k\ge4$. There exist $V_{\rm ax}>0$ and $\delta_{\rm ax}>0$,
	independent of $\tau\ge0$ and $\alpha\in\R$, and a unique smooth map
	\begin{equation}
		\label{eq:axisymmetric-branch-map}
		(-V_{\rm ax},V_{\rm ax})\ni V
		\longmapsto
		\xi^{\rm ax}(V)\in\H^{k+2}_{\rm ax}
	\end{equation}
	with $\xi^{\rm ax}(0)=0$ such that, upon setting
	\begin{equation}
		\label{eq:axisymmetric-branch-definition}
		h^{\rm ax}_{0,V}:=\Theta(\xi^{\rm ax}(V)),
	\end{equation}
	the graph $h^{\rm ax}_{0,V}\in\cM^{k+2}$ satisfies
	\begin{equation}
		\label{eq:axisymmetric-branch-bernoulli}
		\cB(h^{\rm ax}_{0,V},\alpha,V)=0
		\qquad
		\text{for every }\tau\ge0\text{ and every }\alpha\in\R.
	\end{equation}
	If $h\in\cM^{k+2}$ is axisymmetric,
	$\norm{h}_{\H^{k+2}}<\delta_{\rm ax}$, $\abs V<V_{\rm ax}$, and
	$\cB(h,\alpha,V)=0$ for some $\tau\ge0$ and $\alpha\in\R$, then
	\begin{equation}
		\label{eq:axisymmetric-branch-uniqueness}
		h=h^{\rm ax}_{0,V}.
	\end{equation}
	The branch is independent of $\tau$ and $\alpha$, is even in the
	translation speed, and is invariant under reflection in the equatorial plane:
	\begin{equation}
		\label{eq:axisymmetric-branch-symmetries}
		h^{\rm ax}_{0,-V}=h^{\rm ax}_{0,V},
		\qquad
		h^{\rm ax}_{0,V}(X,Y,-Z)=h^{\rm ax}_{0,V}(X,Y,Z).
	\end{equation}
	The chart coordinate has the same symmetries:
	\begin{equation}
		\label{eq:axisymmetric-coordinate-symmetries}
		\xi^{\rm ax}(-V)=\xi^{\rm ax}(V),
		\qquad
		\xi^{\rm ax}(V)(X,Y,-Z)=\xi^{\rm ax}(V)(X,Y,Z).
	\end{equation}
	The centroid of
	$\Omega_{h^{\rm ax}_{0,V}}$ is the origin.
	Finally, with
	\begin{equation}
		P_2(Z):=\frac{3Z^2-1}{2},
		\qquad
		h_{02}:=-\frac{3}{16}P_2(Z),
	\end{equation}
	we have
	\begin{align}
		\label{eq:axisymmetric-branch-expansion}
		h^{\rm ax}_{0,V}
		 & =V^2h_{02}+O_{\H^{k+2}}(V^4), \\
		\label{eq:axisymmetric-coordinate-expansion}
		\xi^{\rm ax}(V)
		 & =V^2h_{02}+O_{\H^{k+2}}(V^4).
	\end{align}
	Consequently,
	\begin{equation}
		\label{eq:axisymmetric-weight-expansion}
		M_{h^{\rm ax}_{0,V}}
		=1+2V^2h_{02}+O_{\H^{k+2}}(V^4),
		\qquad
		\partial_Vh^{\rm ax}_{0,V}\big|_{V=0}
		=
		\partial_V\xi^{\rm ax}(0)=0.
	\end{equation}
	For each fixed $V$, the graph and the associated potentials are smooth up to
	the interface, with the exterior decay stated in
	Proposition~\ref{prop:neumann-sobolev}.
\end{lemma}

\begin{proof}
	We apply the implicit-function theorem directly to the Bernoulli map on the
	fixed-volume, vertically centred axisymmetric slice. The chart $\Theta$
	preserves axisymmetry by uniqueness of the two constraint corrections in
	\eqref{eq:constraint-correction-equations}. Let $\mathsf P_{\rm ax}$ be the
	$\L^2(\Sph)$-orthogonal projection that retains the axisymmetric modes and
	removes the modes $1$ and $Z$. It restricts to a bounded map from
	$\H^k(\Sph)$ onto $\H^k_{\rm ax}$. Set
	\begin{equation}
		\label{eq:axisymmetric-ift-map}
		\cG_{\rm ax}(\xi,V)
		:=
		\mathsf P_{\rm ax}\cB(\Theta(\xi),1,V),
		\qquad
		\xi\in\H^{k+2}_{\rm ax}.
	\end{equation}
	By Proposition~\ref{prop:smooth-bernoulli-map}, this is a smooth map into
	$\H^k_{\rm ax}$. Equation~\eqref{eq:axisymmetric-bernoulli-reduction} shows
	that neither the choice of angular velocity, fixed here at $1$, nor the value
	of $\tau$ used to define $\cB$ affects the map.

	We claim that, near the origin, zeros of \eqref{eq:axisymmetric-ift-map} are
	exactly zeros of the full Bernoulli map. Only the converse implication needs
	proof. Put $h=\Theta(\xi)$ and suppose $\cG_{\rm ax}(\xi,V)=0$. Since
	$\cB(h,1,V)$ is axisymmetric and its projection onto
	$\H^k_{\rm ax}$ vanishes, there are $a,b\in\R$ such that
	\begin{equation}
		\label{eq:axisymmetric-bernoulli-two-multipliers}
		\cB(h,1,V)=a+bZ.
	\end{equation}
	The definition of $\Pi_h$ gives the first identity
	\begin{equation}
		\label{eq:axisymmetric-multiplier-identities}
		\int_{\Sph}M_h\cB(h,1,V)\dd S=0.
	\end{equation}
	Defining the vertical-translation weight
	\begin{equation*}
		\mathsf K_h
		:=
		\bigl(e_3\cdot(n_h\circ\Psi_h)\bigr)J_h,
	\end{equation*}
	vertical-translation invariance of the reduced action, as used in the proof
	of Lemma~\ref{lem:variational_structure}, gives
	\begin{equation}
		\label{eq:axisymmetric-vertical-noether-identity}
		\int_{\Sph}\mathsf K_h\cB(h,1,V)\dd S=0.
	\end{equation}
	Indeed, this follows from \eqref{eq:tranlational_EEE}, minding that the variational and Bernoulli densities differ only by a constant,
	and $\int_{\Sph}\mathsf K_h\dd S=\int_{\Gamma_h}e_3\cdot n_h\dd S=0$.
	Substitution of \eqref{eq:axisymmetric-bernoulli-two-multipliers} into
	\eqref{eq:axisymmetric-multiplier-identities} and
	\eqref{eq:axisymmetric-vertical-noether-identity} gives a homogeneous system
	for $(a,b)$ with coefficient matrix
	\begin{equation*}
		\begin{pmatrix}
			\displaystyle\int_{\Sph}M_h\dd S
			 & \displaystyle\int_{\Sph}M_hZ\dd S         \\[0.6em]
			\displaystyle\int_{\Sph}\mathsf K_h\dd S
			 & \displaystyle\int_{\Sph}\mathsf K_hZ\dd S
		\end{pmatrix}.
	\end{equation*}
	At $h=0$ this matrix is
	$\operatorname{diag}(4\pi,4\pi/3)$, so it remains invertible near the round
	sphere. Thus $a=b=0$, proving the claim.

	At the round sphere, $D\Theta(0)=\id$ and
	$\cB(0,1,0)=0$. Lemma~\ref{lem:linearisation-round-sphere} therefore gives
	\begin{equation}
		\label{eq:axisymmetric-branch-linearisation}
		D_\xi\cG_{\rm ax}(0,0)Y_\ell^0
		=
		\bigl(\ell(\ell+1)-2\bigr)Y_\ell^0,
		\qquad \ell\ge2.
	\end{equation}
	The constant mode is absent from $\H^{k+2}_{\rm lin}$, and the mode $Z$ is
	absent because of the vertical-centring condition. The remaining
	eigenvalues are bounded below by a positive multiple of $1+\ell^2$. Hence
	\begin{equation*}
		D_\xi\cG_{\rm ax}(0,0):
		\H^{k+2}_{\rm ax}\longrightarrow\H^k_{\rm ax}
	\end{equation*}
	is an isomorphism. The implicit-function theorem yields the unique smooth
	map $\xi^{\rm ax}(V)$. The equivalence just proved gives
	\eqref{eq:axisymmetric-branch-bernoulli}, and IFT uniqueness gives
	\eqref{eq:axisymmetric-branch-uniqueness}. Together with
	\eqref{eq:axisymmetric-bernoulli-reduction}, this also proves independence of
	$\tau$ and $\alpha$.

	Evenness of the axisymmetric equation in $V$ and uniqueness give
	$h^{\rm ax}_{0,-V}=h^{\rm ax}_{0,V}$. If
	$S(X,Y,Z):=(X,Y,-Z)$, reflection of the graph by $S$ and transformation of the
	potentials by
	\begin{equation*}
		\widetilde\phi^{\rm in/out}(y)
		=-\phi^{\rm in/out}(Sy)
	\end{equation*}
	produce another axisymmetric solution with the same $V$. The two constraints
	are preserved by $S$, so local uniqueness gives the second identity in
	\eqref{eq:axisymmetric-branch-symmetries}. The chart $\Theta$ is injective
	and commutes with equatorial reflection. Applying $\Theta^{-1}$ to
	\eqref{eq:axisymmetric-branch-symmetries} therefore gives
	\eqref{eq:axisymmetric-coordinate-symmetries}. Moreover, axisymmetry eliminates the two
	horizontal components of the centroid, while the vertical-centring constraint
	eliminates its third component. Hence the centroid is the origin.

	It remains to compute the leading term. Smoothness and evenness give
	$h^{\rm ax}_{0,V}=V^2h_{02}+O_{\H^{k+2}}(V^4)$. By
	Lemma~\ref{lem:round-translational-potentials},
	\begin{equation*}
		\cB(0,1,V)
		=
		\Pi_0\left[-\frac98V^2(1-Z^2)+2\right]
		=
		\frac34V^2P_2(Z).
	\end{equation*}
	The coefficient of $V^2$ in the Bernoulli equation is therefore
	\begin{equation*}
		\left[\cB(h^{\rm ax}_{0,V},1,V)\right]_{V^2}=D_h\cB(0,1,0)h_{02}+\frac34P_2=0.
	\end{equation*}
	Since $P_2$ is the $(\ell,m)=(2,0)$ mode and its eigenvalue is $4$ by
	\eqref{eq:lambda_20}, this gives $h_{02}=-3P_2/16$. Finally,
	$\Theta(\xi)-\xi=O(\norm\xi^2)$ and $h^{\rm ax}_{0,V}=O(V^2)$ yield
	\eqref{eq:axisymmetric-coordinate-expansion}. The identities
	$M_h=(1+h)^2$ and \eqref{eq:axisymmetric-branch-expansion} give
	\eqref{eq:axisymmetric-weight-expansion}. Spatial smoothness follows
	from Lemma~\ref{lem:small-sobolev-solution-spatial-regularity}.
\end{proof}

\section{Gradient Lyapunov--Schmidt reduction with \texorpdfstring{$\so$}{SO(2)} symmetry}
Let $\cR_\theta$ denote rotation through angle $\theta$ around the
$e_3$-axis. The field $W_{\alpha,V}$ is equivariant:
\begin{equation*}
	W_{\alpha,V}(\cR_\theta y)
	=\cR_\theta W_{\alpha,V}(y).
\end{equation*}
In the setting of the two-phase reduced action
\eqref{eq:two_phase_reduced_action_definition}, let
$D_\theta=\cR_\theta D$, where
$D\in\set{\Omega,\R^3\setminus\overline{\Omega}}$. Uniqueness gives
$u_{D_\theta}(\cR_\theta y)=u_D(y)$; gradients, normals, surface measure,
and the Neumann data transform accordingly. Therefore
\begin{equation}\label{eq:action_is_invariant_under_rotations}
	\scrE(\cR_\theta\Omega,\alpha,V)
	=
	\scrE(\Omega,\alpha,V).
\end{equation}
This invariance carries over to the normal-graph formulation of the action
defined in \eqref{eq:graph_parametrized_two_phase_action}.
The fixed-volume condition \eqref{eq:vertical-centring_V} and the vertical-centring condition \eqref{eq:vertical-centring_C}
are also invariant under these rotations.

The two ingredients used in the lemma below are standard in bifurcation
theory. Symmetries are routinely used to reduce complexity,
while variational or gradient structures often remove
bad directions of the reduced problem; see, for example,
\cite{MR950168,MR2004250}. In free boundary fluid problems, related
mechanisms play a similar role in the construction of special coherent
states \cite{MR5001433,baldi2025bifurcationmultipleeigenvaluesrotating,MR1781220}.
We decided to combine both effects in a single lemma.

\begin{lemma}
	\label{lem:gradient-LS}
	Let $\cX_0$ be a real Hilbert space with a strongly continuous
	orthogonal $\so$-action $(\theta,\xi)\mapsto\cU_\theta\xi$. Let
	\[
		\cX_2\hookrightarrow\cX_1\hookrightarrow\cX_0
	\]
	be continuous embeddings of invariant Hilbert spaces, with $\cX_2$ dense
	in $\cX_1$, and assume that
	the restrictions of $\cU_\theta$ to $\cX_1$ and $\cX_2$ are strongly
	continuous bounded representations. Let $U\subset\cX_2$ be an invariant
	neighbourhood of $0$, let $\Lambda\subset\R^d$ be a neighbourhood of $0$,
	and let
	\[
		\scrJ\colon U\times\Lambda\to\R
	\]
	be a smooth $\so$-invariant functional. Suppose that its $\cX_0$-gradient
	\[
		\cF\colon U\times\Lambda\to\cX_1
	\]
	is smooth and satisfies
	\begin{equation}\label{eq:L2-gradient-representation}
		D_\xi\scrJ(\xi,\mu)[\eta]
		=
		\avg{\cF(\xi,\mu),\eta}_{\cX_0}
		, \qquad \eta \in \cX_2.
	\end{equation}
	Assume $\cF(0,0)=0$ and that
	\[
		L:=D_\xi\cF(0,0)\colon\cX_2\to\cX_1
	\]
	is Fredholm with a two-dimensional kernel $K\subset\cX_2$ containing no
	nonzero fixed vector of the $\so$-action. Let $\proj$ be the
	$\cX_0$-orthogonal projection onto $K$ and put $\idproj=\id-\proj$.
	Assume that
	\[
		L\colon K^\perp\cap\cX_2\longrightarrow K^\perp\cap\cX_1
	\]
	is an isomorphism, where orthogonality is taken in $\cX_0$. Then, for
	$a\in K$ and $\mu\in\Lambda$ small, the range equation
	\[
		\idproj\cF(a+w,\mu)=0,
		\qquad w\in K^\perp\cap\cX_2,
	\]
	has a unique small smooth solution $w=w(a,\mu)$. Moreover, the reduced
	kernel equation has the form
	\begin{equation}\label{eq:reduced-gradient-form}
		\proj\cF(a+w(a,\mu),\mu)
		=
		A(\abs a^2,\mu)a
	\end{equation}
	for a smooth scalar function $A$, where $\abs a$ is the norm induced by
	the $\cX_0$-inner product on $K$.
\end{lemma}

\begin{proof}
	Invariance of $\scrJ$ and orthogonality of the action on $\cX_0$ give,
	for $\eta\in\cX_2$,
	\[
		\avg{\cF(\cU_\theta\xi,\mu),\cU_\theta\eta}_{\cX_0}
		=
		\avg{\cF(\xi,\mu),\eta}_{\cX_0}.
	\]
	The difference
	$\cF(\cU_\theta\xi,\mu)-\cU_\theta\cF(\xi,\mu)$ belongs to $\cX_1$.
	Choose a sequence $(\eta_j)_j$ in $\cX_2$ converging to this difference in $\cX_1$,
	hence also in $\cX_0$. The preceding identity then shows that its
	$\cX_0$-norm vanishes. Thus
	\[
		\cF(\cU_\theta\xi,\mu)=\cU_\theta\cF(\xi,\mu).
	\]
	Differentiating at $(\xi,\mu)=(0,0)$ shows that $L$ is equivariant, so
	$K$ is invariant. Furthermore, differentiating
	\eqref{eq:L2-gradient-representation} once more gives the Hessian symmetry
	\[
		\avg{L\xi,\eta}_{\cX_0}
		=
		\avg{\xi,L\eta}_{\cX_0}
		, \qquad \xi,\eta \in \cX_2.
	\]
	Thus $L$ maps $K^\perp\cap\cX_2$ into $K^\perp\cap\cX_1$. Since $K$ is
	finite-dimensional and contained in $\cX_2$, $\proj$ and $\idproj$ restrict
	to bounded operators on each $\cX_j$, $j=0,1,2$, and commute with the
	action.

	The derivative with respect to $w$ of
	\[
		(a,w,\mu)\longmapsto\idproj\cF(a+w,\mu)
	\]
	at $(0,0,0)$ is the assumed isomorphism
	$L\colon K^\perp\cap\cX_2\to K^\perp\cap\cX_1$. The implicit-function
	theorem therefore gives the unique smooth solution $w=w(a,\mu)$. By
	equivariance and uniqueness,
	\begin{equation}\label{eq:w-equivariant}
		w(\cU_\theta a,\mu)=\cU_\theta w(a,\mu).
	\end{equation}

	Define
	\[
		\Phi(a,\mu)=\scrJ(a+w(a,\mu),\mu).
	\]
	Equations \eqref{eq:w-equivariant} and the invariance of $\scrJ$ show that
	$\Phi(\cU_\theta a,\mu)=\Phi(a,\mu)$. For $b\in K$, the range equation and
	$D_aw(a,\mu)[b]\in K^\perp$ give
	\begin{align*}
		D_a\Phi(a,\mu)[b]
		 & =\avg{\cF(a+w,\mu),b+D_aw(a,\mu)[b]}_{\cX_0} \\
		 & =\avg{\proj\cF(a+w,\mu),b}_{\cX_0}.
	\end{align*}
	Consequently,
	\begin{equation}\label{eq:reduced-gradient-identity}
		\nabla_a\Phi(a,\mu)=\proj\cF(a+w(a,\mu),\mu).
	\end{equation}

	Because $K$ is a two-dimensional nontrivial real representation of
	$\so$, there is an orthogonal complex structure $J$ on $K$ and an integer
	$q\ge1$ such that
	\[
		\cU_\theta\big|_K=\cos(q\theta)\id+\sin(q\theta)J.
	\]
	The standard smooth equivariant normal form \cite[Chapter~XII, Sections~4--5]{MR950168} on this representation gives
	\[
		\proj\cF(a+w(a,\mu),\mu)
		=A(\abs a^2,\mu)a+B(\abs a^2,\mu)Ja
	\]
	with smooth scalar functions $A$ and $B$. On the other hand, differentiating
	the invariance of $\Phi$ along the group orbit and using
	\eqref{eq:reduced-gradient-identity} yields
	\[
		0=\avg{\nabla_a\Phi(a,\mu),qJa}_{\cX_0}
		=qB(\abs a^2,\mu)\abs a^2.
	\]
	Since $q\ge1$, it follows that $B=0$ away from $a=0$, and continuity gives
	$B(0,\mu)=0$. This proves
	\eqref{eq:reduced-gradient-form}.
\end{proof}

\begin{remark}
	\label{rem:LS_higher_dimensional_kernel}
	If the kernel has dimension greater than two, the range construction,
	equivariance of \(w\), and the reduced-gradient identity
	\eqref{eq:reduced-gradient-identity} remain valid, although the reduced
	equation need not be radial. Setting
	\[
		g(a,\mu)
		:=
		\proj\cF(a+w(a,\mu),\mu),
	\]
	invariance gives
	\[
		\left\langle g(a,\mu),Ja\right\rangle_{\cX_0}=0,
		\qquad
		Ja
		=
		\frac{\ddd}{\ddd\theta}\Big|_{\theta=0}\cU_\theta a.
	\]

	In particular, suppose that \(K=E\oplus F\) is an orthogonal invariant splitting,
	where \(E\) is a nontrivial two-dimensional real representation.
	Write \(a=e+f\) and \(g=g_E+g_F\) according to this splitting. If
	\[
		e\ne0,
		\qquad
		g_F(a)=0,
		\qquad
		\left\langle g_E(a),e\right\rangle_{\cX_0}=0,
	\]
	then \(g(a)=0\). Indeed, the orbit identity gives
	\(\left\langle g_E(a),Je\right\rangle_{\cX_0}=0\), and
	\(e,Je\) span \(E\).
\end{remark}

\section{The nonresonant case \texorpdfstring{$\tau>0$}{tau > 0}}
\label{C3}
We recall
\begin{equation*}
	\cO=
	\set{\frac{2}{77}, \frac{2}{7},
		\frac{23}{60},\frac{10}{9},\frac{46}{27}}.
\end{equation*}

First, we prove the nondimensional version of Theorem \ref{thm:din-pos-dimensional}
in the case $\tau>0$ and $\tau\notin\cO$.

\begin{theorem}
	\label{thm:din-pos-dimensionless}
	Assume
	\[
		\tau>0,
		\qquad
		\tau\notin\cO .
	\]
	Let
	\[
		\alpha_*^2=\frac{24}{3\tau+2},
		\qquad
		\lambda_{31}(\alpha_*)
		=
		\frac{22\tau+14}{3\tau+2}.
	\]
	Then, for every $k\ge4$, there are constants
	$V_0,s_0,\eta>0$, depending only on $k$ and $\tau$,
	such that, for every $V\in\R$ with $\abs{V}<V_0$ and every $\abs{s}<s_0$,
	there exists $\alpha=\alpha(s,V)>0$ such that the nondimensional
	free boundary problem \eqref{eq:main-problem-nondimensional}, with
	$\dout=1$ and $\din=\tau$, has a smooth solution $\Omega(s,V)$ whose
	boundary is the radial graph
	\[
		\Gamma(s,V)
		=
		\{(1+h(s,V)(\omega))\omega:\omega\in\Sph\}
	\]
	with $h(s,V)\in \C^\infty(\Sph)$.
	After rotation about the $e_3$-axis, the branch is normalised by the
	kernel direction $sXZ$, and then has the following expansion:
	\begin{equation}
		\label{eq:h-expansion-dimensionless-theorem}
		\begin{aligned}
			h(s,V)
			={} &
			-\frac{3V^2}{16}\frac{3Z^2-1}{2}
			+sXZ                     \\
			    & \quad
			-\sqrt{6(3\tau+2)}\,Vs\,Y
			\left[
				\frac{3}{40(2\tau+1)}
				+
				\frac{5}{22\tau+14}
				\left(Z^2-\frac15\right)
				\right] \\
			    & \quad
			+O_{\H^{k+2}(\Sph)}
			\left(s^2+\abs{s}V^2+V^4\right).
		\end{aligned}
	\end{equation}
	Moreover, the associated frequencies and centroids satisfy
	\begin{gather}
		\label{eq:alpha-expansion-dimensionless-theorem}
		\alpha(s,V)
		=
		\alpha_*+O(s^2+V^2),\\\label{eq:centroid-expansion-dimensionless-theorem}
		\bar{y}_{s,V}:=\frac{3}{4\pi}\int_{\Omega(s,V)}y\dd y
		=-\frac{9}{10(2\tau+1)\alpha_*}sVe_2+O(s^2+\abs{s}V^2)
	\end{gather}
	with $\bar{y}_{s,V}\cdot e_3=0$.
	In particular, if in addition $0<\abs{s}\le\eta\abs{V}$, then
	$\bar{y}_{s,V}\cdot e_2\neq 0$. If $s=0$, then $h(0,V)$ is axisymmetric and $\bar{y}_{0,V}=0$.
	For each fixed $V$, the non-axisymmetric branch is locally unique
	modulo rotations about the $e_3$-axis among near-spherical solutions in the
	fixed-volume, vertically centred slice with frequency near $\alpha_*$. On the axisymmetric set, the shape is
	locally unique, while $\alpha$ remains free. The precise statement is
	Lemma~\ref{lem:nonlinear-branch-smooth-main-proof}.
\end{theorem}

\subsection{The \texorpdfstring{$\so$}{SO(2)} Lyapunov--Schmidt construction}

\begin{lemma}
	\label{lem:nonlinear-branch-smooth-main-proof}
	Assume $\tau>0$ and $\tau\notin\cO$, and fix an integer
	$k\ge4$. Then there are $s_0,V_0>0$ and a smooth map
	\[
		(-s_0,s_0)\times(-V_0,V_0)
		\ni (s,V)
		\mapsto
		\left(\alpha(s,V),h(s,V)\right)
		\in (0,\infty)\times\cM^{k+2}
	\]
	with
	\[
		\alpha(0,0)=\alpha_*,
		\qquad
		h(0,0)=0,
		\qquad h(0,V)=h^{\rm ax}_{0,V},\qquad
		\partial_s h(0,0)=XZ,
	\]
	such that, for every $\abs{s}<s_0$ and $\abs{V}<V_0$, the near-spherical
	domain $\Omega(s,V)$ with boundary given by the radial graph
	\[
		\Gamma(s,V)
		=
		\{(1+h(s,V)(\omega))\omega:\omega\in\Sph\},
	\]
	together with the associated interior and exterior Neumann potentials and a
	suitable constant $c=c(s,V)$ solves the overdetermined free boundary value
	problem \eqref{eq:main-problem-nondimensional} with $\alpha=\alpha(s,V)$.
	Let $P_K$ be the $\L^2(\Sph)$-orthogonal projection onto
	$K=\Span\set{XZ,YZ}$. After decreasing $s_0,V_0$, there is
	$\varepsilon>0$ such that the following local classification holds. Fix the
	speed $V$ with $\abs V<V_0$, and let
	$(\widetilde\alpha,\widetilde h)$ be any solution in the same
	fixed-volume, vertically centred slice satisfying
	\[
		\abs{\widetilde\alpha-\alpha_*}
		+\norm{\widetilde h}_{\H^{k+2}}<\varepsilon .
	\]
	Set $a=P_K\Theta^{-1}(\widetilde h)$. If $a\ne0$, there are unique
	\[
		0<s=\frac{\norm a_{\L^2}}{\norm{XZ}_{\L^2}}<s_0,
		\qquad \theta\in\R/(2\pi\mathbb Z),
	\]
	such that
	\[
		(\widetilde\alpha,\widetilde h)
		=
		\bigl(\alpha(s,V),\cU_\theta h(s,V)\bigr).
	\]
	If $a=0$, then $\widetilde h=h(0,V)$, whereas
	$\widetilde\alpha$ is arbitrary in the indicated neighbourhood. Thus the non-axisymmetric branch is locally unique modulo rotations about the
	$e_3$-axis. On the axisymmetric set, the shape is locally unique, while
	$\widetilde\alpha$ remains free.
	Moreover, each solution on this branch is smooth, i.e.
	$h(s,V)\in \C^\infty(\Sph)$, and the corresponding potentials are smooth
	up to $\Gamma(s,V)$.
\end{lemma}

\begin{proof}
	We recall that $\cR_\theta$ denotes rotation by angle $\theta$ around the
	$e_3$-axis, and let
	\[
		(\cU_\theta f)(\omega)
		=
		f(\cR_{-\theta}\omega)
	\]
	be the induced action on functions on $\Sph$.

	In order to incorporate the constraints \eqref{eq:vertical-centring_C} and
	\eqref{eq:vertical-centring_V} defining the manifold $\cM^{k+2}$, we prefer
	to work in the local smooth chart
	$\Theta:U\subset\H^{k+2}_{\rm lin}\rightarrow\cM^{k+2}$ constructed in
	\eqref{eq:linear-slice} and \eqref{eq:Theta-chart}.
	In particular, \eqref{eq:Theta-expansion} gives
	\[
		\Theta(0)=0,\qquad
		D\Theta(0)=\id_{\H^{k+2}_{\rm lin}},
		\qquad
		\Theta(\xi)-\xi=O\left(\norm{\xi}_{\H^{k+2}}^2\right).
	\]

	The chart is $\so$-equivariant. Indeed, the constraints
	$\cV$, $\cC_3$, and the linear slice $\H^{k+2}_{\rm lin}$ are invariant
	under $\cU_\theta$, and the
	correction $c(\xi)+d(\xi)Z$ in \eqref{eq:Theta-chart} is uniquely determined.
	Hence we may choose $U$ invariant and obtain
	\begin{equation}
		\Theta(\cU_\theta\xi)
		=
		\cU_\theta\Theta(\xi),
		\qquad \xi\in U, \,
		\theta\in\R/(2\pi\mathbb Z).
	\end{equation}

	We apply Lemma~\ref{lem:gradient-LS} with
	\[
		\cX_0=\L^2_{\rm lin}
		:=\set{g\in\L^2(\Sph):\avg{g,1}_{\L^2}=\avg{g,Z}_{\L^2}=0},
		\qquad
		\cX_1=\H^k_{\rm lin},
		\qquad
		\cX_2=\H^{k+2}_{\rm lin}.
	\]
	These spaces have the required continuous dense embeddings, and rotations
	act orthogonally on $\cX_0$ and boundedly on both Sobolev spaces.
	Define
	the pulled-back constrained action
	\[
		\scrJ(\xi,\alpha,V)
		:=
		\scrE(\Omega_{\Theta(\xi)},\alpha,V)=\scrE(\Theta(\xi),\alpha,V)
	\]
	with $\xi\in U$ and $(\alpha,V)\in\R^2$, as our functional. For the
	definition of $\scrE$, we recall
	\eqref{eq:two_phase_reduced_action_definition} and
	\eqref{eq:graph_parametrized_two_phase_action}.

	First of all, Lemma~\ref{lem:variational_structure} and the regularity of the
	chart $\Theta$ imply that $\scrJ$ is smooth. Moreover,
	$\cR_\theta\Omega_{\Theta(\xi)}=\Omega_{\cU_\theta\Theta(\xi)}
		=\Omega_{\Theta(\cU_\theta\xi)}$ and
	\eqref{eq:action_is_invariant_under_rotations} imply that
	$\scrJ(\cdot,\alpha,V)$ is $\so$-invariant.

	Let $\cF(\xi,\alpha,V)$
	be its
	$\L^2(\Sph)$-gradient, i.e.
	\begin{equation*}
		D_\xi\scrJ(\xi,\alpha,V)[\eta]
		=
		\avg{\cF(\xi,\alpha,V),\eta}_{\L^2(\Sph)}
		\qquad
		\text{for every }\eta\in\H^{k+2}_{\rm lin}.
	\end{equation*}
	The regularity of $\scrJ$ a priori only implies that
	$\cF(\cdot,\alpha,V)$ is smooth as a map
	$U\rightarrow\H^{-(k+2)}_{\rm lin}$. However,
	since $D\Theta(\xi)$ maps $\H^{k+2}_{\rm lin}$ isomorphically onto
	$T_{\Theta(\xi)}\cM$,
	the chain rule and Lemma~\ref{lem:variational_structure} imply that
	\begin{align}\label{eq:derivative-scrJ}
		D_\xi\scrJ(\xi,\alpha,V)[\eta]=\int_{\Sph}\cB(\Theta(\xi),\alpha,V)M_{\Theta(\xi)}D\Theta(\xi)\eta\dd S
	\end{align}
	for all $\xi\in U$, $(\alpha,V)\in\R^2$, and
	$\eta\in\H^{k+2}_{\rm lin}$.
	The Bernoulli map and the chart $\Theta$ are smooth, cf. Proposition \ref{prop:smooth-bernoulli-map}. In order to precisely characterise $\cF$ and its regularity, we now compute the adjoint of $D\Theta(\xi)$.
	For $\xi\in U$ set $h=\Theta(\xi)$ and
	\[
		M_h=(1+h)^2,
		\qquad
		N_h=(1+h)^3Z.
	\]
	Since
	\[
		D\Theta(\xi)[\eta]
		=
		\eta+Dc(\xi)[\eta]+Dd(\xi)[\eta]Z,
	\]
	differentiating the two constraint identities
	$\cV(\Theta(\xi))=\cC_3(\Theta(\xi))=0$ in the direction
	$\eta\in\H^{k+2}_{\rm lin}$ and using
	\eqref{eq:radial-volume-weight}--\eqref{eq:radial-C3-variation} gives
	\[
		\mathsf A_h
		\binom{Dc(\xi)[\eta]}{Dd(\xi)[\eta]}
		=
		-\binom{\avg{M_h,\eta}_{\L^2}}{\avg{N_h,\eta}_{\L^2}},
		\qquad
		\mathsf A_h
		:=
		\begin{pmatrix}
			\avg{M_h,1}_{\L^2} & \avg{M_h,Z}_{\L^2} \\
			\avg{N_h,1}_{\L^2} & \avg{N_h,Z}_{\L^2}
		\end{pmatrix}.
	\]
	At $h=0$ this matrix is
	\[
		\mathsf A_0
		=
		\begin{pmatrix}
			4\pi & 0               \\
			0    & \dfrac{4\pi}{3}
		\end{pmatrix},
	\]
	so it is invertible for $h$ near zero. Write
	$\mathsf A_h^{-1}=(\mathsf a_{ij}(h))_{i,j=1}^2$. Then
	\[
		\begin{aligned}
			r_c(\xi)
			 & :=-\proj_{\rm lin}\bigl(
			\mathsf a_{11}(h)M_h+\mathsf a_{12}(h)N_h\bigr), \\
			r_d(\xi)
			 & :=-\proj_{\rm lin}\bigl(
			\mathsf a_{21}(h)M_h+\mathsf a_{22}(h)N_h\bigr)
		\end{aligned}
	\]
	belong to $\H^{k+2}_{\rm lin}$, depend smoothly on $\xi$, and satisfy
	\[
		Dc(\xi)[\eta]=\avg{r_c(\xi),\eta}_{\L^2},
		\qquad
		Dd(\xi)[\eta]=\avg{r_d(\xi),\eta}_{\L^2}.
	\]
	The $\L^2$-adjoint of
	$D\Theta(\xi)\colon\H^{k+2}_{\rm lin}\to\H^{k+2}(\Sph)$ when paired with an $\H^k$-function defines
	the smooth map
	$D\Theta(\xi)^*\colon\H^k(\Sph)\to\H^k_{\rm lin}$ given by
	\[
		D\Theta(\xi)^*g
		=
		\proj_{\rm lin}g
		+
		\avg{g,1}_{\L^2}r_c(\xi)
		+
		\avg{g,Z}_{\L^2}r_d(\xi).
	\]
	Hence
	\begin{equation}\label{eq:relation_F_to_cB}
		\cF(\xi,\alpha,V)
		=
		D\Theta(\xi)^*
		\left[\cB(\Theta(\xi),\alpha,V)M_{\Theta(\xi)}\right].
	\end{equation}
	Proposition~\ref{prop:smooth-bernoulli-map} and the smoothness of $\Theta$
	now show that
	\[
		\cF\colon U\times\R^2\longrightarrow\H^k_{\rm lin}
	\]
	is smooth.

	Furthermore, $\cF(0,\alpha_*,0)=0$ because $\Theta(0)=0$ and
	\eqref{eq:unperturbed_solution} holds. Let
	\[
		L_*:=D_\xi\cF(0,\alpha_*,0)
		:\H^{k+2}_{\rm lin}\longrightarrow\H^k_{\rm lin}.
	\]
	Differentiating \eqref{eq:relation_F_to_cB} and using
	$\Theta(0)=0$, $D\Theta(0)=\id_{\H^{k+2}_{\rm lin}}$,
	$M_0=1$, and $\cB(0,\alpha_*,0)=0$ gives first
	\[
		L_*=\proj_{\rm lin}D_h\cB(0,\alpha_*,0)
		\quad\text{on }\H^{k+2}_{\rm lin}.
	\]
	The spherical-harmonic diagonalisation
	\eqref{eq:lambda-lm} shows that $D_h\cB(0,\alpha_*,0)$ preserves the linear
	slice. Hence the projection may be dropped and
	\[
		L_*=D_h\cB(0,\alpha_*,0)\big|_{\H^{k+2}_{\rm lin}}.
	\]

	The same diagonalisation shows that
	$L_*$ acts on every admissible spherical harmonic of
	degree $\ell$ and azimuthal number $m$ as multiplication by
	$\lambda_{\ell m}(\alpha_*)$. The constant mode is removed by the volume
	constraint and the mode $Z$ is removed by the vertical-centring gauge, i.e. $Y^0_0,Y^0_1\notin \H^{k+2}_{\rm lin}$. By
	Lemma~\ref{lem:nonresonance},
	\[
		K:=\ker L_*
		=
		\Span\{XZ,YZ\}.
	\]
	Note that $K$ does not have any fixed points of the $\so$ action except $0$.

	Next we claim
	that
	\[
		L_*\colon
		K^\perp\cap\H^{k+2}_{\rm lin}
		\to
		K^\perp\cap\H^k_{\rm lin}
	\]
	is an isomorphism.
	Indeed, with respect to the real spherical harmonics,
	we have that $K^\perp\cap \H^{k+2}_{\rm lin}$ is spanned by $\left(Y^m_\ell\right)$, $(\ell,m)\notin \set{(0,0),(1,0),(2,1)}$, and, as said, in this basis
	$L_*$ is diagonal with
	eigenvalues
	\[
		\lambda_{\ell m}(\alpha_*)
		=
		\ell(\ell+1)-2
		-
		\alpha_*^2m^2
		\left(\frac{\tau}{\ell}+\frac{1}{\ell+1}\right).
	\]
	Since $m\le\ell$, the second term is $O(\ell)$ uniformly in $m$, while the
	first term is positive and of order $\ell^2$. Hence there are constants
	$c>0$ and $\ell_0\in\N$ such that
	\[
		\abs{\lambda_{\ell m}(\alpha_*)}
		\ge
		c(1+\ell)^2
		\qquad
		\text{for all }\ell\ge\ell_0 .
	\]
	For the finitely many remaining modes $\ell\le \ell_0$
	outside	$K$, nonresonance gives a positive minimum of
	$\abs{\lambda_{\ell m}(\alpha_*)}$. Therefore
	\begin{equation}
		\norm{u}_{\H^{k+2}(\Sph)}
		\le
		C\norm{L_*u}_{\H^k(\Sph)}
		\qquad
		\text{for all }
		u\in K^\perp\cap\H^{k+2}_{\rm lin}.
	\end{equation}
	Thus $L_*$ is Fredholm and has a bounded inverse from
	$K^\perp\cap\H^k_{\rm lin}$ to $K^\perp\cap\H^{k+2}_{\rm lin}$.

	We therefore can apply Lemma \ref{lem:gradient-LS}.
	Let $\proj$ be the $\L^2(\Sph)$-orthogonal projection onto $K$, and let
	$\idproj=\id-\proj$. Applying said lemma with the parameters
	$
		\mu=(\alpha-\alpha_*,V),
	$
	we obtain a unique local smooth range correction
	\[
		w=w(a,\alpha,V)\in K^\perp\cap\H^{k+2}_{\rm lin},
		\qquad a\in K,
	\]
	satisfying
	\begin{equation*}
		\idproj\cF(a+w(a,\alpha,V),\alpha,V)=0.
	\end{equation*}
	The remaining reduced kernel equation has the radial form
	\begin{equation}
		\label{eq:kernel-equation-main-proof}
		\proj\cF(a+w(a,\alpha,V),\alpha,V)
		=
		A(\norm{a}_{\L^2}^2,\alpha,V)a
	\end{equation}
	for a scalar smooth function $A$.

	To identify the coefficient $A$ at the origin of the reduced equation, one
	must linearise \eqref{eq:kernel-equation-main-proof} in the kernel variable
	$a$. Fix $V=0$ and
	for $\alpha$ near $\alpha_*$, write
	\begin{equation}\label{eq:L_alpha_expressions}
		L_\alpha:=D_\xi\cF(0,\alpha,0)
		=\proj_{\rm lin}D_h\cB(0,\alpha,0).
	\end{equation}
	The round-sphere operator $D_h\cB(0,\alpha,0)$ is diagonal in spherical
	harmonics and therefore preserves $\H^{k+2}_{\rm lin}$. Thus, on the linear
	slice, $L_\alpha=D_h\cB(0,\alpha,0)$.
	Since the round sphere is a solution for every $\alpha$ when $V=0$, cf. \eqref{eq:unperturbed_solution}, we have
	$w(0,\alpha,0)=0$.
	Moreover, differentiating the range equation
	\[
		\idproj\cF(a+w(a,\alpha,0),\alpha,0)=0
	\]
	at $a=0$ in a direction $b\in K$ gives
	\[
		\idproj L_\alpha\left(b+D_aw(0,\alpha,0)[b]\right)=0.
	\]
	The space $K$ is invariant under $L_\alpha$,
	as it acts as multiplication by $\lambda_{21}(\alpha)$. Hence
	$\idproj L_\alpha b=0$, and the invertibility of $\idproj L_\alpha$ on $K^\perp$
	persists, after shrinking the neighbourhood of $\alpha_*$, by
	operator-norm continuity of $L_\alpha$ and openness of the set of bounded
	isomorphisms. It therefore
	implies
	\[
		D_aw(0,\alpha,0)[b]=0 .
	\]
	Therefore the linearisation of the left-hand side of
	\eqref{eq:kernel-equation-main-proof} is
	\[
		D_a\left[
			\proj\cF(a+w(a,\alpha,0),\alpha,0)
			\right]_{a=0}[b]
		=
		\proj L_\alpha b
		=
		\lambda_{21}(\alpha)b .
	\]
	On the other hand,
	\[
		D_a\left[
			A(\norm{a}_{\L^2}^2,\alpha,0)a
			\right]_{a=0}[b]
		=
		A(0,\alpha,0)b.
	\]
	Comparing the two
	linearisations for arbitrary $b\in K$ gives
	\[
		A(0,\alpha,0)
		=
		\lambda_{21}(\alpha).
	\]
	Therefore, $A(0,\alpha_*,0)=0$ by the definition of $\alpha_*$, and by \eqref{eq:transversality},
	\begin{equation}
		\partial_\alpha A(0,\alpha_*,0)
		=
		\partial_\alpha\lambda_{21}(\alpha_*)
		=
		-\frac{8}{\alpha_*}
		\neq0 .
	\end{equation}
	The implicit-function theorem gives a unique smooth function
	\[
		\alpha=\widehat\alpha(\varrho,V),
		\qquad
		\varrho=\norm{a}_{\L^2}^2,
	\]
	near $(0,0)$, such that
	\[
		A(\varrho,\widehat\alpha(\varrho,V),V)=0,
		\qquad
		\widehat\alpha(0,0)=\alpha_* .
	\]
	Consequently, for every small $a\in K$, the element
	\[
		\xi(a,V)
		=
		a+w(a,\widehat\alpha(\norm{a}_{\L^2}^2,V),V)
	\]
	solves the full
	gradient equation
	\[
		\cF(\xi(a,V),\widehat\alpha(\norm{a}_{\L^2}^2,V),V)=0.
	\]
	By \eqref{eq:derivative-scrJ}, and since $D\Theta(\xi)$ is an isomorphism
	between $\H^{k+2}_{\rm lin}$ and $T_{\Theta(\xi)}\cM^{k+2}$, this is
	equivalent to $\Theta(\xi)$ being a critical point of
	$\scrE(\cdot,\widehat\alpha(\norm{a}_{\L^2}^2,V),V)$ on $\cM^{k+2}$ and
	thus, by Lemma~\ref{lem:variational_structure}, equivalent to
	$\cB(\Theta(\xi(a,V)),\widehat\alpha(\norm{a}_{\L^2}^2,V),V)=0$. In turn,
	this is equivalent to the free boundary problem
	\eqref{eq:main-problem-nondimensional}; cf. \eqref{eq:Bernoulli-map}.

	We now fix the phase by choosing
	\[
		a=sXZ,
		\qquad s\in\R,
		\abs{s}\text{ small}.
	\]
	Define
	\[
		\alpha(s,V)
		:=
		\widehat\alpha(\norm{sXZ}_{\L^2}^2,V),
		\qquad
		h(s,V):=\Theta(\xi(sXZ,V)).
	\]
	Since $w(0,\alpha_*,0)=0$, $D_aw(0,\alpha_*,0)=0$ on $K$, and
	$D\Theta(0)=\id$, we have
	\[
		\alpha(0,0)=\alpha_*,
		\qquad
		h(0,0)=0,
		\qquad
		\partial_s h(0,0)=XZ .
	\]
	After shrinking the neighbourhood if necessary, we may assume
	$\alpha(s,V)>0$ and
	$1+h(s,V)>0$ throughout the branch.
	The other phases are obtained by applying $\cU_\theta$, which is
	equivalently rotation of the interface around the $e_3$-axis.
	This proves the existence part of the lemma.

	We next prove the asserted classification.
	After decreasing $V_0$ if necessary, choose $\varepsilon>0$ so that
	every solution satisfying the neighbourhood condition in the lemma lies,
	in the above coordinates, within the neighbourhoods of both
	implicit-function arguments.
	A free-boundary solution in
	the constraint slice satisfies the constrained-gradient equation. Hence, if
	$\widetilde\xi=\Theta^{-1}(\widetilde h)$ and
	$a=P_K\widetilde\xi$, then
	\[
		\cF(\widetilde\xi,\widetilde\alpha,V)=0
	\]
	and uniqueness in the range equation gives
	\begin{equation}
		\label{eq:nonresonant-classification-range}
		\widetilde\xi=a+w(a,\widetilde\alpha,V).
	\end{equation}
	The action on the kernel is
	\[
		\cU_\theta(XZ)=\cos\theta\,XZ+\sin\theta\,YZ,
		\qquad
		\cU_\theta(YZ)=-\sin\theta\,XZ+\cos\theta\,YZ.
	\]
	If $a\ne0$, there are unique
	\[
		s=\frac{\norm a_{\L^2}}{\norm{XZ}_{\L^2}}>0,
		\qquad \theta\in\R/(2\pi\mathbb Z),
	\]
	such that $a=\cU_\theta(sXZ)$. The kernel equation obtained from
	\eqref{eq:nonresonant-classification-range} is
	\[
		A(\norm a_{\L^2}^2,\widetilde\alpha,V)a=0.
	\]
	Since $a\ne0$, the scalar implicit-function theorem gives
	\[
		\widetilde\alpha
		=
		\widehat\alpha(\norm a_{\L^2}^2,V)
		=
		\alpha(s,V).
	\]
	Equivariance of $w$ and of $\Theta$ now yields
	\[
		(\widetilde\alpha,\widetilde h)
		=
		\bigl(\alpha(s,V),\cU_\theta h(s,V)\bigr).
	\]
	Projecting this identity back onto $K$ proves uniqueness of $s$ and
	$\theta$.

	Decrease $V_0$ further, if necessary, so that
	$V_0\le V_{\rm ax}$.
	Suppose instead that $a=0$. Equivariance and uniqueness in the range
	equation imply that $w(0,\widetilde\alpha,V)$, and hence
	$\widetilde h$, is axisymmetric. The local uniqueness assertion in
	Lemma~\ref{lem:axisymmetric-translating-branch} therefore gives
	\[
		\widetilde h=h^{\rm ax}_{0,V}=h(0,V).
	\]
	Conversely, Lemma~\ref{lem:axisymmetric-translating-branch} shows that this
	graph solves the Bernoulli equation for every nearby
	$\widetilde\alpha$.
	\medskip
	The construction gives
	\[
		h(s,V)\in\H^{k+2}(\Sph)
	\]
	and the preceding argument proves that it satisfies the Bernoulli
	equation.
	Lemma~\ref{lem:small-sobolev-solution-spatial-regularity} therefore gives
	\[
		h(s,V)\in\C^\infty(\Sph)
	\]
	for every fixed $(s,V)$, as well as smoothness of the corresponding
	potentials up to the interface and the stated exterior decay. This
	completes the proof of the lemma.
\end{proof}

\subsection{Asymptotics of the branch}
In order to guarantee a genuine helial centroid motion we compute further
Taylor coefficients of the branch
constructed by the Lyapunov--Schmidt reduction. Recall that the branch is
parametrised by
\[
	a=sXZ
\]
and
\[
	\xi(s,V)
	=
	sXZ
	+
	w(sXZ,\alpha(s,V),V),
	\qquad
	h(s,V)
	=
	\Theta(\xi(s,V)).
\]
The function $sXZ$ is the chosen direction in the kernel
\[
	K=\ker L_*
	=
	\Span\{XZ,YZ\},
\]
where $L_*=L_{\alpha_*}$ with $L_\alpha$ given in \eqref{eq:L_alpha_expressions}.
The remaining part
\[
	w(sXZ,\alpha(s,V),V)\in K^\perp
\]
is the range correction determined by the range equation.

We write a Taylor expansion
\begin{equation} \label{eq:taylor-h-main-proof}
	h(s,V)
	=
	h_{00}
	+s h_{10}
	+Vh_{01}
	+s^2h_{20}
	+sVh_{11}
	+V^2h_{02}
	+\cdots
\end{equation}
with the convention that no
factorials are included in the notation for Taylor coefficients.

By Lemma \ref{lem:nonlinear-branch-smooth-main-proof} we already know that
\begin{align*}
	h_{00}=0,\quad h_{10}=XZ.
\end{align*}
Moreover, the identification
\begin{equation}
	\label{eq:nonresonant-axisymmetric-identification}
	h(0,V)=h^{\rm ax}_{0,V}
\end{equation}
of Lemma \ref{lem:nonlinear-branch-smooth-main-proof}
and the evenness in $V$, cf. \eqref{eq:axisymmetric-branch-symmetries}, imply $h_{01}=0$ and
\begin{align}\label{eq:preliminary_expansion_of_hsv}
	h(s,V)
	=
	sXZ
	+
	V^2h_{02}
	+
	sVh_{11}
	+
	O(s^2+\abs{s}V^2+V^4).
\end{align}

\medskip

\begin{lemma}
	\label{lem:h11-coefficient-main-proof}
	The mixed $sV$-coefficient in the expansion \eqref{eq:taylor-h-main-proof} is
	\[
		h_{11}
		=
		-\frac{9}{20(\tau+1/2)}\frac1{\alpha_*}Y
		-
		\frac{5\alpha_*}{2\lambda_{31}(\alpha_*)}
		Y\left(Z^2-\frac15\right).
	\]
\end{lemma}

\begin{proof}
	Set
	\[
		\delta h:=h(s,V)-sXZ .
	\]
	Then the $sV$-part of $\delta h$ is $sVh_{11}$. Taylor expanding
	the Bernoulli map in the $h$-variable around the leading deformation $sXZ$
	gives
	\begin{align*}
		0=\cB(h(s,V),\alpha(s,V),V)
		 & =
		\cB(sXZ,\alpha(s,V),V) \\
		 & \quad
		+
		D_h\cB(sXZ,\alpha(s,V),V)[\delta h]
		+
		O(\norm{\delta h}^2).
	\end{align*}
	Since
	\[
		\delta h
		=
		V^2h_{02}
		+
		sVh_{11}
		+
		O(s^2+\abs{s}V^2+V^4),
	\]
	the quadratic remainder contains no $sV$-term. Moreover, by smoothness of
	$\cB$,
	\[
		D_h\cB(sXZ,\alpha(s,V),V)
		=
		D_h\cB(0,\alpha_*,0)
		+
		O(s)+O(V)+O(\alpha(s,V)-\alpha_*).
	\]
	Hence, when this operator is applied to $\delta h$, only the leading operator
	can contribute to the $sV$-coefficient. Therefore
	\[
		\left[
			D_h\cB(sXZ,\alpha(s,V),V)[\delta h]
			\right]_{sV}
		=
		L_*h_{11},
		\qquad
		L_*=D_h\cB(0,\alpha_*,0).
	\]

	It remains to identify the order-$sV$ part of
	\[
		\cB(sXZ,\alpha(s,V),V).
	\]
	There are two possible sources. First, for fixed $\alpha=\alpha_*$, the
	leading deformation $sXZ$ interacts with the translational datum $VZ$.
	Second, the parameter $\alpha(s,V)$ itself may have a linear dependence on
	$V$, and this can multiply the linear $s$-dependence of the Bernoulli map.
	We separate these two effects.

	Since
	\begin{equation}\label{eq:first_expansion_alpha}
		\alpha(s,V)
		=
		\alpha_*
		+
		\beta V
		+
		O(s^2+V^2),
		\qquad
		\beta:=\partial_V\alpha(0,0),
	\end{equation}
	where there is no term linear in $s$ because
	$\alpha(s,V)=\widehat\alpha(\norm{sXZ}_{\L^2}^2,V)$, Taylor expansion in
	$\alpha$ gives
	\begin{align*}
		 & \cB(sXZ,\alpha(s,V),V)   \\
		 & =
		\cB(sXZ,\alpha_*,V)
		+
		\beta V
		D_\alpha\cB(sXZ,\alpha_*,V) \\
		 & \hphantom{=}+
		\text{terms which do not contribute to order }sV .
	\end{align*}
	The only possible additional order-$sV$ contribution coming from the linear
	$V$-variation of $\alpha$ is therefore
	\[
		\beta sV\,
		D_\alpha D_h\cB(0,\alpha_*,0)[XZ].
	\]
	There is no separate order-$\alpha V$ term at $h=0$, by the round-sphere
	identity \eqref{eq:round-translational-F}.

	We recall that
	$
		L_\alpha=D_h\cB(0,\alpha,0).
	$
	Since $XZ$ is the $\ell=2$, $m=1$ eigenmode,
	\[
		L_\alpha XZ
		=
		\lambda_{21}(\alpha)XZ,
		\qquad
		\lambda_{21}(\alpha_*)=0.
	\]
	Hence
	\[
		D_\alpha D_h\cB(0,\alpha_*,0)[XZ]
		=
		\lambda_{21}'(\alpha_*)XZ
		\in K.
	\]
	So far this gives the following identity for the $sV$-coefficients
	\begin{align}\label{eq:intermediate_summary_sV}
		0=\left[\cB(sXZ,\alpha_*,V)\right]_{sV}+\beta\lambda'_{21}(\alpha_*)XZ+L_*h_{11}.
	\end{align}

	The remaining
	term is the mixed $sV$ shape-calculus
	coefficient in Lemma~\ref{lem:mixed-sV-residual-XZ}, evaluated at
	$\alpha=\alpha_*$.
	This is the only point in the argument where the round-sphere computation
	\eqref{eq:round-translational-F} is needed. Since the unprojected Bernoulli
	residual at $h=0$ has no term linear in $V$, differentiating $\Pi_h$ with
	respect to $h$ produces no additional order-$sV$ term. Thus the mixed
	coefficient is obtained by applying the fixed projection $\Pi_0$, which
	subtracts its spherical average.
	To be more precise,
	\begin{align*}
		N_{sV} & :=\left[\cB(sXZ,\alpha_*,V)\right]_{sV}=\left[\Pi_{sXZ}\left(B_{sXZ,\alpha_*,V}\circ\Psi_{sXZ}\right)\right]_{sV} \\
		       & =\Pi_{0}\left(\left[B_{sXZ,\alpha_*,V}\circ\Psi_{sXZ}\right]_{sV}\right)=	-\alpha_*
		\left[
			\frac9{20}Y
			-\frac52Y\left(Z^2-\frac15\right)
			\right].
	\end{align*}
	Here the last equality follows from \eqref{eq:B-sV-residual} and the fact that $\Pi_0[Y]=Y$, $\Pi_0[YZ^2]=YZ^2$.
	The
	modes
	appearing on the right-hand side are the $\ell=1,m=1$ mode $Y$ and the
	$\ell=3,m=1$ mode $Y(Z^2-1/5)$, hence $N_{sV}\in K^\perp$.

	By \eqref{eq:intermediate_summary_sV} we obtain
	\begin{align*}
		L_*h_{11}+N_{sV}=-\beta\lambda'_{21}(\alpha_*)XZ
	\end{align*}
	with the left-hand side being in $K^\perp$, recall that $L_*$ maps into $K^\perp$, and the right-hand side in $K$. By \eqref{eq:transversality} we deduce
	\begin{align}\label{eq:beta_gleich_0}
		\beta=\partial_V\alpha(0,0)=0,
	\end{align}
	as well as
	\[
		L_*h_{11}+N_{sV}=0.
	\]

	We claim that $h_{11}\in K^\perp$. Indeed,
	Lemma \ref{lem:nonlinear-branch-smooth-main-proof} gives
	$\partial_Vh(0,0)=0$. Differentiating the volume and vertical-centring
	constraints then shows that $h_{11}\in\H^{k+2}_{\rm lin}$. Moreover,
	$\Theta(\xi)=\xi+c(\xi)+d(\xi)Z$ and
	$\xi(s,V)=sXZ+w(sXZ,\alpha(s,V),V)$ with $w\in K^\perp$, so that differentiation shows,
	$h_{11}\in K^\perp$.

	Hence, $L_*h_{11}$ uniquely determines $h_{11}$ and we can proceed
	as in the proof of Lemma~\ref{lem:axisymmetric-translating-branch}. We use
	\[
		\lambda_{11}(\alpha_*)
		=-\alpha_*^2\left(\tau+\frac12\right),
		\qquad
		\lambda_{31}(\alpha_*)
		=\frac{22\tau+14}{3\tau+2},
	\]
	cf.\ \eqref{eq:lambda_11} and \eqref{eq:lambda_31}. Solving mode by mode gives
	\begin{equation*}
		h_{11}
		=
		-\frac{9}{20(\tau+1/2)}\frac1{\alpha_*}Y
		-
		\frac{5\alpha_*}{2\lambda_{31}(\alpha_*)}
		Y\left(Z^2-\frac15\right).
	\end{equation*}
	This is the stated coefficient.
\end{proof}

Combining Lemmas \ref{lem:nonlinear-branch-smooth-main-proof}--\ref{lem:h11-coefficient-main-proof}, the dimensionless branch expansion is
\begin{equation}\label{eq:dimensionless-h-expansion-main-proof}
	\begin{aligned}
		h(s,V)
		={} &
		-\frac{3}{16}V^2P_2(Z)
		+sXZ
		-\frac{9}{20(\tau+1/2)}\frac{V}{\alpha_*}sY \\
		    & \quad
		-\frac{5\alpha_*}{2\lambda_{31}(\alpha_*)}Vs\,
		Y\left(Z^2-\frac15\right)
		+O(s^2+\abs{s}V^2+V^4).
	\end{aligned}
\end{equation}

Finally, by means of \eqref{eq:first_expansion_alpha}, \eqref{eq:beta_gleich_0} we also have shown the following statement.
\begin{lemma}
	\label{lem:alpha-expansion-order-main-proof}
	There holds
	\begin{equation}
		\alpha(s,V)
		=
		\alpha_*+O(s^2+V^2).
	\end{equation}
\end{lemma}

\subsection{Helical motion of the centroid}

\begin{lemma}
	\label{lem:helical-centroid-main-proof}
	For the branch constructed above, as $(s,V)\to(0,0)$,
	\begin{equation}\label{eq:dimensionless-centroid-expansion-main-proof}
		\bar{y}_{s,V}:=\frac1{\abs{\Omega(s,V)}}
		\int_{\Omega(s,V)}y\dd y
		=
		-\frac{9}{20(\tau+1/2)}
		\frac{V}{\alpha_*}s\,e_2
		+
		O(s^2+\abs{s}V^2),
	\end{equation}
	and $\bar{y}_{s,V}\cdot e_3=0$.
	Moreover, there is $\eta>0$ such that, after possibly decreasing
	$V_0>0$, whenever
	$0<\abs{V}<V_0$ and
	$
		0<\abs{s}\le \eta\abs{V},
	$
	one has
	$
		\bar{y}_{s,V}\cdot e_2
		\neq 0.
	$
\end{lemma}

\begin{proof}
	For a
	radial graph $r=1+h(\omega)$, there holds
	\[
		\int_{\Omega_h}y\dd y
		=
		\frac14\int_{\Sph}(1+h)^4\omega\dd S,
	\]
	cf. \eqref{eq:vertical-centring_C}.
	Since $\abs{\Omega_h}=\frac{4\pi}{3}$,
	we thus have
	\begin{equation}
		\label{eq:horizontal-centroid-exact-main-proof}
		\bar{y}_{s,V}
		=
		\frac3{16\pi}\int_{\Sph}(1+h(s,V))^4\omega\dd S .
	\end{equation}
	The constraint $h(s,V)\in\cM^{k+2}$ gives
	$\bar y_{s,V}\cdot e_3=0$. At $s=0$, the identification
	\eqref{eq:nonresonant-axisymmetric-identification} and
	Lemma~\ref{lem:axisymmetric-translating-branch} give
	$\bar y_{0,V}=0$ for all sufficiently small $V$.
	Thus the centroid has no pure $V$-terms. By
	smoothness,
	\begin{equation}
		\label{eq:centroid-factor-s-main-proof}
		\bar{y}_{s,V}
		=
		s\partial_s \bar{y}_{0,V}+O(s^2).
	\end{equation}
	It remains to compute $\partial_s \bar{y}_{0,V}$. Differentiating
	\eqref{eq:horizontal-centroid-exact-main-proof} gives
	\[
		\partial_s \bar{y}_{0,V}
		=
		\frac{3}{4\pi}\int_{\Sph}
		(1+h(0,V))^3
		\partial_s h(0,V)
		\omega\dd S .
	\]
	Now the Taylor expansion
	\eqref{eq:preliminary_expansion_of_hsv}
	gives
	\[
		h(0,V)=O(V^2),
		\qquad
		\partial_s h(0,V)
		=
		XZ+Vh_{11}+O(V^2).
	\]
	Consequently
	\begin{align*}
		\partial_s \bar{y}_{0,V}
		 & =
		\frac{3}{4\pi}\int_{\Sph}(XZ+Vh_{11})\omega\dd S
		+
		O(V^2)=\frac{3V}{4\pi}\int_{\Sph}h_{11} \omega \dd S+O(V^2).
	\end{align*}
	The coefficient $h_{11}$ has been computed in Lemma \ref{lem:h11-coefficient-main-proof}. The $\ell=3$ part of
	$h_{11}$, i.e. $Y(Z^2-\frac15)$, is orthogonal to the
	$\ell=1$ modes $X,Y,Z$, while the $\ell=1$
	part
	together with
	\eqref{eq:centroid-integrals}
	gives the
	contribution
	\[
		\partial_s \bar{y}_{0,V}
		=
		-\frac{9}{20(\tau+1/2)}
		\frac{V}{\alpha_*}e_2
		+
		O(V^2).
	\]
	Combining this with \eqref{eq:centroid-factor-s-main-proof}, we obtain
	\[
		\bar y_{s,V}
		=
		-\frac{9}{20(\tau+1/2)}
		\frac{sV}{\alpha_*}e_2
		+
		O(s^2+\abs{s}V^2),
	\]
	which is \eqref{eq:dimensionless-centroid-expansion-main-proof}.

	Finally, choose $\eta>0$ sufficiently small, decrease $V_0$ if
	necessary so that $\eta V_0<s_0$, and assume
	$0<\abs{s}\le\eta\abs{V}$.
	The coefficient of $sVe_2$ in \eqref{eq:dimensionless-centroid-expansion-main-proof} is non-zero,
	so the leading term has size comparable to $\abs{s}\abs{V}$, whereas
	the ratio of $s^2+\abs{s}V^2$ to $\abs{sV}$ is bounded by
	$\eta+\abs V$. It is therefore a lower-order perturbation after choosing
	$\eta$ and $V_0$ sufficiently small.
	Hence $\bar y_{s,V}\cdot e_2\neq0$ whenever $s\neq0$ in this regime.
\end{proof}

\subsection{Conclusion}
\begin{proof}[Proof of Theorem \ref{thm:din-pos-dimensionless}]
	Lemma~\ref{lem:nonlinear-branch-smooth-main-proof} gives a smooth branch of
	solutions of the nondimensional problem
	\eqref{eq:main-problem-nondimensional}, with
	\[
		\alpha(0,0)=\alpha_*,
		\qquad
		h(0,0)=0,
		\qquad h(0,V)=h^{\rm ax}_{0,V},\qquad
		\partial_s h(0,0)=XZ .
	\]
	For solutions with nonzero kernel projection, the phase is fixed by
	choosing the unique representative $sXZ$ with $s>0$; all other phases are
	obtained by rotation about the $e_3$-axis.
	The final part of
	Lemma~\ref{lem:nonlinear-branch-smooth-main-proof} gives the asserted local
	classification, including the free value of $\alpha$ on the axisymmetric
	set.

	The expansion of the interface is exactly
	\eqref{eq:dimensionless-h-expansion-main-proof}, while
	Lemma~\ref{lem:alpha-expansion-order-main-proof} gives
	\eqref{eq:alpha-expansion-dimensionless-theorem}. Finally,
	Lemma~\ref{lem:helical-centroid-main-proof} gives
	\eqref{eq:centroid-expansion-dimensionless-theorem} and the non-vanishing of
	the horizontal centroid in the regime
	$0<\abs{s}\le\eta\abs V$.
\end{proof}

\begin{proof}[Proof of Theorem \ref{thm:din-pos-dimensional} in the case $\tau>0$ and $\tau\notin\cO$]
	Set
	\[
		\tau=\frac{\din}{\dout},
		\qquad
		\bar V=\left(\frac{\sigma}{\dout R}\right)^{1/2},
		\qquad
		\widehat V=\frac{V}{\bar V}.
	\]
	Then
	\[
		\widehat V=\sqrt{\We}>0,
		\qquad
		\widehat V^2=\We.
	\]
	Apply
	Theorem~\ref{thm:din-pos-dimensionless} with $k=4$, this value of $\tau$,
	and the speed $\widehat V$.
	To distinguish the nondimensional solution objects from their rescaled
	counterparts below, we denote the resulting domain, graph function,
	potentials, angular velocity, and Bernoulli constant by hats. We also write
	\[
		\widehat\alpha_*
		:=
		\left(\frac{24}{3\tau+2}\right)^{1/2}>0,
		\qquad
		\widehat\lambda_{31}(\widehat\alpha_*)
		:=
		\frac{22\tau+14}{3\tau+2}
	\]
	for the corresponding dimensionless critical constants. Set
	\[
		\We_0:=\widehat{V}_0^2 .
	\]
	Thus, in the case of positive velocities,
	$0<\We<\We_0$ is exactly $0<\widehat V<\widehat{V}_0$.

	Rescale the nondimensional solution as in
	Section~\ref{sec:nondimensionalisation}:
	\[
		\Omega(s,\We)=R\widehat\Omega(s,\widehat V),
		\qquad
		\Gamma(s,\We)=R\widehat\Gamma(s,\widehat V),
		\qquad
		h(s,\We)=\widehat h(s,\widehat V),
	\]
	\[
		\phinout(y)
		=
		R\bar V\,
		\widehat\phi^{\rm in/out}\!\left(\frac yR\right),
		\qquad
		\alpha(s,\We)=\frac{\bar V}{R}\widehat\alpha(s,\widehat V),
		\qquad
		c(s,\We)=\frac{\sigma}{R}\widehat c(s,\widehat V).
	\]
	By the nondimensionalisation, this gives a solution of
	\eqref{eq:main-problem}, and the volume constraint becomes
	\[
		\abs{\Omega(s,\We)}
		=
		R^3\abs{\widehat\Omega(s,\widehat V)}
		=
		\frac{4\pi}{3}R^3 .
	\]

	Since $h=\widehat h$ by
	\eqref{eq:radial-height-rescaling}, the expansion
	\eqref{eq:h-expansion} follows immediately from
	\eqref{eq:h-expansion-dimensionless-theorem}, using
	\[
		\widehat V=\sqrt{\We},
		\qquad
		\widehat V^2=\We,
	\]
	and the identities, cf. \eqref{eq:dimensionless-critical-constants}
	and \eqref{eq:lambda_31},
	\begin{align}\label{eq:numerical_identities}
		\frac{9}{20(\tau+1/2)}\frac1{\widehat\alpha_*}
		 & =
		\sqrt{6(3\tau+2)}\,\frac{3}{40(2\tau+1)},
		\\
		\frac{5\widehat\alpha_*}{2\widehat\lambda_{31}(\widehat\alpha_*)}
		 & =
		\sqrt{6(3\tau+2)}\,\frac{5}{22\tau+14}.\nonumber
	\end{align}
	The dimensionless remainder is
	$O_{\H^6}(s^2+\abs{s}\We+\We^2)$.  The embedding
	$\H^6(\Sph)\hookrightarrow\C^4(\Sph)$ and
	$\abs{s}\We\le\frac12(s^2+\We^2)$ yield
	\eqref{eq:h-expansion}.

	Moreover,
	\[
		\alpha_*
		=
		\frac{\bar V}{R}\widehat\alpha_*,
		\qquad
		\alpha_*^2
		=
		\frac{\sigma}{\dout R^3}\frac{24}{3\tau+2}
		=
		\frac{24\sigma}{R^3(3\din+2\dout)} .
	\]
	Remark~\ref{rem:nonresonant-sharp-alpha-expansion} below and
	$\widehat V=\sqrt{\We}$ give \eqref{eq:alpha-expansion}.

	Finally, the centroid rescales by
	\[
		\bar{y}_{s,\We}
		=
		\frac1{\abs{\Omega(s,\We)}}\int_{\Omega(s,\We)}y\dd y
		=
		\frac{R}{\abs{\widehat\Omega(s,\widehat V)}}
		\int_{\widehat\Omega(s,\widehat V)}\widehat y\dd \widehat y
		=R\widehat{\bar y}_{s,\widehat V}.
	\]
	Therefore \eqref{eq:centroid-expansion} follows from
	\eqref{eq:centroid-expansion-dimensionless-theorem}
	and the first identity in \eqref{eq:numerical_identities}.
	Choose the same $\eta$ as in the dimensionless theorem. If
	$0<\abs{s}\le \eta\sqrt{\We}$, then
	$0<\abs{s}\le \eta\abs{\widehat V}$, so the non-vanishing of the horizontal
	centroid is inherited from Theorem~\ref{thm:din-pos-dimensionless}.
	The rescaling is bijective and preserves the fixed-volume,
	vertically centred radial-graph class and rotations about the $e_3$-axis. Hence it also transfers smoothness and the stated
	local classification.
\end{proof}

\section{The endpoint case \texorpdfstring{$\tau=0$}{tau = 0}}
In this section we prove the existence of the spiralling bubbles in the endpoint case $\tau=0$.
Throughout this section, weight $m$ refers to the azimuthal frequency $m$.
At the endpoint, the linearisation of the Bernoulli equation at the sphere, at the Rayleigh--Lamb frequency and at speed $V=0$, restricted to the fixed-volume,
vertically centred slice, has two, instead of one, resonant planes, given by
\begin{equation}
	K_1=\Span\set{XZ,YZ}
\end{equation}
and
\begin{equation}
	K_6
	=
	\Span
	\set{
		Z\Re((X+iY)^6),
		Z\Im((X+iY)^6)
	}.
\end{equation}
The subscripts denote azimuthal weight, not spherical-harmonic degree: $K_1$
is degree two and $K_6$ is degree seven.

Following the additional neutral modes of azimuthal weight six along the axisymmetric translating branch, it turns out that the associated
eigenvalue becomes nonzero but is only of order $V^2$, so solving for this mode
loses a factor of order $V^{-2}$. The underlying $\so$-equivariance compensates for this
loss: the primary weight-one mode can force the weight-six mode only at sixth
order in its amplitude.

After solving for the modes outside $K_1\oplus K_6$, we capture this balance
by an anisotropic blow-up, meaning that the different variables are rescaled
by different powers of $V$:
\[
	a=VtXZ,
	\qquad
	u=V^2U,
	\qquad
	\alpha=\alpha_1(V)+V^2\beta,
\]
where $a\in K_1$ is the primary amplitude and $u\in K_6$ is the additional
mode. Here $t\in\mathbb R$ is the rescaled amplitude of the primary mode, and
$\alpha_1(V)$ is the critical angular velocity at translation speed $V$
from which the non-axisymmetric branch bifurcates.
In these variables, the two remaining, appropriately rescaled equations extend smoothly to
$V=0$, and their linearisation with respect to $(U,\beta)$ is invertible. The
implicit-function theorem in combination with the $\so$-gradient structure then determines the additional-mode correction and
the angular velocity for $\norm a=O(V)$, with
\[
	\norm u=O\left(\frac{\norm a^6}{V^2}\right).
\]

We obtain the following result.

\begin{theorem}
	\label{thm:endpoint-din-zero}
	Assume $\tau=0$. For every
	$k\ge4$, there exist constants $V_0>0$ and $\eta>0$ such that, for every
	$0<V<V_0$ and every $\abs{s}\le\eta V$, there are smooth solutions of the
	nondimensional free-boundary problem \eqref{eq:main-problem-nondimensional},
	with prescribed volume $4\pi/3$, angular speed
	$\alpha=\alpha(s,V)>0$, and boundaries
	\[
		\Gamma(s,V)=\set{(1+h(s,V)(\omega))\omega:\omega\in\Sph}.
	\]
	After rotation about the $e_3$-axis, the branch is normalised by the
	$(\ell,m)=(2,1)$ direction $sXZ$. For $\omega=(X,Y,Z)$,
	\begin{equation}
		\label{eq:endpoint-h-expansion}
		\begin{aligned}
			h(s,V)
			={} &
			-\frac{3V^2}{16}\frac{3Z^2-1}{2}
			+sXZ
			-\frac9{10}\frac{V}{2\sqrt{3}}sY \\
			    & \quad
			-\frac{5\sqrt{3}}{7}Vs\,
			Y\left(Z^2-\frac15\right)
			+O_{\H^{k+2}}\bigl(s^2+\abs{s}V^2+V^4\bigr).
		\end{aligned}
	\end{equation}
	The angular speed satisfies
	\begin{equation}
		\label{eq:endpoint-alpha-expansion}
		\alpha(s,V)
		=
		2\sqrt{3}
		-
		\frac{2601}{1960}\frac{V^2}{2\sqrt{3}}
		+O(s^2+V^3).
	\end{equation}
	At $s=0$, this value of $\alpha$ is only the continuous selection inherited
	from the non-axisymmetric branch; the same axisymmetric surface $h(0,V)=h^{\rm ax}_{0,V}$ solves the
	physical problem for every nearby $\alpha$.
	If $0<\abs{s}\le\eta V$,
	then
	\begin{equation}
		\label{eq:endpoint-centroid-expansion}
		\perpProj\left(
		\frac1{\abs{\Omega(s,V)}}\int_{\Omega(s,V)}y\,\dd y
		\right)
		=
		-\frac9{10}\frac{V}{2\sqrt{3}}s\,e_2
		+O(s^2+\abs{s}V^2),
	\end{equation}
	and this horizontal centroid is non-zero after decreasing $V_0$ and $\eta$.
	For each fixed $V>0$, the non-axisymmetric branch is locally unique
	modulo rotations about the $e_3$-axis in an endpoint-size neighbourhood of
	the axisymmetric branch. On the axisymmetric set, the shape is locally
	unique, while $\alpha$ remains free. The precise uniqueness statement is
	Lemma~\ref{lem:endpoint-local-uniqueness}.
\end{theorem}

For the proof, let $\proj_j$ be the fixed $\L^2(\Sph)$-orthogonal projection onto
$K_j$, $j\in\{1,6\}$, and set
\[
	\proj_{1,6}:=\proj_1+\proj_6,
	\qquad
	\idproj_{1,6}:=\id-\proj_{1,6}.
\]
We also abbreviate
\begin{align*}
	K_{16}:=K_1\oplus K_6.
\end{align*}

\begin{lemma}\label{lem:endpoint-complement-isomorphism-simple}
	At $\tau=0$ and $\alpha_*^2=12$, the round linearisation restricts to an
	isomorphism
	\[
		D_\xi\cF(0,\alpha_*,0):
		K_{16}^{\perp}\cap\H^{k+2}_{\rm lin}
		\longrightarrow
		K_{16}^{\perp}\cap\H^k_{\rm lin}.
	\]
\end{lemma}
\begin{proof}
	By Lemmas~\ref{lem:linearisation-round-sphere} and
	\ref{lem:nonresonance}, the only non-axisymmetric zero modes in the fixed
	linear slice are $K_1$ and $K_6$. On a spherical harmonic $Y_\ell^m$ the
	remaining eigenvalue is
	\[
		\lambda_{\ell m}(\alpha_*)
		=
		\ell(\ell+1)-2-\frac{12m^2}{\ell+1}.
	\]
	Since $\abs m\le\ell$, this has magnitude comparable to
	$1+\ell(\ell+1)$ for all sufficiently large $\ell$. The finitely many
	remaining non-kernel modes have a positive minimum after division by
	$1+\ell(\ell+1)$. Hence
	\[
		\abs{\lambda_{\ell m}(\alpha_*)}
		\ge c\bigl(1+\ell(\ell+1)\bigr)
	\]
	for all modes $Y^m_\ell\in K_{16}^\perp$, which gives the
	isomorphism by spherical-harmonic expansion.
\end{proof}

We use the same chart $\Theta$, functional $\scrJ$ and gradient $\cF$ as in the previous analysis. Recalling \eqref{eq:derivative-scrJ}, there in particular holds
\begin{align}\label{eq:recalled_relation_F_B}
	\int_{\Sph}\cF(\xi,\alpha,V)\eta\dd S=D_\xi\scrJ(\xi,\alpha,V)[\eta]=\int_{\Sph}\cB(\Theta(\xi),\alpha,V)M_{\Theta(\xi)}D\Theta(\xi)\eta\dd S.
\end{align}
Consequently, $\cF(\xi,\alpha,V)=0$, $\xi\in\H^{k+2}_{\rm lin}$ small if and only if $\cB(h,\alpha,V)=0$, $h=\Theta(\xi)\in\cM^{k+2}$.
The first implication follows from
Lemma~\ref{lem:variational_structure}: if $\cF(\xi,\alpha,V)=0$, then
$h=\Theta(\xi)$ is a constrained critical point of $\scrE$. The reverse implication is clear.

Performing a Lyapunov-Schmidt reduction as in the proof of Lemma \ref{lem:nonlinear-branch-smooth-main-proof}, cf. Remark \ref{rem:LS_higher_dimensional_kernel}, using Lemma~\ref{lem:endpoint-complement-isomorphism-simple} for the range isomorphism, we find for any $a\in K_{16}$ near $0$, $\alpha$ near $\alpha_*$, $V$ near $0$, a unique $w(a,\alpha,V)\in K_{16}^\perp$ near $0$, such that
\begin{align}\label{eq:range_equation_endpoint_case}
	{\idproj_{1,6}}\cF(a+w(a,\alpha,V),\alpha,V)=0,
\end{align}
where ${\idproj_{1,6}}$ denotes the $\L^2(\Sph)$-orthogonal projection $\H^{k}_{\rm lin}\rightarrow K_{16}^\perp$. Denoting ${\proj_{1,6}}:=\id-{\idproj_{1,6}}$, it therefore remains to find $a\in K_{16}$ small such that the kernel equation
\begin{align}\label{eq:endpoint_kernel_equation}
	\cG(a,\alpha,V):={\proj_{1,6}}\cF(a+w(a,\alpha,V),\alpha,V)=0
\end{align}
is satisfied for $V\neq 0$ small and suitable $\alpha$ near $\alpha_*$.

We first identify the zero-amplitude range solution with the universal
axisymmetric branch of Lemma~\ref{lem:axisymmetric-translating-branch}. This
will in particular show that $a=0$ solves
\eqref{eq:endpoint_kernel_equation} for every admissible $(\alpha,V)$.
As indicated above the key idea is then to investigate the linearisation along this translating branch instead of the round sphere solution. Indeed, the linearisation of \eqref{eq:endpoint_kernel_equation} at $a=0$, $V=0$ is identically vanishing, but, as a careful expansion will show, admits a continuation of solutions with $a\neq 0$ for small $V\neq 0$.

\begin{lemma}
	\label{lem:endpoint_axisymmetric_branch_with_w}
	After shrinking the domain of $w$ so that $\abs V<V_{\rm ax}$ and that
	$\xi^{\rm ax}(V)$ lies in the range-uniqueness ball, let
	$(\alpha,V)$ be such that $w(0,\alpha,V)$ is well-defined. Then
	\begin{equation}
		\label{eq:endpoint-axisymmetric-range-identification}
		w(0,\alpha,V)=\xi^{\rm ax}(V),
		\qquad
		\Theta(w(0,\alpha,V))=h^{\rm ax}_{0,V},
	\end{equation}
	independently of $\alpha$, and $a=0$ solves
	\eqref{eq:endpoint_kernel_equation}. In particular, all the symmetry,
	regularity, uniqueness, and expansion conclusions of
	Lemma~\ref{lem:axisymmetric-translating-branch} apply, including
	\eqref{eq:axisymmetric-branch-expansion} and
	\eqref{eq:axisymmetric-coordinate-expansion}.
\end{lemma}
\begin{proof}
	At $\tau=0$, Lemma~\ref{lem:axisymmetric-translating-branch} and the
	equivalence following \eqref{eq:recalled_relation_F_B} give
	\begin{equation*}
		\cF(\xi^{\rm ax}(V),\alpha,V)=0.
	\end{equation*}
	The coordinate $\xi^{\rm ax}(V)$ is axisymmetric and hence belongs to
	$K_{16}^{\perp}$. It therefore solves the range equation
	\eqref{eq:range_equation_endpoint_case} with $a=0$. Uniqueness of the range
	correction gives the first identity in
	\eqref{eq:endpoint-axisymmetric-range-identification}; the second follows
	from \eqref{eq:axisymmetric-branch-definition}. Projecting the displayed
	full-gradient identity onto $K_{16}$ yields
	$\cG(0,\alpha,V)=0$.
\end{proof}

\subsection{Linearisation along the translating branch} We recall that $\so$ acts on $K_{16}$ via
\begin{align}
	(\cU_\theta a)(\omega)=a(\cos\theta X+\sin\theta Y,-\sin \theta X+\cos\theta Y,Z).
\end{align}
Based on the $\so$-gradient structure we deduce that
the linearisation of \eqref{eq:endpoint_kernel_equation} at the axisymmetric branch has the following structure.
\begin{lemma}\label{lem:derivative_G}
	In block matrix form, corresponding to $K_{16}=K_1\oplus K_6$, there holds
	\begin{align*}
		D_a\cG(0,\alpha,V)=\begin{pmatrix}
			                   \mu_1(\alpha,V)\id_{K_1} & 0                        \\
			                   0                        & \mu_6(\alpha,V)\id_{K_6}
		                   \end{pmatrix},
	\end{align*}
	where
	\begin{gather}\begin{gathered}\label{eq:formula_mu_m}
			\mu_m(\alpha,V)=\int_{\Sph}D_h\cB(h^{\rm ax}_{0,V},\alpha,V)\left[D\Theta(w(0,\alpha,V))b_m\right]M_{h^{\rm ax}_{0,V}}D\Theta(w(0,\alpha,V))a_m\dd S,\\
			b_m:=a_m+D_aw(0,\alpha,V)a_m,
		\end{gathered}
	\end{gather}
	for any $a_m\in K_m$ with $\norm{a_m}_{\L^2(\Sph)}=1$, $m=1,6$.
\end{lemma}
\begin{proof}
	We first of all recall that $\cG(\cdot,\alpha,V)$ is the gradient of the reduced functional $a\mapsto \scrJ(a+w(a,\alpha,V),\alpha,V)$, cf. Remark \ref{rem:LS_higher_dimensional_kernel}. Thus $D_a\cG(\cdot,\alpha,V)$ is its Hessian and therefore symmetric.

	Moreover, $\cG(\cdot,\alpha,V)$ is $\so$-equivariant. Differentiation at $0$ implies the same for $D_a\cG(0,\alpha,V)$. Since the action is orthogonal, it follows that
	\begin{align*}
		\langle D_a\cG(0,\alpha,V)a_1,a_6 \rangle_{\L^2(\Sph)} = \langle D_a\cG(0,\alpha,V)\cU_\theta a_1,\cU_\theta a_6 \rangle_{\L^2(\Sph)}
	\end{align*}
	for all $a_1\in K_1$, $a_6\in K_6$, $\theta\in\R/(2\pi\mathbb Z)$. Picking the class $\theta=\pi$, for which $\cU_\pi a_1=-a_1$ and $\cU_\pi a_6=a_6$, we deduce that the off-diagonal blocks must vanish.

	On the diagonal blocks we have that $\cU_{\frac{\pi}{2}}a_1=-\cU_{-\frac{\pi}{2}}a_1$ is $\L^2(\Sph)$-orthogonal to $a_1$. By symmetry and equivariance of $D_a\cG(0,\alpha,V)$ we deduce
	\begin{align*}
		\langle D_a\cG(0,\alpha,V)a_1,\cU_{\frac{\pi}{2}}a_1\rangle_{\L^2(\Sph)} & =-\langle \cU_{\frac{\pi}{2}}D_a\cG(0,\alpha,V)a_1,a_1\rangle_{\L^2(\Sph)} \\&=-\langle \cU_{\frac{\pi}{2}}a_1,D_a\cG(0,\alpha,V)a_1\rangle_{\L^2(\Sph)}.
	\end{align*}
	The argument applies on $K_6$ with $\theta =\frac{\pi}{12}$ instead of $\frac{\pi}{2}$. Hence the diagonal blocks are multiples of the identity.

	The precise characterisation of the eigenvalues then follows from
	\begin{align*}
		\mu_m(\alpha,V)=\langle D_a\cG(0,\alpha,V)a_m,a_m\rangle_{\L^2(\Sph)},\quad \norm{a_m}_{\L^2(\Sph)}=1
	\end{align*}
	and differentiation of \eqref{eq:endpoint_kernel_equation} and \eqref{eq:recalled_relation_F_B}, minding that $\cB(h^{\rm ax}_{0,V},\alpha,V)=0$.
\end{proof}

Our goal is to obtain the following expansion for the eigenvalues $\mu_m(\alpha,V)$.

\begin{lemma}\label{lem:expansion_mu} For each $C_0>0$ there is a constant $V_0>0$, such that in the regime
	\[
		0\le V\le V_0,\quad \abs{\alpha-\alpha_*}\le C_0V^2
	\]
	we have
	\begin{align}
		\label{eq:endpoint-mu1-expansion}
		\mu_1(\alpha,V)
		 & =
		\lambda_{21}(\alpha)-\frac{867}{980}V^2+O(V^3), \\
		\label{eq:endpoint-mu6-expansion}
		\mu_6(\alpha,V)
		 & =
		\lambda_{76}(\alpha)-\frac{4174947}{284240}V^2+O(V^3)
	\end{align}
	uniformly as $V\rightarrow 0$, i.e. the constants in the $O(V^3)$ terms may depend on $C_0$, but not on $(V,\alpha)$ in the stated regime. Here $\lambda_{21}(\alpha)$, $\lambda_{76}(\alpha)$ are the eigenvalues in \eqref{eq:lambda-lm} with $\tau=0$.
\end{lemma}

The restriction of $\alpha$ can be motivated from the non-resonant case, where \eqref{eq:alpha-expansion-dimensionless-theorem} holds true.
The proof of Lemma \ref{lem:expansion_mu} is carried out in the following four subsections. Section \ref{sec:expansion_of_the_chart} essentially allows us to ignore the derivatives of the chart $\Theta$, in Section \ref{sec:expansion_bernoulli_linearisation} we treat the derivative of the Bernoulli map $\cB$ at the axisymmetric branch, and Section \ref{sec:schur_complements} deals with $D_aw(0,\alpha,V)$. We conclude in Section \ref{sec:proof_lemma_expansion}.

\subsection{Expansion of the chart}\label{sec:expansion_of_the_chart}
The key observation concerning our chart $\Theta$ is the following.

\begin{lemma}\label{lem:chart_axisymmetric_xi_vs_m}
	Let $\xi \in \H^{k+2}_{\rm lin}$ be sufficiently small, axisymmetric and even in $Z$. Then for every spherical harmonic $Y_{\ell}^m$ with azimuthal weight $m\neq 0$ there holds
	\begin{align*}
		D\Theta(\xi)[Y_{\ell}^m]=Y_{\ell}^m.
	\end{align*}
\end{lemma}
\begin{proof}
	We recall that
	$
		\Theta(\xi)=\xi+c(\xi)+d(\xi)Z
	$
	is defined by the volume and vertical centring constraints $\cV(\Theta(\xi))=0$, $\cC_3(\Theta(\xi))=0$, and that it is $\so$-equivariant and commutes with reflection in $Z$. Differentiating the volume constraint in direction $\eta=Y_\ell^m$, we find
	\begin{align*}
		0 & =\int_{\Sph}(1+\Theta(\xi))^2\eta\dd S+Dc(\xi)[\eta]\int_{\Sph}(1+\Theta(\xi))^2\dd S+ Dd(\xi)[\eta]\int_{\Sph}(1+\Theta(\xi))^2Z\dd S.
	\end{align*}
	The first integral vanishes because $\Theta(\xi)$ is axisymmetric and $m\neq0$, while the last integral vanishes by parity in $Z$. Thus $Dc(\xi)[\eta]=0$. Differentiating the vertical-centring constraint similarly gives $Dd(\xi)[\eta]=0$.
\end{proof}

\begin{lemma}\label{lem:first_simplification_of_eigenvalues} The eigenvalues $\mu_m(\alpha,V)$, $m=1,6$ defined in \eqref{eq:formula_mu_m} satisfy
	\begin{align*}
		\mu_m(\alpha,V)=\int_{\Sph}D_h\cB(h^{\rm ax}_{0,V},\alpha,V)\left[a_m+D_aw(0,\alpha,V)a_m\right](1+2V^2h_{02})a_m\dd S+O(V^3)
	\end{align*}
	for any $a_m \in K_m$, $\norm{a_m}_{\L^2(\Sph)}=1$,
	uniformly in $\alpha$ in a neighbourhood of $\alpha_*$, as $V\rightarrow 0$.
\end{lemma}
\begin{proof}
	{By \eqref{eq:axisymmetric-branch-expansion} we have}
	\begin{align*}
		{M_{h^{\rm ax}_{0,V}}=(1+h^{\rm ax}_{0,V})^2
			=1+2V^2h_{02}+O(V^4)}\end{align*}
	independently of $\alpha$.
	Moreover, since \(m\neq0\), Lemmas
	\ref{lem:endpoint_axisymmetric_branch_with_w} and
	\ref{lem:chart_axisymmetric_xi_vs_m} give
	\[
		D\Theta(w(0,\alpha,V))[a_m]=a_m.
	\]
	Furthermore,
	\[
		w(0,\alpha,V)=V^2h_{02}+O(V^4),
	\]
	so the smoothness of \(\Theta\) and \(D\Theta(0)=\id\) imply
	\[
		D\Theta(w(0,\alpha,V))-\id=O(V^2).
	\]
	Lemma~\ref{lem:expansion_of_D_aw} gives
	\[
		D_aw(0,\alpha,V)a_m=O(V)
	\]
	uniformly in \(\alpha\). Consequently,
	\[
		D\Theta(w(0,\alpha,V))
		[D_aw(0,\alpha,V)a_m]
		=
		D_aw(0,\alpha,V)a_m+O(V^3).
	\]
	Therefore
	\[
		D\Theta(w(0,\alpha,V))[b_m]
		=
		a_m+D_aw(0,\alpha,V)a_m+O(V^3).
	\]
	Substitution in \eqref{eq:formula_mu_m}, together with
	\[
		M_{h^{\rm ax}_{0,V}}
		=
		1+2V^2h_{02}+O(V^4),
	\]
	proves the assertion.
\end{proof}

\subsection{Expansion of the Bernoulli linearisation}\label{sec:expansion_bernoulli_linearisation}

Let
\[
	L^{\cB}_{V,\alpha}:=D_h\cB(h^{\rm ax}_{0,V},\alpha,V)
	:\H^{k+2}(\Sph)\longrightarrow\H^k(\Sph).
\]
Define
\begin{align}
	\label{eq:endpoint-abstract-L0}
	L_0^\alpha\eta
	 & :=D_h\cB(0,\alpha,0)[\eta],                                           \\
	L_1^\alpha\eta
	 & :=\partial_VD_h\cB(0,\alpha,0)[\eta], \label{eq:endpoint-abstract-L1} \\
	L_2^\alpha\eta
	 & :=\frac12\partial_V^2D_h\cB(0,\alpha,0)[\eta]
	+D_h^2\cB(0,\alpha,0)[h_{02},\eta]. \label{eq:endpoint-abstract-L2}
\end{align}

\begin{lemma}
	\label{lem:endpoint-L-expansion}
	There holds
	\begin{equation}
		\label{eq:endpoint-L-expansion-prelemma}
		L^{\cB}_{V,\alpha}
		=
		L_0^\alpha+VL_1^\alpha+V^2L_2^\alpha
		+O_{\cL(\H^{k+2},\H^k)}(V^3),
	\end{equation}
	uniform in $\alpha$
	in a neighbourhood of $\alpha_*$.
	In addition,
	for every $C_0>0$, uniformly in the regime
	$\abs{\alpha-\alpha_*}\le C_0V^2$, we have
	\begin{equation}
		\label{eq:endpoint-L-expansion-frozen}
		L^{\cB}_{V,\alpha}
		=
		L_0^\alpha+VL_1^{\alpha_*}+V^2L_2^{\alpha_*}
		+O_{\cL(\H^{k+2},\H^k)}(V^3).
	\end{equation}
\end{lemma}

\begin{proof}
	By Proposition~\ref{prop:smooth-bernoulli-map}, $\cB$ is $\C^\infty$ as a map
	$U^{k+2}_\delta\times\R^2\to\H^k(\Sph)$. Since
	$V\mapsto h^{\rm ax}_{0,V}$ is smooth and independent of $\alpha$, the operator
	\[
		L^{\cB}_{V,\alpha}=D_h\cB(h^{\rm ax}_{0,V},\alpha,V)
		:\H^{k+2}(\Sph)\longrightarrow \H^k(\Sph)
	\]
	depends
	smoothly on $V$ and $\alpha$. Taylor expansion in $V$ at the round sphere gives \eqref{eq:endpoint-L-expansion-prelemma}.
	Because the axisymmetric profile only starts at $V^2$, the three coefficients have the form \eqref{eq:endpoint-abstract-L0}--\eqref{eq:endpoint-abstract-L2}.
	The frozen expansion \eqref{eq:endpoint-L-expansion-frozen} follows from the smooth dependence of $L_1^\alpha$, $L_2^\alpha$ on $\alpha$, and $\abs{\alpha-\alpha_*}\le C_0V^2$.
\end{proof}

For the
exact computation of the Taylor coefficients, we use the notation and coefficient-extraction
convention of Section~\ref{sec:shape-calculus}. Moreover, it is more convenient to pass to complex spherical harmonics. Thus $Y_\ell^m$ denotes the
complex $\L^2(\Sph)$-normalised spherical harmonic satisfying
$\partial_\varphi Y_\ell^m=imY_\ell^m$.
We define
\[
	\avg{f,g}_{\CC}=\int_{\Sph}\overline f\,g\,\dd S.
\]
The coefficients
$a_{\ell,\pm}^m$ are those in \eqref{eq:a-coefficients}, and the recursions
\eqref{eq:Z-recursion}--\eqref{eq:gradZ-recursion} are used throughout.

If $f=\sum_{n\ge0} f_n$ is its spherical-harmonic decomposition by degree,
we write
\[
	\frP_{\ell} f:=f_\ell
\]
for the degree-$\ell$ projection.

The
order-zero term $L_0^\alpha$ is the endpoint round-sphere
linearisation, already computed in Lemma~\ref{lem:linearisation-round-sphere}.
We record the endpoint form for later use.

\begin{lemma}
	\label{lem:endpoint-L0}
	Let $\tau=0$. For every $\alpha>0$ and every spherical harmonic
	$Y_\ell^m$, $\ell\ge1$, one has
	\begin{equation}
		L^\alpha_0Y_\ell^m
		=
		\lambda_{\ell m}(\alpha)Y_\ell^m,
	\end{equation}
	where
	\begin{equation}
		\lambda_{\ell m}(\alpha)
		=
		\ell(\ell+1)-2
		-
		\alpha^2\frac{m^2}{\ell+1}.
	\end{equation}
\end{lemma}

Next we turn to the first-order term $L_1^{\alpha_*}$.

\begin{lemma}
	\label{lem:endpoint-L1}
	We have
	\[
		L_1^{\alpha_*}Y_\ell^m
		=
		b_{\ell+1,\ell}^{(m)}Y_{\ell+1}^m
		+
		b_{\ell-1,\ell}^{(m)}Y_{\ell-1}^m
	\]
	with
	\begin{equation}
		\label{eq:endpoint-b-coefficients}
		b_{\ell+1,\ell}^{(m)}
		=
		-\frac{3i\alpha_*m}{2}
		\frac{2\ell+1}{\ell+1}a_{\ell,+}^m,
		\qquad
		b_{\ell-1,\ell}^{(m)}
		=
		\frac{3i\alpha_*m}{2}
		\frac{2\ell-1}{\ell}a_{\ell,-}^m,
	\end{equation}
	with the lower term omitted if $a_{\ell,-}^m=0$.
	In particular $L_1^{\alpha_*}$ has no diagonal
	part in the degree $\ell$.
\end{lemma}

\begin{proof}
	By \eqref{eq:endpoint-abstract-L1}, $L_1^{\alpha_*}\eta$ is the coefficient of $Vt$ in an expansion of
	$\cB(t\eta,\alpha_*,V)$ at the round sphere $(t,V)=(0,0)$.
	Since $\tau=0$, and since the curvature has no
	explicit dependence on $V$, this coefficient comes only from the exterior
	Bernoulli term
	\[
		F_{\alpha_*,V}(y,q)
		=
		\frac12\abs{q-Ve_3}^2+
		\alpha_*y\cdot(e_3\times q),
	\]
	which has to be considered at $y=\Psi_{t\eta}(\omega)$ and $q=\qout_{t,V}$. The latter is precisely the pulled-back exterior gradient investigated in Lemma \ref{lem:first-order-neumann-expansion}.

	Regarding the projection $\Pi_{t\eta}$, abbreviating $F_{\alpha_*,V}:=F_{\alpha_*,V}(\Psi_{t\eta},\qout_{t,V})$, we observe that
	\begin{align*}
		\left[\cB(t\eta,\alpha_*,V)\right]_{tV} & =-\left[\Pi_{t\eta}\left(F_{\alpha_*,V}\right)\right]_{tV}                                                                                             \\
		                                        & =-\left[F_{\alpha_*,V}-\frac{\int_{\Sph}M_{t\eta}\left(V[F_{\alpha_*,V}]_V+tV[F_{\alpha_*,V}]_{tV}\right)\dd S}{\int_{\Sph}M_{t\eta}\dd S}\right]_{tV} \\
		                                        & =-\left([F_{\alpha_*,V}]_{tV}-\frac{1}{4\pi}\int_{\Sph}[F_{\alpha_*,V}]_{tV} \dd S\right),
	\end{align*}
	due to $[F_{\alpha_*,V}]_V=0$ by \eqref{eq:round-translational-F}.
	Regarding the $tV$-coefficient, the projection thus subtracts only the average.

	Hence it
	will not change the coefficient for non-axisymmetric spherical harmonics, $m\neq 0$,
	below.
	And if $m=0$, all
	terms below will anyhow vanish because $\partial_\varphi Y_\ell^0=0$.

	Let now $\eta=Y_\ell^m$. We use the first-order boundary-gradient expansion
	\eqref{eq:boundary-gradient-expansion-outer}. It separates the exterior
	gradient into three pieces relevant at order $Vt$:
	\[
		\qout_{t,V}
		=
		\qout_V
		+t\qout_{{\rm rot},\eta}
		+t\left[
			V\eta\left(-3Z\omega+\frac32\nablaS Z\right)
			+\nabla\uout_{V,\eta}\big|_{r=1}
			\right]
		+O(t^2).
	\]
	The round translational field is given by
	\eqref{eq:round-translational-gradients}:
	\begin{equation}
		\label{eq:endpoint-L1-qoutV-recall}
		\qout_V-Ve_3=-\frac{3V}{2}\nablaS Z.
	\end{equation}
	Since $\uout_{V,\eta}$ has also order $V$, cf. \eqref{eq:exterior-translation-shape-datum},
	the kinetic part of $[F_{\alpha_*,V}]_{Vt}$ pairs
	\eqref{eq:endpoint-L1-qoutV-recall} only with the rotational gradient
	$\qout_{{\rm rot},\eta}$.

	The rotational Neumann datum, cf. \eqref{eq:rotational-potentials-def}, is
	\[
		g_{{\rm rot},\eta}=-\alpha_*\partial_\varphi Y_\ell^m,
	\]
	and \eqref{eq:rotational-boundary-gradient-out} gives
	\[
		\qout_{{\rm rot},\eta}
		=
		g_{{\rm rot},\eta}\omega
		-
		\frac1{\ell+1}\nablaS g_{{\rm rot},\eta}.
	\]
	Hence the kinetic $Vt$ coefficient for $\eta=Y_\ell^m$ is
	\begin{align}
		\label{eq:endpoint-L1-kinetic-piece}
		\cK_1Y_\ell^m
		 & :=
		\left(-\frac32\nablaS Z\right)
		\cdot\qout_{{\rm rot},Y_\ell^m}                      \\
		 & =
		-\frac{3}{2(\ell+1)}\nablaS Z\cdot
		\nablaS(\alpha_*\partial_\varphi Y_\ell^m) \nonumber \\
		 & =
		-\frac{3\alpha_*}{2(\ell+1)}
		\partial_\varphi(\nablaS Z\cdot\nablaS Y_\ell^m). \nonumber
	\end{align}
	In the last line we used $\partial_\varphi Z=0$, cf.~\eqref{eq:dphi-def}.
	We expand this, using
	\eqref{eq:gradZ-recursion}:
	\[
		\nablaS Z\cdot\nablaS Y_\ell^m
		=
		-\ell a_{\ell,+}^mY_{\ell+1}^m
		+(\ell+1)a_{\ell,-}^mY_{\ell-1}^m.
	\]
	Substitution into \eqref{eq:endpoint-L1-kinetic-piece} gives
	\begin{equation}
		\label{eq:endpoint-L1-kinetic-expanded}
		\cK_1Y_\ell^m
		=
		\frac{3\alpha_*}{2}
		\left[
			\frac{\ell}{\ell+1}a_{\ell,+}^m\partial_\varphi Y_{\ell+1}^m
			-
			a_{\ell,-}^m\partial_\varphi Y_{\ell-1}^m
			\right].
	\end{equation}

	The rotational part $\alpha_*y\cdot(e_3\times q)$ has two possible $Vt$
	contributions. The displacement of the evaluation point gives
	\[
		\alpha_*\eta\,\omega\cdot\left(e_3\times\frac{\qout_V}{V}\right)
		=
		-\alpha_*\eta\,(e_3\times\omega)\cdot\frac{\qout_V}{V}=0,
	\]
	because $\qout_V$ is axisymmetric by
	\eqref{eq:round-translational-gradients}.

	The remaining $Vt$ contribution therefore is
	\begin{align}
		\cR_1Y_\ell^m
		 & :=
		-\alpha_*(e_3\times\omega)\cdot\left[
			                               \eta\left(-3Z\omega+\frac32\nablaS Z\right)
			                               +\frac{1}{V}\nabla\uout_{V,\eta}\big|_{r=1}
			                               \right] \\
		 & =
		-\alpha_*(e_3\times\omega)\cdot\frac1V\nabla\uout_{V,Y_\ell^m}\Big|_{r=1}. \nonumber
	\end{align}
	By \eqref{eq:general-mode-translation-tangential-gradient},
	\begin{equation}
		\left(\frac1V\nabla\uout_{V,Y_\ell^m}\right)_{\tan}\Big|_{r=1}
		=
		-\frac32a_{\ell,+}^m\nablaS Y_{\ell+1}^m
		+
		\frac32\frac{\ell-1}{\ell}a_{\ell,-}^m\nablaS Y_{\ell-1}^m.
	\end{equation}
	Since $(e_3\times\omega)\cdot\nablaS f=\partial_\varphi f$, this gives
	\begin{equation}
		\label{eq:endpoint-L1-rotational-piece}
		\cR_1Y_\ell^m
		=
		\frac{3\alpha_*}{2}
		\left[
			a_{\ell,+}^m\partial_\varphi Y_{\ell+1}^m
			-
			\frac{\ell-1}{\ell}a_{\ell,-}^m\partial_\varphi Y_{\ell-1}^m
			\right].
	\end{equation}

	Adding \eqref{eq:endpoint-L1-kinetic-expanded} and
	\eqref{eq:endpoint-L1-rotational-piece} gives the exterior $Vt$ coefficient
	\[
		(\cK_1+\cR_1)Y_\ell^m
		=
		\frac{3\alpha_*}{2}
		\left[
			\frac{2\ell+1}{\ell+1}a_{\ell,+}^m\partial_\varphi Y_{\ell+1}^m
			-
			\frac{2\ell-1}{\ell}a_{\ell,-}^m\partial_\varphi Y_{\ell-1}^m
			\right].
	\]
	Finally, the Bernoulli jump is interior minus exterior, and the interior part
	is absent at $\tau=0$. Therefore
	\begin{equation}
		L_1^{\alpha_*}Y_\ell^m
		=
		-\frac{3\alpha_*}{2}
		\left[
			\frac{2\ell+1}{\ell+1}a_{\ell,+}^m\partial_\varphi Y_{\ell+1}^m
			-
			\frac{2\ell-1}{\ell}a_{\ell,-}^m\partial_\varphi Y_{\ell-1}^m
			\right].
	\end{equation}
	Using $\partial_\varphi Y_\ell^m=imY_\ell^m$ gives \eqref{eq:endpoint-b-coefficients}.
\end{proof}

Concerning the second-order term $L_2^{\alpha_*}$ we will only require its action on the diagonal.

\begin{lemma}
	\label{lem:endpoint-L2-diagonal}
	Let $m\neq0$ and $\ell\ge\abs m$. For the complex $\L^2(\Sph)$-normalised spherical
	harmonics we set
	\[
		d_{\ell m}:=\avg{Y_\ell^m,L_2^{\alpha_*}Y_\ell^m}_{\CC}.
	\]
	Then
	\begin{equation}
		\label{eq:endpoint-dlm-formula}
		\begin{aligned}
			d_{\ell m}
			={} &
			\Bigg\{
			\frac38\left(\ell(\ell+1)+4\right)I_{\ell m} \\
			    & \quad+
			\frac94
			\left[
				1
				-
				(\ell+3)(a_{\ell,+}^m)^2
				-
				\left(
				1+\frac{(\ell-1)^2}{\ell}
				\right)(a_{\ell,-}^m)^2
				\right]                    \\
			    & \quad+
			\frac94
			\frac{m^2(\ell-2)}{(\ell+1)^2}I_{\ell m}
			\Bigg\},
		\end{aligned}
	\end{equation}
	where
	\begin{equation}
		\label{eq:endpoint-I-lm}
		I_{\ell m}
		:=
		\int_{\Sph}P_2\abs{Y_\ell^m}^2\,\dd S
		=
		\frac{3\left((a_{\ell,+}^m)^2+(a_{\ell,-}^m)^2\right)-1}{2}.
	\end{equation}
\end{lemma}

\begin{proof}
	The proof has four
	parts. First we
	utilise an
	affine two-parameter family $h_{\eps,t}=\eps f+t\eta$
	to characterise $d_{\ell,m}$ and to split it into three parts
	\[
		d_{\ell m}=d^T_{\ell m}+d^H_{\ell m}+d^R_{\ell m},
	\]
	that will be identified as a translational, curvature, and rotational contribution.
	The remaining three parts compute the
	contributions separately.

	\textbf{Part 1.}
	Let $\eta=Y_\ell^m$. From \eqref{eq:endpoint-abstract-L2} at $\alpha=\alpha_*$,
	\begin{equation}
		\label{eq:endpoint-dlm-split-proof}
		d_{\ell m}
		=
		\frac12\avg{\eta
			,
			\partial_V^2D_h\cB(0,\alpha_*,0)[\eta]}_{\CC}
		+
		\avg{\eta
			,
			D_h^2\cB(0,\alpha_*,0)[h_{02},\eta]}_{\CC}.
	\end{equation}

	We let
	\[
		f:=h_{02}=-\frac{3}{16}P_2(Z)
	\]
	and introduce an auxiliary parameter $\eps$, independent of $V$, by setting
	\[
		h_{\eps,t}:=\eps f+t\eta,
		\qquad
		\Psi_{\eps,t}:=\Psi_{h_{\eps,t}}.
	\]
	For any smooth expression $(\varepsilon,t,V)\mapsto G(\eps,t,V)$, our coefficient convention gives,
	after the substitution $\eps=V^2$,
	\[
		[(t,V)\mapsto G(V^2,t,V)]_{V^2t}=[G]_{V^2t}+[G]_{\eps t}.
	\]
	Taylor expansion and application to $(\varepsilon,t,V)\mapsto \cB(h_{\varepsilon,t},\alpha_*,V)$ show that
	\begin{align}\label{eq:formula_d_ell_m_in_terms_of_cB}
		d_{\ell m} & =\avg{\eta,\left[\cB(h_{V^2,t},\alpha_*,V)\right]_{V^2t}}_{\CC}                                                                                        \\
		           & =\avg{\eta,\left[\cB(t\eta,\alpha_*,V)\right]_{V^2t}}_{\CC}+\avg{\eta,\left[\cB(h_{\varepsilon,t},\alpha_*,0)\right]_{\varepsilon t}}_{\CC}. \nonumber
	\end{align}
	Thus the coefficient in \eqref{eq:endpoint-dlm-split-proof} is the sum of the
	pure speed coefficient $[\cdot]_{V^2t}$ at the round sphere and the
	background-shape coefficient $[\cdot]_{\eps t}$ at $V=0$. Terms of type
	$\eps Vt$, $\eps V^2t$, or $\eps^2t$ become, respectively,
	$V^3t$, $V^4t$, or $V^4t$ after this substitution and do not contribute to $L_2^{\alpha_*}$.

	The projection in the Bernoulli map
	does not affect the diagonal
	coefficient. Indeed, for every $h$ and every scalar function $E$,
	\[
		\avg{Y_\ell^m,\Pi_hE}_{\CC}
		=
		\avg{Y_\ell^m,E}_{\CC}
		-
		\avg{E}_{M_h}\int_{\Sph}\overline{Y_\ell^m}\,\dd S
		=
		\avg{Y_\ell^m,E}_{\CC},
	\]
	because $m\ne0$. Hence the $h$-dependence of $\Pi_h$ contributes only constants,
	which are orthogonal to $Y_\ell^m$
	, and
	we may therefore work with the unprojected pulled-back Bernoulli density $B_{h_{\varepsilon,t}}\circ\Psi_{\varepsilon,t}$.

	Let
	$Q_{\eps,t,V}^{\rm out}$ denote the exterior boundary gradient pulled back
	to the reference sphere,
	\[
		Q_{\eps,t,V}^{\rm out}
		:=
		\nabla\phiout_{h_{\eps,t},\alpha_*,V}
		\circ\Psi_{\eps,t}.
	\]
	Write
	\begin{align*}
		\scrK_{\eps,t,V}
		 & :=
		\frac12\abs{Q_{\eps,t,V}^{\rm out}-Ve_3}^2,    \\
		\scrR_{\eps,t,V}
		 & :=
		\alpha_*\Psi_{\eps,t}\cdot
		\left(e_3\times Q_{\eps,t,V}^{\rm out}\right), \\
		\scrH_{\eps,t}
		 & :=
		H_{h_{\eps,t}}\circ\Psi_{\eps,t}.
	\end{align*}
	Since $\tau=0$,
	\[
		B_{h_{\eps,t}}\circ\Psi_{\eps,t}
		=
		-\scrK_{\eps,t,V}
		-\scrR_{\eps,t,V}
		+\scrH_{\eps,t}.
	\]
	Consequently,
	\begin{align}
		\label{eq:endpoint-L2-bookkeeping-all-terms}
		d_{\ell m}
		 & =
		\avg{Y_\ell^m,
				-[\scrK]_{V^2t}-[\scrK]_{\eps t}
				-[\scrR]_{V^2t}-[\scrR]_{\eps t}
				+[\scrH]_{\eps t}}_{\CC}           \\
		 & =:d_{\ell m}^T+d_{\ell m}^R+d_{\ell m}^H
		\nonumber
	\end{align}
	where $[\scrH]_{V^2t}=0$ because the curvature has no explicit
	$V$-dependence.
	As in \eqref{eq:formula_d_ell_m_in_terms_of_cB} the $[\cdot]_{V^2t}$-coefficients can be considered with $\varepsilon=0$, while the $[\cdot]_{\varepsilon t}$-coefficients can be considered with $V=0$.
	We now identify each term in
	\eqref{eq:endpoint-L2-bookkeeping-all-terms}.

	\textbf{Part 2.}
	First consider the kinetic term. For a fixed shape the Neumann problem is
	linear in the vector field $W_{\alpha_*,V}=\alpha_*e_3\times y+Ve_3$. Hence
	\[
		Q_{\eps,t,V}^{\rm out}
		=
		Q_{\eps,t}^{\rm rot}+VQ_{\eps,t}^{\rm tr},
	\]
	where $Q^{\rm rot}$ is the solution for $V=0$ and $Q^{\rm tr}$ is the solution
	for the unit translational field, i.e. $\alpha=0$, $V=1$. If $t=0$, the graph $r=1+\eps f$ is
	axisymmetric. The exact normal-component formula
	\eqref{eq:W-normal-exact}, with $V=0$ and $\partial_\varphi f=0$, gives
	\[
		\bigl(W_{\alpha_*,0}\cdot n_{\eps f}\bigr)
		\circ\Psi_{\eps f}=0.
	\]
	Thus $Q_{\eps,0}^{\rm rot}=0$ and
	$Q_{\eps,t}^{\rm rot}=O(t)$. At $V=0$ we obtain
	\begin{equation}
		\scrK_{\eps,t,0}
		=
		\frac12\abs{Q_{\eps,t}^{\rm rot}}^2
		=O(t^2),
	\end{equation}
	so
	\begin{equation}
		[\scrK]_{\eps t}=0.
	\end{equation}
	At $\eps=0$, \eqref{eq:boundary-gradient-expansion-outer} gives
	\[
		Q_{0,t,V}^{\rm out}-Ve_3
		=
		-\frac{3V}{2}\nabla_{\Sph}Z
		+
		t\qout_{{\rm rot},\eta}
		+
		tV\left[
			\eta\left(-3Z\omega+\frac32\nabla_{\Sph}Z\right)
			+
			\frac1V\nabla\uout_{V,\eta}\big|_{r=1}
			\right]
		+
		O(t^2).
	\]
	For the $V^2t$-coefficient of the corresponding kinetic term, we only have the pairing of $-\frac{3V}{2}\nabla_{\Sph}Z$ with
	the tangential part of the $tV$-coefficient
	, namely
	\[
		\frac32\eta\nabla_{\Sph}Z
		+
		\left(\frac1V\nabla\uout_{V,\eta}\right)_{\tan}\Big|_{r=1}
	\]
	as relevant contribution.
	With $\eta=Y^m_\ell$, define
	\[
		\Upsilon_{\ell m}
		:=
		\left(\frac1V\nabla\uout_{V,Y_\ell^m}\right)_{\tan}\Big|_{r=1}.
	\]
	We therefore have
	\begin{align*}
		[\scrK]_{V^2t}
		 & =
		\left(-\frac32\nablaS Z\right)\cdot
		\left(
		\frac32Y_\ell^m\nablaS Z+\Upsilon_{\ell m}
		\right) \\
		 & =
		-\frac94Y_\ell^m\abs{\nablaS Z}^2
		-
		\frac32\nablaS Z\cdot\Upsilon_{\ell m}.
	\end{align*}
	Since the Bernoulli density contains the exterior kinetic term with a minus
	sign, its contribution to the diagonal coefficient is
	\begin{equation}
		\label{eq:endpoint-kinetic-survives-as-T}
		d^T_{\ell m}
		=
		-\avg{Y_\ell^m,[\scrK]_{V^2t}}_{\CC}
		=
		\avg{Y_\ell^m,
			\frac94Y_\ell^m\abs{\nablaS Z}^2
			+
			\frac32\nablaS Z\cdot\Upsilon_{\ell m}}_{\CC}.
	\end{equation}

	We now evaluate the two terms in
	\eqref{eq:endpoint-kinetic-survives-as-T}. The first one is
	\begin{equation}
		\label{eq:endpoint-gradZ2-diagonal}
		\avg{Y_\ell^m,
			Y_\ell^m\abs{\nablaS Z}^2}_{\CC}
		=
		1-(a_{\ell,+}^m)^2-(a_{\ell,-}^m)^2,
	\end{equation}
	because $\abs{\nablaS Z}^2=1-Z^2$, cf. \eqref{eq:low-mode-products}, and
	\eqref{eq:Z-recursion} gives the displayed diagonal coefficient of
	$Z^2$.
	For the second factor, we use
	\eqref{eq:general-mode-translation-tangential-gradient}:
	\begin{equation}
		\Upsilon_{\ell m}
		=
		-\frac32a_{\ell,+}^m\nablaS Y_{\ell+1}^m
		+
		\frac32\frac{\ell-1}{\ell}a_{\ell,-}^m\nablaS Y_{\ell-1}^m.
	\end{equation}
	Applying \eqref{eq:gradZ-recursion} to $Y_{\ell+1}^m$ and
	$Y_{\ell-1}^m$, and observing that $a^m_{\ell\pm 1,\mp}=a^m_{\ell,\pm}$ gives the two degree-$\ell$ coefficients
	\[
		\frP_{\ell}(\nablaS Z\cdot\nablaS Y_{\ell+1}^m)
		=
		(\ell+2)a_{\ell,+}^mY_\ell^m,
		\qquad
		\frP_{\ell}(\nablaS Z\cdot\nablaS Y_{\ell-1}^m)
		=
		-(\ell-1)a_{\ell,-}^mY_\ell^m.
	\]
	Therefore
	\begin{equation}
		\label{eq:endpoint-gradZ-tau-diagonal}
		\avg{Y_\ell^m,
			\nablaS Z\cdot\Upsilon_{\ell m}}_{\CC}
		=
		-\frac32
		\left[
			(\ell+2)(a_{\ell,+}^m)^2
			+
			\frac{(\ell-1)^2}{\ell}(a_{\ell,-}^m)^2
			\right].
	\end{equation}
	Substituting \eqref{eq:endpoint-gradZ2-diagonal} and
	\eqref{eq:endpoint-gradZ-tau-diagonal} into
	\eqref{eq:endpoint-kinetic-survives-as-T} gives
	\begin{equation}
		\label{eq:endpoint-dT}
		d_{\ell m}^{T}
		=
		\frac94
		\left[
			1
			-
			(\ell+3)(a_{\ell,+}^m)^2
			-
			\left(
			1+\frac{(\ell-1)^2}{\ell}
			\right)(a_{\ell,-}^m)^2
			\right].
	\end{equation}

	\textbf{Part 3.}
	Next consider the rotational part. At $\eps=0$ the expression
	$\scrR_{0,t,V}$ is linear in $Q_{0,t,V}^{\rm out}$ and has no other
	explicit $V$-dependence. Since $Q_{0,t,V}^{\rm out}$ is affine in $V$ for a
	fixed graph, $\scrR_{0,t,V}$ is affine in $V$. Hence
	\begin{equation}
		[\scrR]_{V^2t}=0.
	\end{equation}

	It remains to compute $-[\scrR]_{\eps t}$ at $V=0$. We first of all observe that $Q^{\rm out}_{\varepsilon,0,0}$ is vanishing. Indeed, $f$ and thus the induced domains are axisymmetric. Consequently $(W_{\alpha_*,0}\cdot n_{\varepsilon,0})\circ \Psi_{\varepsilon,0}=0$ on $\Sph$, which implies $\phiout_{h_{\varepsilon,0},\alpha_*,0}=0$.
	Thus we have
	\begin{align*}
		-[\scrR]_{\eps t} & =-\alpha_*\left[\Psi_{\eps,t}\cdot\left(e_3\times Q^{\rm out}_{\eps,t,0}\right)\right]_{\varepsilon t}                                                                                                              \\
		                  & =-\alpha_*\left[(1+\varepsilon f)\omega\cdot\left(e_3\times Q^{\rm out}_{\eps,t,0}\right)\right]_{\varepsilon t}-\alpha_*\eta\,\omega\cdot\left(e_3\times \left[Q^{\rm out}_{\eps,0,0}\right]_{\varepsilon }\right) \\
		                  & =\alpha_*\left[(1+\varepsilon f)(e_3\times\omega)\cdot \left[Q^{\rm out}_{\eps,t,0}\right]_t\right]_{\varepsilon }.
	\end{align*}
	Now, since $\phiout_{h_{\varepsilon,0},\alpha_*,0}=0$, the $t$-coefficient of $Q^{\rm out}_{\varepsilon,t,0}$ has only one contribution
	\begin{align*}
		\left[Q^{\rm out}_{\eps,t,0}\right]_t & =\nabla\left(\left.\frac{\ddd}{\ddd t}\right|_{t=0}\phiout_{h_{\varepsilon,t},\alpha_*,0}\right)((1+\varepsilon f)\omega).
	\end{align*}
	Differentiating the equation defining $\phiout_{h_{\varepsilon,t},\alpha_*,0}$ with respect to $t$ at $t=0$, one sees that $\left.\frac{\ddd}{\ddd t}\right|_{t=0}\phiout_{h_{\varepsilon,t},\alpha_*,0}$ is precisely the potential $\phiout_\varepsilon$ in Lemma \ref{lem:rotational-shape-axisymmetric-background}, i.e. the exterior rotational potential
	obtained by linearising in the $t\eta$ direction on the background
	$r=1+\eps f$.

	Next, observe that for every smooth potential $\phi$ defined near the surface induced by $\varepsilon f$, we have
	\begin{equation}
		(e_3\times y)\cdot\nabla\phi(y)
		=\partial_\varphi\left(\phi\bigl((1+\eps f(\omega))\omega\bigr)\right),
		\qquad
		y=(1+\eps f(\omega))\omega,
	\end{equation}
	because $f$ is axisymmetric.

	We therefore have
	\begin{equation}
		d_{\ell m}^R=\avg{Y_\ell^m,-\,[\scrR]_{\eps t}}_{\CC}
		=
		\alpha_*\avg{Y_{\ell}^m,\partial_\varphi
			\left[
				\phiout_{\eps}\big|_{r=1+\eps f}
				\right]_{\eps}}_{\CC}.
	\end{equation}

	The remaining $\varepsilon$-coefficient is computed in Lemma \ref{lem:rotational-shape-axisymmetric-background}, in particular in \eqref{eq:rotational-shape-axisymmetric-trace-outer}, by which
	\begin{equation}
		\label{eq:endpoint-rot-trace-shape-recall}
		\left[\eout_\eps\big|_{r=1+\eps f}\right]_{\eps}
		=
		fg-\sum_n\frac1{n+1}k_n^{\rm out}
	\end{equation}
	with
	\[
		f=h_{02}=-\frac{3}{16}P_2,
		\qquad
		g:=-\alpha_*\partial_\varphi Y_\ell^m,
	\]
	and
	Neumann datum $k^{\rm out}=\sum_n k_n^{\rm out}$ given by
	\begin{equation}
		\label{eq:endpoint-rot-k-datum-recall}
		k^{\rm out}
		=
		\left(
		(\ell+2)fg
		-
		\frac1{\ell+1}\nablaS f\cdot\nablaS g
		\right).
	\end{equation}
	Only the degree-$\ell$ part of \eqref{eq:endpoint-rot-trace-shape-recall}
	can contribute to the diagonal matrix element.
	By \eqref{eq:Ilm-def}
	\begin{equation}
		\frP_{\ell}(P_2Y_\ell^m)=I_{\ell m}Y_\ell^m,
		\qquad
		I_{\ell m}
		=
		\frac{3\left((a_{\ell,+}^m)^2+(a_{\ell,-}^m)^2\right)-1}{2},
	\end{equation} and thus
	\begin{equation}
		\label{eq:endpoint-fg-degree-l}
		\frP_{\ell}(fg)
		=
		-\frac{3}{16}I_{\ell m}g,
	\end{equation}
	Since $P_2$ is axisymmetric, the coefficient field $\nablaS P_2$ is
	axisymmetric. Hence differentiation in $\varphi$ commutes with $u\mapsto \nablaS P_2\cdot\nablaS u$. Using
	$g=-\alpha_*\partial_\varphi Y_\ell^m$, we therefore have
	\[
		\nablaS P_2\cdot\nablaS g
		=
		-\alpha_*
		\partial_\varphi
		\bigl(
		\nablaS P_2\cdot\nablaS Y_\ell^m
		\bigr).
	\]
	We now compute the degree-$\ell$ part of
	$\nablaS P_2\cdot\nablaS Y_\ell^m$. The product identity
	\[
		2\nablaS P_2\cdot\nablaS Y_\ell^m
		=
		\Delta_{\Sph}(P_2Y_\ell^m)
		-
		(\Delta_{\Sph}P_2)Y_\ell^m
		-
		P_2\Delta_{\Sph}Y_\ell^m
	\]
	gives, after taking the degree-$\ell$ part,
	\[
		\begin{aligned}
			2\frP_{\ell}(\nablaS P_2\cdot\nablaS Y_\ell^m)
			 & =
			-\ell(\ell+1)\frP_{\ell}(P_2Y_\ell^m)
			+6\frP_{\ell}(P_2Y_\ell^m)
			+\ell(\ell+1)\frP_{\ell}(P_2Y_\ell^m) \\
			 & =
			6I_{\ell m}Y_\ell^m.
		\end{aligned}
	\]
	Here we used
	\[
		\frP_{\ell}(P_2Y_\ell^m)=I_{\ell m}Y_\ell^m,
		\qquad
		\Delta_{\Sph}P_2=-6P_2,
		\qquad
		\Delta_{\Sph}Y_\ell^m=-\ell(\ell+1)Y_\ell^m.
	\]
	Thus
	\[
		\frP_{\ell}(\nablaS P_2\cdot\nablaS Y_\ell^m)
		=
		3I_{\ell m}Y_\ell^m.
	\]
	Since $\partial_\varphi$ preserves the degree-$\ell$ subspace, it follows
	that
	\[
		\begin{aligned}
			\frP_{\ell}(\nablaS P_2\cdot\nablaS g)
			 & =
			-\alpha_*
			\partial_\varphi
			\frP_{\ell}(\nablaS P_2\cdot\nablaS Y_\ell^m) \\
			 & =
			-\alpha_*
			\partial_\varphi
			\bigl(3I_{\ell m}Y_\ell^m\bigr)
			=
			3I_{\ell m}g.
		\end{aligned}
	\]
	Finally, using $f=-\frac{3}{16}P_2$, we obtain
	\begin{equation}
		\label{eq:endpoint-gradf-gradg-degree-l}
		\frP_{\ell}(\nablaS f\cdot\nablaS g)
		=
		-\frac{3}{16}
		\frP_{\ell}(\nablaS P_2\cdot\nablaS g)
		=
		-\frac{9}{16}I_{\ell m}g.
	\end{equation}
	Combining \eqref{eq:endpoint-rot-k-datum-recall},
	\eqref{eq:endpoint-fg-degree-l}, and
	\eqref{eq:endpoint-gradf-gradg-degree-l} gives
	\begin{equation}
		\frP_{\ell}(k^{\rm out})
		=
		-\frac{3}{16}I_{\ell m}
		\left[
			(\ell+2)-\frac3{\ell+1}
			\right]g.
	\end{equation}
	Substitution into the trace formula
	\eqref{eq:endpoint-rot-trace-shape-recall} gives the degree-$\ell$ trace
	correction
	\begin{equation}
		\frP_{\ell}\left(fg+v\big|_{r=1}\right)
		=
		-\frac{3}{16}I_{\ell m}
		\left[
			1-
			\frac{(\ell+2)-3/(\ell+1)}{\ell+1}
			\right]g.
	\end{equation}

	Because
	$g=-\alpha_*\partial_\varphi Y_\ell^m$ and
	$\partial_\varphi g=\alpha_*m^2Y_\ell^m$, we obtain
	\begin{align}
		\label{eq:endpoint-dR}
		d_{\ell m}^{R}
		 & =
		-\frac{3\alpha_*^2}{16}m^2I_{\ell m}
		\left[
			1-
			\frac{(\ell+2)-3/(\ell+1)}{\ell+1}
			\right] \\
		 & =
		\frac94
		\frac{m^2(\ell-2)}{(\ell+1)^2}I_{\ell m}. \nonumber
	\end{align}

	\textbf{Part 4.}
	Finally, we turn to the curvature part, where
	\begin{equation}
		d^H_{\ell m}=\avg{Y_\ell^m,[\scrH]_{\eps t}}_{\CC}.
	\end{equation}

	This coefficient has been computed in Lemma \ref{lem:curvature-cross-h02}. Formula \eqref{eq:curvature-cross-h02-diagonal} states that
	\begin{equation}
		\label{eq:endpoint-dH}
		d_{\ell m}^{H}=\avg{Y_\ell^m,
			\left[H_{\eps f+tY_\ell^m}\circ\Psi_{\eps f+tY_\ell^m}\right]_{\eps t}
		}_{\CC}
		=
		\frac38
		\left(\ell(\ell+1)+4\right)I_{\ell m}.
	\end{equation}

	\smallskip

	Combining \eqref{eq:endpoint-dT}, \eqref{eq:endpoint-dR}, and \eqref{eq:endpoint-dH}
	gives \eqref{eq:endpoint-dlm-formula}.
\end{proof}

\subsection{Schur complements}\label{sec:schur_complements}

We now turn to the expansion of $D_aw(0,\alpha,V)$. Here the first order term is sufficient.

\begin{lemma}\label{lem:expansion_of_D_aw}
	Uniformly in $\alpha$ from a neighbourhood of $\alpha_*$, as $V\rightarrow 0$ there holds
	\begin{align*}
		D_aw(0,\alpha,V)=-V(L_0^\alpha)^{-1}{\idproj_{1,6}}L_1^\alpha+O_{\cL(K_{16}\cap\H^{k+2},K_{16}^\perp\cap \H^{k+2})}(V^2),
	\end{align*}
	where $L_0^\alpha$, $L_1^\alpha$ are the coefficients from \eqref{eq:endpoint-L-expansion-prelemma}, and $(L_0^\alpha)^{-1}$ is understood as the inverse of $L_0^\alpha:K_{16}^\perp\cap \H^{k+2}_{\rm lin}\rightarrow K_{16}^\perp\cap\H^k_{\rm lin}$.
\end{lemma}
\begin{proof}
	First of all we recall that $L^{\alpha_*}_0$ is an isomorphism from $K_{16}^\perp\cap \H^{k+2}_{\rm lin}$ to $K_{16}^\perp\cap\H^k_{\rm lin}$ by Lemma~\ref{lem:endpoint-complement-isomorphism-simple}. Since the dependence on $\alpha$ is smooth, the same remains true for $L_0^\alpha$ with $\alpha$ close to $\alpha_*$. Thus $(L^\alpha_0)^{-1}$ as stated is well-defined.

	Differentiation of \eqref{eq:range_equation_endpoint_case} at $a=0$ in the direction of $\tilde{a}\in K_{16}$ gives
	\begin{align}\label{eq:derivative_range_equation_endpoint_case}
		{\idproj_{1,6}} D_\xi\cF(w(0,\alpha,V),\alpha,V)[\tilde{a}+D_aw(0,\alpha,V)\tilde{a}]=0.
	\end{align}
	Evaluation at $V=0$, use of $w(0,\alpha,0)=0$ and $D_\xi\cF(0,\alpha,0)=L_0^\alpha$, cf. \eqref{eq:recalled_relation_F_B}, imply
	\begin{align*}
		{\idproj_{1,6}}L_0^\alpha D_aw(0,\alpha,0)\tilde{a}=-{\idproj_{1,6}}L_0^\alpha\tilde{a}.
	\end{align*}
	Since $L_0^\alpha$ is diagonal, cf. Lemma \ref{lem:endpoint-L0}, the right-hand side vanishes and thus $L_0^\alpha D_aw(0,\alpha,0)\tilde{a}$ is an element of $K_{16}$. On the other hand, by construction $w(a,\alpha,V)\in K_{16}^\perp$, such that the diagonality of $L^\alpha_0$ gives $L_0^\alpha D_aw(0,\alpha,0)\tilde{a}\in K_{16}^\perp$. Together and since $L_0^\alpha$ is an isomorphism on $K_{16}^\perp$ we find $D_aw(0,\alpha,0)=0$ for all $\alpha$ sufficiently close to $\alpha_*$.

	Consequently
	\begin{align}\label{eq:first_order_dvda_w_0}
		D_aw(0,\alpha,V)=V\partial_VD_aw(0,\alpha,0)+O(V^2).
	\end{align}
	In order to compute $\partial_VD_aw(0,\alpha,0)$, we differentiate
	\eqref{eq:derivative_range_equation_endpoint_case} with respect to $V$ at
	$V=0$. Since $\partial_Vw(0,\alpha,0)=0$ by
	\eqref{eq:axisymmetric-coordinate-expansion}, there holds
	\begin{align}\label{eq:first_order_dvda_w_1}
		{\idproj_{1,6}} L_0^\alpha\partial_VD_aw(0,\alpha,0)[\tilde{a}]=-{\idproj_{1,6}}\partial_VD_\xi\cF(0,\alpha,0)[\tilde{a}].
	\end{align}
	We claim that
	\begin{align}\label{eq:first_order_dvda_w_2}
		{\idproj_{1,6}}\partial_VD_\xi\cF(0,\alpha,0)[\tilde{a}]={\idproj_{1,6}}L_1^\alpha[\tilde{a}].
	\end{align}
	Indeed, differentiation of \eqref{eq:recalled_relation_F_B} at $\xi=w(0,\alpha,V)$ in the direction of $\tilde{a}$ gives
	\begin{align*}
		\int_{\Sph} & D_\xi\cF(w(0,\alpha,V),\alpha,V)[\tilde{a}]\eta\dd S                                                                                          \\
		            & =\int_{\Sph}D_h\cB(\Theta(w(0,\alpha,V)),\alpha,V)D\Theta(w(0,\alpha,V))[\tilde{a}]M_{\Theta(w(0,\alpha,V))}D\Theta(w(0,\alpha,V))[\eta]\dd S
	\end{align*}
	for all $\eta\in K_{16}^\perp\cap\H^{k+2}_{\rm lin}$.
	Here we used that $\cB(\Theta(w(0,\alpha,V)),\alpha,V)=\cB(h^{\rm ax}_{0,V},\alpha,V)=0$. Taking now the $V$-derivative at $V=0$, minding again that $\partial_Vw(0,\alpha,0)=0$, we find
	\begin{align*}
		\int_{\Sph} & \partial_VD_\xi\cF(0,\alpha,0)[\tilde{a}]\eta\dd S=\int_{\Sph}L_1^\alpha[\tilde{a}]\eta\dd S
	\end{align*}
	by means of \eqref{eq:Theta-expansion} and \eqref{eq:endpoint-L-expansion-prelemma}. This implies \eqref{eq:first_order_dvda_w_2}.

	Combining \eqref{eq:first_order_dvda_w_1} and \eqref{eq:first_order_dvda_w_2} we therefore have
	\begin{align}\label{eq:first_order_dvda_w_4}
		{\idproj_{1,6}} L_0^\alpha\partial_VD_aw(0,\alpha,0)=-{\idproj_{1,6}} L_1^\alpha.
	\end{align}
	Since $w$ anyhow maps into $K_{16}^\perp$ and $L_0^\alpha$ is diagonal, we can drop the projection on the left-hand side, such that \eqref{eq:first_order_dvda_w_0} and \eqref{eq:first_order_dvda_w_4} imply the statement.
\end{proof}

\subsection{Conclusion of the eigenvalue expansion}\label{sec:proof_lemma_expansion}

The coefficients $L_i^\alpha$, $i=0,1,2$, below are again taken from the expansion \eqref{eq:endpoint-L-expansion-prelemma}, and as in Lemma \ref{lem:expansion_of_D_aw}, $(L^\alpha_0)^{-1}:K_{16}^\perp\cap \H^k_{\rm lin}\rightarrow K_{16}^\perp\cap \H^{k+2}_{\rm lin}$. For now we keep the expansion at an arbitrary $\alpha$ near $\alpha_*$.

\begin{lemma}\label{lem:conclusion_eigenvalue_expansion}
	Let $a_m\in K_m$, $m=1,6$ with $\norm{a_m}_{\L^2(\Sph)}=1$. Then
	\begin{align*}
		\mu_m(\alpha,V) & =\int_{\Sph}a_mL_0^\alpha a_m\dd S+V\int_{\Sph}a_mL_1^\alpha a_m\dd S                                                                            \\
		                & \phantom{===}+V^2\int_{\Sph}a_m\left(2h_{02}L_0^\alpha +L_2^\alpha -L_1^\alpha(L_0^\alpha)^{-1}{\idproj_{1,6}} L_1^\alpha \right)a_m\dd S+O(V^3)
	\end{align*}
	uniformly in $\alpha$ near $\alpha_*$ as $V\rightarrow 0$.
\end{lemma}
\begin{proof}
	By Lemmas \ref{lem:first_simplification_of_eigenvalues}, \ref{lem:endpoint-L-expansion} and \ref{lem:expansion_of_D_aw} we have
	\begin{align*}
		\mu_m(\alpha,V) & =\int_{\Sph}L_0^\alpha[a_m+D_aw(0,\alpha,V)a_m](1+2V^2h_{02})a_m\dd S                                               \\
		                & \phantom{===}+V\int_{\Sph}L_1^\alpha[a_m+D_aw(0,\alpha,V)a_m]a_m\dd S+V^2\int_{\Sph}L_2^\alpha [a_m]a_m\dd S+O(V^3)
	\end{align*}
	uniformly in $\alpha$. Since
	\begin{align*}
		\int_{\Sph}L_0^\alpha[D_aw(0,\alpha,V)a_m]a_m\dd S=0,
	\end{align*}
	which is a consequence of $w(a,\alpha,V)\in K_{16}^\perp$ and the diagonality of $L_0^\alpha$,
	the stated expansion follows by another application of Lemma \ref{lem:expansion_of_D_aw}.
\end{proof}

Considering now $\alpha$ to be order $V^2$ close to $\alpha_*$ we can characterise the eigenvalues $\mu_m(\alpha,V)$ in terms of the quantities $\lambda_{\ell m}(\alpha)$, $b^{(m)}_{\ell\pm 1,\ell}$ and $d_{\ell m}$ from Lemmas \ref{lem:endpoint-L0}, \ref{lem:endpoint-L1} and \ref{lem:endpoint-L2-diagonal}.
\begin{corollary}\label{cor:eigenvalue_formula_general_lm}
	Let $(\ell,m)\in\set{(2,1),(7,6)}$.
	For each $C_0>0$ there is a constant $V_0>0$ such that, if
	$
		0\le V\le V_0,$ $\abs{\alpha-\alpha_*}\le C_0V^2,
	$
	then
	\begin{align}\label{eq:eigenvalue_formula_general_lm}
		\mu_m(\alpha,V)=\lambda_{\ell m}(\alpha)+V^2\left(d_{\ell m}-\sum_{n=\ell\pm 1}\frac{b^{(m)}_{\ell, n}b^{(m)}_{n,\ell }}{\lambda_{nm}(\alpha_*)}\right)+O(V^3)
	\end{align}
	{uniformly as $V\rightarrow 0$.}
\end{corollary}
\begin{remark}
	In the positive resonant cases with $\ell=m$, see
	Section~\ref{sec:positive-resonant-density-ratios}, the lower neighbour
	$Y_{\ell-1}^{m}$ does not exist. The sum over $n=\ell\pm1$ therefore consists
	only of the term $n=\ell+1$.
\end{remark}

\begin{proof}[Proof of Corollary \ref{cor:eigenvalue_formula_general_lm}.]
	Evaluating every term except the first one in the formula of Lemma~\ref{lem:conclusion_eigenvalue_expansion} at $\alpha_*$ changes the order-$V$ and order-$V^2$ contributions by $O(V\abs{\alpha-\alpha_*})$ and $O(V^2\abs{\alpha-\alpha_*})$, respectively. Since $\abs{\alpha-\alpha_*}\le C_0V^2$, the combined additional error is $O(V^3)$. Thus
	\begin{align*}
		\mu_m(\alpha,V) & =\int_{\Sph}a_mL_0^\alpha a_m\dd S+V\int_{\Sph}a_mL_1^{\alpha_*} a_m\dd S                                                                                            \\
		                & \phantom{===}+V^2\int_{\Sph}a_m\left(2h_{02}L_0^{\alpha_*} +L_2^{\alpha_*} -L_1^{\alpha_*}(L_0^{\alpha_*})^{-1}{\idproj_{1,6}} L_1^{\alpha_*} \right)a_m\dd S+O(V^3) \\
		                & =\lambda_{\ell m}(\alpha)+V^2\left(d_{\ell m}-\int_{\Sph}a_mL_1^{\alpha_*}(L_0^{\alpha_*})^{-1}{\idproj_{1,6}} L_1^{\alpha_*}a_m\dd S\right)+O(V^3).
	\end{align*}
	For the second equality we have used the definitions of $\lambda_{\ell m}(\alpha)$ and $d_{\ell m}$ in Lemmas \ref{lem:endpoint-L0}, \ref{lem:endpoint-L2-diagonal}, as well as $a_m\in K_{16}=\ker L_0^{\alpha_*}$ and the fact that $L_1^{\alpha_*}$ has no diagonal contribution in view of Lemma \ref{lem:endpoint-L1}.

	We compute the remaining term by means of Lemmas~\ref{lem:endpoint-L1} and \ref{lem:endpoint-L0}. After complexifying the real plane $K_m$ and the real-linear operators, the same scalar block eigenvalue may be computed on $K_m\otimes_{\R}\CC$. We may therefore take $a_m=Y_\ell^m\in K_m\otimes_{\R}\CC$ and use the Hermitian pairing $\avg{\cdot,\cdot}_{\CC}$ defined above. We then have
	\begin{align*}
		(L_0^{\alpha_*})^{-1}{\idproj_{1,6}} L_1^{\alpha_*}a_m & =(L_0^{\alpha_*})^{-1}{\idproj_{1,6}}\left(b^{(m)}_{\ell+1,\ell}Y_{\ell+1}^m+b^{(m)}_{\ell-1,\ell}Y_{\ell-1}^m\right)                           \\
		                                                       & =(L_0^{\alpha_*})^{-1}\left(b^{(m)}_{\ell+1,\ell}Y_{\ell+1}^m+b^{(m)}_{\ell-1,\ell}Y_{\ell-1}^m\right)                                          \\
		                                                       & =\frac{b^{(m)}_{\ell+1,\ell}}{\lambda_{\ell+1,m}(\alpha_*)}Y_{\ell+1}^m+\frac{b^{(m)}_{\ell-1,\ell}}{\lambda_{\ell-1,m}(\alpha_*)}Y_{\ell-1}^m.
	\end{align*}
	Applying $L_1^{\alpha_*}$ once more to each of these two terms produces two further terms, of which only the return to degree $\ell$ contributes to the Hermitian pairing $\avg{Y_\ell^m,\cdot}_{\CC}$. For $Y^m_{\ell+1}$ this gives the additional factor $b^{(m)}_{\ell,\ell+1}$, while for $Y^m_{\ell-1}$ it gives the factor $b^{(m)}_{\ell,\ell-1}$.
\end{proof}

\begin{proof}[Proof of Lemma \ref{lem:expansion_mu}]
	By means of Corollary \ref{cor:eigenvalue_formula_general_lm} it remains to compute the coefficients in \eqref{eq:eigenvalue_formula_general_lm} for $(\ell,m)=(2,1)$ and $(\ell,m)=(7,6)$. We recall the definition of $a_{\ell,\pm}^m$ in \eqref{eq:a-coefficients} and $I_{\ell m}$ in \eqref{eq:endpoint-I-lm}.

	For $(\ell,m)=(2,1)$,
	\[
		(a_{2,+}^1)^2=\frac8{35},
		\qquad
		(a_{2,-}^1)^2=\frac15,
		\qquad
		I_{21}=\frac17,
	\]
	and Lemma~\ref{lem:endpoint-L2-diagonal} gives $d_{21}=-129/280$. The
	neighbouring eigenvalues are $\lambda_{11}(\alpha_*)=-6$ and
	$\lambda_{31}(\alpha_*)=7$. Using \eqref{eq:endpoint-b-coefficients},
	\[
		\frac{b_{2,1}^{(1)}b_{1,2}^{(1)}}{-\lambda_{11}(\alpha_*)}
		=
		\frac{
			\left(-\frac{9i\alpha_*}{4}a_{1,+}^1\right)
			\left(\frac{9i\alpha_*}{4}a_{2,-}^1\right)
		}{6}
		=
		\frac{81}{40},
	\]
	and
	\[
		\frac{b_{2,3}^{(1)}b_{3,2}^{(1)}}{-\lambda_{31}(\alpha_*)}
		=
		\frac{
			\left(\frac{5i\alpha_*}{2}a_{3,-}^1\right)
			\left(-\frac{5i\alpha_*}{2}a_{2,+}^1\right)
		}{-7}
		=
		-\frac{120}{49}.
	\]
	Here we used $\alpha_*^2=12$, $(a_{1,+}^1)^2=(a_{2,-}^1)^2=1/5$, and
	$(a_{3,-}^1)^2=(a_{2,+}^1)^2=8/35$. Therefore
	\[
		d_{21}
		+
		\sum_{n=1,3}\frac{b_{2,n}^{(1)}b_{n,2}^{(1)}}{-\lambda_{n1}(\alpha_*)}
		=
		-\frac{129}{280}+\frac{81}{40}-\frac{120}{49}
		=
		-\frac{867}{980}.
	\]

	For $(\ell,m)=(7,6)$,
	\[
		(a_{7,+}^6)^2=\frac{28}{255},
		\qquad
		(a_{7,-}^6)^2=\frac1{15},
		\qquad
		I_{76}=-\frac4{17},
	\]
	and Lemma~\ref{lem:endpoint-L2-diagonal} gives $d_{76}=-75447/9520$. The
	neighbouring eigenvalues are
	$\lambda_{66}(\alpha_*)=-152/7$ and $\lambda_{86}(\alpha_*)=22$. Again using
	\eqref{eq:endpoint-b-coefficients},
	\[
		\frac{b_{7,6}^{(6)}b_{6,7}^{(6)}}{-\lambda_{66}(\alpha_*)}
		=
		\frac{
			\left(-\frac{117i\alpha_*}{7}a_{6,+}^6\right)
			\left(\frac{117i\alpha_*}{7}a_{7,-}^6\right)
		}{152/7}
		=
		\frac{13689}{1330},
	\]
	and
	\[
		\frac{b_{7,8}^{(6)}b_{8,7}^{(6)}}{-\lambda_{86}(\alpha_*)}
		=
		\frac{
			\left(\frac{135i\alpha_*}{8}a_{8,-}^6\right)
			\left(-\frac{135i\alpha_*}{8}a_{7,+}^6\right)
		}{-22}
		=
		-\frac{25515}{1496}.
	\]
	Here $(a_{6,+}^6)^2=(a_{7,-}^6)^2=1/15$ and
	$(a_{8,-}^6)^2=(a_{7,+}^6)^2=28/255$. Therefore
	\[
		d_{76}
		+
		\sum_{n=6,8}\frac{b_{7,n}^{(6)}b_{n,7}^{(6)}}{-\lambda_{n6}(\alpha_*)}
		=
		-\frac{75447}{9520}+\frac{13689}{1330}-\frac{25515}{1496}
		=
		-\frac{4174947}{284240}.
	\]
	Substitution in \eqref{eq:eigenvalue_formula_general_lm} proves
	\eqref{eq:endpoint-mu1-expansion} and \eqref{eq:endpoint-mu6-expansion}.
\end{proof}

\subsection{Anisotropic blow-up variables}

We now carry out the anisotropic rescaling announced at the
beginning of this section. The following equivariant selection
estimate supplies the vanishing orders required for smoothness of the rescaled
system introduced in \eqref{eq:endpoint-anisotropic-maps-simple}, \eqref{eq:inddividual_maps}.

\begin{lemma}
	\label{lem:sharp-weight-one-secondary-variable}
	Let $E_1$ and $E_m$ be real two-dimensional rotation representations of
	weights $1$ and $m\ge2$. Let
	$\Psi=(\Psi_1,\Psi_m):E_1\oplus E_m\to E_1\oplus E_m$ be a smooth
	$\so$-equivariant map, depending smoothly on auxiliary scalar parameters,
	with
	\[
		\Psi(0)=0,
		\qquad
		D\Psi(0)=
		\mu_1\id_{E_1}\oplus\mu_m\id_{E_m}.
	\]
	Then, on a fixed small neighbourhood of $0$, uniformly in the auxiliary
	parameters,
	\begin{align}
		\label{eq:selection-secondary-forcing-simple}
		\norm{\Psi_m(a,0)}
		 & \le C\norm a^m,                                \\
		\label{eq:selection-secondary-linearisation-simple}
		\norm{D_u\Psi_m(a,u)-\mu_m\id_{E_m}}
		 & \le C\bigl(\norm a^2+\norm u\bigr),            \\
		\label{eq:selection-primary-simple}
		\norm{\Psi_1(a,0)-\mu_1a}
		 & \le C\norm a^3,                                \\
		\label{eq:selection-primary-secondary-simple}
		\norm{D_u\Psi_1(a,u)}
		 & \le C\bigl(\norm a^{m-1}+\norm a\norm u\bigr).
	\end{align}
\end{lemma}

\begin{proof}
	Identify $E_1$ and $E_m$ with $\CC$, with complex coordinates $z$ and $q$,
	respectively, so that rotation through $\theta$ acts by
	\[
		(z,q)\longmapsto(e^{i\theta}z,e^{im\theta}q).
	\]
	We suppress the auxiliary scalar parameters, on which rotations act
	trivially. A monomial
	\[
		z^p\bar z^r q^s\bar q^t
	\]
	acquires under rotation the factor
	\[
		e^{i\left(p-r+m(s-t)\right)\theta}.
	\]
	Since the second component transforms with weight $m$, such a monomial can
	occur in $\Psi_m$ only if
	\begin{equation}
		\label{eq:secondary-monomial-weight-rule}
		p-r+m(s-t)=m.
	\end{equation}

	At $q=0$,
	condition
	\eqref{eq:secondary-monomial-weight-rule} becomes
	\begin{align}\label{eq:monomoiallalalal}
		p-r=m.
	\end{align}
	Consequently \(p+r\ge m\), so every Taylor term in
	$\Psi_m(z,0)$ of degree less than $m$ vanishes. Taylor's theorem gives
	\[
		\norm{\Psi_m(z,0)}\le C\abs{z}^m
	\]
	uniformly in the omitted parameters.

	We next consider the derivative in the secondary variable. At $q=0$, every
	real-linear map from the $q$-plane to itself has a unique decomposition
	\[
		D_q\Psi_m(z,0)[v]
		=
		A(z,\bar z)v+B(z,\bar z)\bar v,
		\qquad v\in\CC.
	\]
	Differentiating the equivariance relation in the direction $v$ gives
	\[
		D_q\Psi_m(e^{i\theta}z,0)[e^{im\theta}v]
		=
		e^{im\theta}D_q\Psi_m(z,0)[v].
	\]
	Comparison of the complex-linear and anti-complex-linear parts yields
	\[
		A(e^{i\theta}z,e^{-i\theta}\bar z)=A(z,\bar z),
		\qquad
		B(e^{i\theta}z,e^{-i\theta}\bar z)
		=
		e^{2im\theta}B(z,\bar z).
	\]
	Thus a Taylor monomial $z^p\bar z^r$ in $A$ must satisfy $p-r=0$.
	Every nonconstant such monomial has degree at least two, and hence
	\[
		A(z,\bar z)=A(0,0)+O(\abs{z}^2).
	\]
	A monomial in $B$ must instead satisfy $p-r=2m$, so its degree is at least
	$2m$. Therefore
	\[
		B(z,\bar z)=O(\abs{z}^{2m})=O(\abs{z}^2).
	\]
	The identity
	\[
		D_q\Psi_m(0,0)=\mu_m\id_{E_m}
	\]
	means that $A(0,0)=\mu_m$ and $B(0,0)=0$. It follows that
	\[
		\norm{D_q\Psi_m(z,0)-\mu_m\id_{E_m}}
		\le C\abs{z}^2.
	\]
	Finally, smoothness and the mean-value theorem give
	\[
		\norm{D_q\Psi_m(z,q)-D_q\Psi_m(z,0)}
		\le C\abs{q}.
	\]
	All these estimates are uniform  in the auxiliary parameters.
	Returning to the real variables $a$ and $u$ proves
	\eqref{eq:selection-secondary-linearisation-simple}.

	For the first component, the analogue of condition \eqref{eq:monomoiallalalal} is $p-r=1$. Thus, the first nonlinear pure \(z\)-monomial has
	degree three, which gives
	\eqref{eq:selection-primary-simple}. At \(u=0\), for a coefficient of $\Psi_1$,
	linear in \(q\), the analogue of \eqref{eq:secondary-monomial-weight-rule} is $p-r\pm m=1$. It therefore contains at least \(m-1\) factors of
	\(z,\bar z\). Hence
	\[
		\norm{D_u\Psi_1(a,0)}
		\le C\norm a^{m-1}.
	\]

	Moreover, rotation through \(2\pi/m\) fixes every \(u\in E_m\) and
	has no nonzero fixed vector in \(E_1\). Thus
	\[
		\Psi_1(0,u)=0,
		\qquad
		D_u\Psi_1(0,u)=0.
	\]
	The difference
	\[
		D_u\Psi_1(a,u)-D_u\Psi_1(a,0)
	\]
	vanishes both when \(a=0\) and when \(u=0\). Taylor expansion therefore gives
	\[
		\norm{D_u\Psi_1(a,u)-D_u\Psi_1(a,0)}
		\le C\norm a\norm u.
	\]
	This proves
	\eqref{eq:selection-primary-secondary-simple}.
\end{proof}

For $a\in K_1$ and $u\in K_6$, define the two components of the reduced map by
\begin{equation}\label{eq:inddividual_maps}
	\Phi_j(a,u,\alpha,V)
	:=
		{\proj_j}\cG(a+u,\alpha,V),
	\qquad j\in\set{1,6}.
\end{equation}
They satisfy the estimates of
Lemma~\ref{lem:sharp-weight-one-secondary-variable} with $m=6$. Indeed,
$\cG$ is smooth, equivariant, and is the gradient of the reduced invariant
functional, while Lemma~\ref{lem:derivative_G} gives its linearisation at the
axisymmetric branch. Before turning to the nonlinear problem, having the bifurcation in mind we first prepare the eigenvalues of the linearisation of $\cG$ by means of a suitable frequency choice $\alpha_1(V)$.

Set
\begin{equation}
	\label{eq:endpoint-kappas-simple}
	\gamma_6
	:=
	\frac{4174947}{284240}
	-\frac{27}{2}\frac{867}{980}
	=
	\frac{38228049}{13927760}
	>0.
\end{equation}

\begin{lemma}\label{lem:endpoint-crossings-simple}
	There are $V_1,c_1,c_6,c_*>0$ and a smooth function $\alpha_1(V)$ such
	that
	\begin{equation}
		\label{eq:endpoint-alpha1-simple}
		\mu_1(\alpha_1(V),V)=0,
		\qquad
		\alpha_1(V)
		=
		\alpha_*
		-\frac{2601}{1960}\frac{V^2}{2\sqrt3}
		+O(V^3).
	\end{equation}
	For $0<V<V_1$ and
	$\abs{\alpha-\alpha_1(V)}\le c_1V^2$,
	\begin{equation}
		\label{eq:endpoint-crossing-bounds-simple}
		\abs{\mu_6(\alpha,V)}\ge c_6V^2,
		\qquad
		\partial_\alpha\mu_1(\alpha,V)\le-c_*.
	\end{equation}
\end{lemma}

\begin{proof}
	At $\tau=0$,
	\[
		\lambda_{21}(\alpha)=4-\frac{\alpha^2}{3},
		\qquad
		\lambda_{76}(\alpha)=54-\frac92\alpha^2
	\]
	by \eqref{eq:lambda-lm}.
	Thus \eqref{eq:endpoint-mu1-expansion} and
	$\partial_\alpha\lambda_{21}(\alpha_*)=-2\alpha_*/3\ne0$ give
	$\alpha_1$ by the implicit-function theorem and yield
	\eqref{eq:endpoint-alpha1-simple}. Substitution in
	\eqref{eq:endpoint-mu6-expansion} gives
	\[
		\mu_6(\alpha_1(V),V)
		=
		-\gamma_6V^2+O(V^3).
	\]
	The positivity in \eqref{eq:endpoint-kappas-simple}, continuity of the
	$\alpha$-derivatives, and a sufficiently small choice of $c_1$ now imply
	\eqref{eq:endpoint-crossing-bounds-simple}.
\end{proof}

\begin{proposition}
	\label{prop:endpoint-anisotropic-reduction-simple}
	Put $e:=XZ\in K_1$. After decreasing $V_1$, there are constants
	$\eta_{\rm an},\delta_{\rm an}>0$ and $0<c_{\rm an}<c_1$, together with
	smooth maps
	\[
		U_*:(-\eta_{\rm an},\eta_{\rm an})\times[0,V_1)\longrightarrow K_6,
		\qquad
		\beta:(-\eta_{\rm an},\eta_{\rm an})\times[0,V_1)\longrightarrow\R
	\]
	with
	$
		U_*(0,V)=0$, $
		\beta(0,V)=0
	$,
	and such that
	\begin{equation}
		\label{eq:endpoint-anisotropic-variables-simple}
		a=Vte,
		\qquad
		u=V^2U_*(t,V),
		\qquad
		\alpha=\alpha_1(V)+V^2\beta(t,V)
	\end{equation}
	defined for $0<V<V_1$ and $0<\abs t<\eta_{\rm an}$, satisfy
	\begin{align*}
		\cG(a+u,\alpha,V)=0.
	\end{align*}
	For these $V$ and $t$, the pair $(U_*(t,V),\beta(t,V))$ is the unique pair
	in the regime
	$\norm U_{\L^2}<\delta_{\rm an}$, $\abs{\widetilde\beta}<c_{\rm an}$
	that solves this equation after the substitutions
	$u=V^2U$ and $\alpha=\alpha_1(V)+V^2\widetilde\beta$.
	Moreover,
	\begin{equation}
		\label{eq:endpoint-anisotropic-estimates-simple}
		\norm{U_*(t,V)}_{\L^2}\le CV^2\abs t^6,
		\qquad
		\abs{\beta(t,V)}\le Ct^2.
	\end{equation}
	Consequently, the variables in
	\eqref{eq:endpoint-anisotropic-variables-simple} satisfy
	\begin{equation}
		\label{eq:endpoint-u-bounds-simple}
		\norm u_{\L^2}
		\le C\frac{\norm a_{\L^2}^6}{V^2}
	\end{equation}
	and
	\begin{equation}
		\label{eq:endpoint-alpha-selection-simple}
		\alpha=\alpha_1(V)+O(\norm a_{\L^2}^2).
	\end{equation}
	For $0<V<V_1$ and $t=0$, the original kernel equation $\cG=0$ instead has
	$U=0$ as the
	unique solution in $\norm U_{\L^2}<\delta_{\rm an}$ for
	every $\abs{\widetilde\beta}<c_{\rm an}$. Thus the original equation does
	not select the angular velocity on the axisymmetric set.
\end{proposition}

\begin{proof}
	Instead of $\cG=0$, or equivalently $\Phi_1=0$, $\Phi_6=0$, cf. \eqref{eq:inddividual_maps}, we consider the system
	\begin{align}
		\label{eq:endpoint-anisotropic-maps-simple}
		\mathfrak F_6(t,U,\widetilde\beta,V)
		 & :=V^{-4}\Phi_6
		(Vte,V^2U,\alpha_1(V)+V^2\widetilde\beta,V), \\
		\mathfrak F_1(t,U,\widetilde\beta,V)
		 & :=
		\frac{\left\langle
			\Phi_1(Vte,V^2U,\alpha_1(V)+V^2\widetilde\beta,V),e
			\right\rangle_{\L^2}}
		{V^3t\norm{e}_{\L^2}^2}, \nonumber
	\end{align}
	which is well-defined as long as $tV\neq 0$.

	We first show that the rescaled maps extend smoothly when $V,t\rightarrow 0$. The weight selection in
	Lemma~\ref{lem:sharp-weight-one-secondary-variable} gives the required
	vanishing orders. Indeed, by said Lemma Taylor's theorem with integral remainder, applied
	successively in $x$ and $u$ with $p=(\alpha,V)$ as a parameter,
	gives
	smooth coefficient maps for which
	\begin{align}
		\label{eq:endpoint-exact-factorisations-simple}
		\Phi_6(xe,0,p)
		 & =x^6R_6(x,p),                                        \\
		D_u\Phi_6(xe,u,p)
		 & =\mu_6(p)\id_{K_6}+x^2B_6(x,p)
		+C_6(x,u,p)[u,\mathord\cdot],\nonumber                  \\
		\Phi_1(xe,0,p)
		 & =\mu_1(p)xe+x^3R_1(x,p),\nonumber                    \\
		D_u\Phi_1(xe,u,p)
		 & =x^5B_1(x,p)+xC_1(x,u,p)[u,\mathord\cdot]. \nonumber
	\end{align}
	Here $C_j(x,u,p)[u,\mathord\cdot]$ is linear in the slot denoted by
	$\mathord\cdot$. These are exact identities. For example, the first and third
	identities follow by Taylor expansion in $x$. For the
	second identity, first factor
	$D_u\Phi_6(xe,0,p)-\mu_6(p)\id_{K_6}$ by $x^2$, and then factor its change
	from $u=0$ by $u$. For the fourth identity, factor
	$D_u\Phi_1(xe,0,p)$ by $x^5$. Its remaining change vanishes when $x=0$
	and when $u=0$, and therefore has the displayed factors $x$ and $u$.

	Next define, initially for $V\ne0$,
	\begin{equation}
		\label{eq:endpoint-rescaled-eigenvalues-simple}
		M_j(\widetilde\beta,V)
		:=\frac{\mu_j(\alpha_1(V)+V^2\widetilde\beta,V)}{V^2},
		\qquad j\in\set{1,6}.
	\end{equation}
	Using \eqref{eq:endpoint-alpha1-simple} and the fundamental theorem of
	calculus in $\alpha$, we obtain
	\[
		M_1(\widetilde\beta,V)
		=
		\widetilde\beta\int_0^1
		\partial_\alpha\mu_1
		(\alpha_1(V)+rV^2\widetilde\beta,V)\dd r.
	\]
	Moreover, the expansion in
	Lemma~\ref{lem:endpoint-crossings-simple}, together with Taylor's theorem
	with integral remainder applied to the smooth function
	$V\mapsto\mu_6(\alpha_1(V),V)$, gives a smooth function $\nu_6$ such that
	\[
		\mu_6(\alpha_1(V),V)=V^2\nu_6(V),
		\qquad
		\nu_6(0)=-\gamma_6.
	\]
	The mean-value formula in the $\alpha$ variable then gives
	\[
		M_6(\widetilde\beta,V)
		=
		\nu_6(V)+\widetilde\beta\int_0^1
		\partial_\alpha\mu_6
		(\alpha_1(V)+rV^2\widetilde\beta,V)\dd r.
	\]
	Thus both functions in \eqref{eq:endpoint-rescaled-eigenvalues-simple}
	extend smoothly to $V=0$.
	At the origin,
	\begin{equation}\label{eq:important_quantities_IFT}
		\partial_{\widetilde\beta}M_1(0,0)
		=-\frac{2\alpha_*}{3},
		\qquad
		M_6(0,0)=-\gamma_6.
	\end{equation}

	Now the first two identities in \eqref{eq:endpoint-exact-factorisations-simple} give
	\begin{align}\label{eq:endpoint-anisotropic-expanded-simple}
		\mathfrak F_6 & =\frac{1}{V^4}\left(\Phi_6(Vte,0,\alpha_1(V) \ldots)+\int_0^1D_u\Phi_6(Vte,\xi V^2U,\alpha_1(V)\ldots)[V^2U]\dd \xi\right) \\
		              & =M_6U+V^2t^6\mathcal R_6
		+t^2\mathcal B_6[U]+\mathcal C_6[U,U]\nonumber
	\end{align}
	with suitable coefficient functions depending smoothly on $(t,U,\widetilde\beta,V)$. In a similar way, the last two identities in \eqref{eq:endpoint-exact-factorisations-simple} give
	\begin{align}\label{eq:endpoint-anisotropic-expanded-simple_F1}
		\mathfrak F_1
		 & =M_1+t^2r_1
		+V^4t^4\mathcal B_1[U]+V^2\mathcal C_1[U,U].
	\end{align}

	Also here,
	all coefficient maps on the right-hand side are smooth in
	$(t,U,\widetilde\beta,V)$. Thus
	\eqref{eq:endpoint-anisotropic-maps-simple} has a smooth extension through
	$Vt=0$. At $(t,U,\widetilde\beta,V)=(0,0,0,0)$,
	\[
		D_{(U,\widetilde\beta)}(\mathfrak F_6,\mathfrak F_1)
		=
		\begin{pmatrix}
			-\gamma_6\id_{K_6} & 0            \\
			0                  & -2\alpha_*/3
		\end{pmatrix}
	\]
	by \eqref{eq:important_quantities_IFT}.
	This derivative is invertible. The finite-dimensional implicit-function
	theorem now gives the unique smooth pair
	$(U_*(t,V),\beta(t,V))$ in a fixed box.

	We record the quantitative bounds. After decreasing the fixed box,
	$M_6$ is uniformly invertible. Equation
	\eqref{eq:endpoint-anisotropic-expanded-simple} gives
	\[
		\norm{U_*}
		\le C\bigl(V^2\abs t^6+t^2\norm{U_*}
		+\norm{U_*}^2\bigr).
	\]
	Absorbing the last two terms proves the first estimate in
	\eqref{eq:endpoint-anisotropic-estimates-simple}. The integral formula for
	$M_1$ has the form $M_1=\widetilde\beta A(\widetilde\beta,V)$, where
	$\abs A$ is bounded below, see \eqref{eq:endpoint-crossing-bounds-simple}. Meanwhile, equation \eqref{eq:endpoint-anisotropic-expanded-simple_F1} gives
	\[
		\abs\beta
		\le C\bigl(t^2+V^4t^4\norm{U_*}
		+V^2\norm{U_*}^2\bigr)
		\le Ct^2.
	\]
	This proves the second estimate. Since $e$ is fixed and nonzero, the two
	physical estimates follow from $a=Vte$ and $u=V^2U_*$.

	At \(t=0\), the pair
	\((U,\widetilde\beta)=(0,0)\) solves the smoothly extended
	system. Uniqueness in the implicit-function theorem therefore gives
	\[
		U_*(0,V)=0,
		\qquad
		\beta(0,V)=0.
	\]

	It remains to compare the rescaled system with the full kernel
	equation. Suppose that \(t\ne0\). Then
	\(\mathfrak F_6=0\) gives \(\Phi_6=0\), while
	\(\mathfrak F_1=0\) gives
	\[
		\left\langle\Phi_1,a\right\rangle_{\L^2}=0,
	\]
	since \(a=Vte\) and \(Vt\ne0\).
	Remark~\ref{rem:LS_higher_dimensional_kernel}, applied with
	\(E=K_1\) and \(F=K_6\), therefore gives
	\[
		\Phi_1=\Phi_6=0.
	\]
	Thus, for \(t\ne0\), the rescaled system is equivalent to the
	full kernel equation.

	At \(t=0\), we return to the original kernel equation. Any solution
	must satisfy its weight-six component. After substituting \(u=V^2U\)
	and dividing by \(V^4\), this equation has the form
	\[
		M_6(\widetilde\beta,V)U+O(\norm U^2)=0.
	\]
	The crossing bound \eqref{eq:important_quantities_IFT} makes the linear term uniformly invertible, and a
	contraction argument therefore gives \(U=0\) as the unique solution in
	the fixed scaled ball.
\end{proof}

\subsection{Conclusion and local classification}

Define the endpoint coefficient
\begin{equation}
	\label{eq:endpoint-h11-ep-simple}
	h_{11}^{\rm ep}
	:=
	-\frac{9}{20\sqrt3}Y
	-\frac{5\sqrt3}{7}Y\left(Z^2-\frac15\right).
\end{equation}

\begin{proof}[Proof of Theorem~\ref{thm:endpoint-din-zero}]
	Choose $V_0\le V_1$ and $0<\eta_0<\eta_{\rm an}$. For
	$\abs s\le\eta_0V$, set
	\[
		t=\frac{s}{V},
		\qquad
		a=sXZ,
		\qquad
		u=V^2U_*(t,V),
		\qquad
		\alpha(s,V)=\alpha_1(V)+V^2\beta(t,V),
	\]
	and
	\[
		\xi(s,V)=a+u+w(a+u,\alpha(s,V),V),
	\]
	and $h(s,V):=\Theta(\xi(s,V))$. The definition of $w$ solves the
	$K_{16}^{\perp}$ equation. Proposition~\ref{prop:endpoint-anisotropic-reduction-simple}
	solves the remaining two kernel equations jointly.
	Thus
	\[
		\cF(\xi(s,V),\alpha(s,V),V)=0.
	\]
	Equation~\eqref{eq:recalled_relation_F_B} gives a solution of the Bernoulli
	free-boundary problem.

	First of all, \eqref{eq:endpoint-alpha1-simple} and
	\eqref{eq:endpoint-alpha-selection-simple} give the improved asymptotics
	\[
		\alpha(s,V)=\alpha_*-\frac{2601}{1960}\frac{V^2}{2\sqrt{3}}+O(s^2+V^3).
	\] Moreover,
	since \(K_6\) is finite-dimensional,
	\eqref{eq:endpoint-u-bounds-simple} and \(a=sXZ\) give
	\[
		\norm{u}_{\H^{k+2}}
		\le C\frac{\abs{s}^6}{V^2}
		\le C\abs{s}^2,
		\qquad
		\abs{s}\le\eta_0V.
	\]
	Thus the \(K_6\)-component is absorbed into the quadratic
	remainder.

	Using Lemmas~\ref{lem:endpoint_axisymmetric_branch_with_w}, \ref{lem:expansion_of_D_aw}, and
	\eqref{eq:Theta-expansion}, we therefore obtain,
	\begin{align*}
		h(s,V)
		=
		h^{\rm ax}_{0,V}
		+sXZ				+O_{\H^{k+2}}\bigl(s^2+\abs{s}V\bigr).
	\end{align*}
	Expanding also the axisymmetric branch by means of \eqref{eq:axisymmetric-branch-expansion}, this is precisely the starting point of the order-\(sV\) calculation in the proof of
	Lemma~\ref{lem:h11-coefficient-main-proof}. Its remaining part applies unchanged:
	the additional \(K_6\)-component is of quadratic order, and the
	angular velocity has no term linear in \(V\). Substituting
	\[
		\tau=0,
		\qquad
		\alpha_*=2\sqrt3,
		\qquad
		\lambda_{31}(\alpha_*)=7
	\]
	in the formula of that lemma gives
	\[
		h_{11}
		=
		-\frac{9}{20\sqrt3}Y
		-\frac{5\sqrt3}{7}
		Y\left(Z^2-\frac15\right)
		=
		h_{11}^{\rm ep}.
	\]

	Therefore,
	\begin{equation}
		\label{eq:endpoint-graph-expansion-simple-proof}
		h(s,V)
		=
		h^{\rm ax}_{0,V}
		+sXZ+sVh_{11}^{\rm ep}
		+O_{\H^{k+2}}\bigl(s^2+\abs{s}V^2\bigr).
	\end{equation}
	Substitution of
	\eqref{eq:axisymmetric-branch-expansion} and
	\eqref{eq:endpoint-h11-ep-simple} gives
	\eqref{eq:endpoint-h-expansion}.

	The centroid computation is the one in
	Lemma~\ref{lem:helical-centroid-main-proof}. The axisymmetric term and
	$sXZ$ have zero horizontal projection, and only the degree-one part
	$-9Y/(20\sqrt3)$ of $h_{11}^{\rm ep}$ contributes at order $sV$. This gives
	\eqref{eq:endpoint-centroid-expansion}.
	Lemma~\ref{lem:small-sobolev-solution-spatial-regularity} gives
	spatial smoothness. Finally, the local uniqueness modulo rotations follows
	from Lemma~\ref{lem:endpoint-local-uniqueness} below.
\end{proof}

\begin{lemma}\label{lem:endpoint-local-uniqueness}
	After decreasing the preceding constants, there are
	$V_{\rm u},c_{\rm u},\eta_{\rm u},\delta_{\rm u},\rho_{\rm u}>0$ such that the following
	holds. Fix $0<V<V_{\rm u}$, and let
	$(\widetilde\alpha,\widetilde h)$ be a sufficiently small radial-graph
	solution in the fixed-volume, vertically centred slice. Put
	\[
		\xi_0(V):=w(0,\widetilde\alpha,V),
		\qquad
		\widetilde x:=\Theta^{-1}(\widetilde h)-\xi_0(V),
	\]
	and decompose
	\[
		a={\proj_1}\widetilde x,
		\qquad
		u={\proj_6}\widetilde x,
		\qquad
		r={\idproj_{1,6}}\widetilde x.
	\]
	Assume
	\begin{equation}
		\label{eq:endpoint-local-tube-simple}
		\abs{\widetilde\alpha-\alpha_1(V)}\le c_{\rm u}V^2,
		\qquad
		\norm a_{\L^2}\le\eta_{\rm u}V,
		\qquad
		\norm u_{\L^2}\le\delta_{\rm u}V^2,
		\qquad
		\norm r_{\H^{k+2}}\le\rho_{\rm u}.
	\end{equation}
	If $a\ne0$, there are unique
	\[
		0<s=\frac{\norm a_{\L^2}}{\norm{XZ}_{\L^2}}<\eta_0V,
		\qquad
		\theta\in\R/(2\pi\mathbb Z),
	\]
	such that
	\begin{equation}
		\label{eq:endpoint-local-classification-simple}
		(\widetilde\alpha,\widetilde h)
		=
		\bigl(\alpha(s,V),\cU_\theta h(s,V)\bigr).
	\end{equation}
	If $a=0$, then $\widetilde h=h^{\rm ax}_{0,V}$, while
	$\widetilde\alpha$ is arbitrary in the strip in
	\eqref{eq:endpoint-local-tube-simple}. Thus the non-axisymmetric branch is locally unique modulo rotations about the
	$e_3$-axis. On the axisymmetric set, the shape is locally unique, while
	$\widetilde\alpha$ remains free.
\end{lemma}

\begin{proof}
	Minding the refined regimes in which our continuation arguments have been carried out, the proof follows as in the second part of the proof of Lemma \ref{lem:nonlinear-branch-smooth-main-proof}.
\end{proof}

\begin{proof}[Proof of Theorem~\ref{thm:din-pos-dimensional} in the case $\tau=0$]
	Apply Theorem~\ref{thm:endpoint-din-zero} with $k=4$ and use the inverse
	nondimensionalisation from Section~\ref{sec:nondimensionalisation}.
	We denote the dimensionless solution objects by hats.
	Since the physical speed satisfies $V>0$, we have
	\[
		\widehat V=\sqrt{\We},
		\qquad
		h(s,\We)=\widehat h(s,\widehat V),
		\qquad
		\frac{\bar y_{s,\We}}{R}
		=
		\widehat{\bar y}_{s,\widehat V},
	\]
	and we take $\We_0:=V_0^2$.

	Consequently, \eqref{eq:endpoint-h-expansion} and
	\eqref{eq:endpoint-centroid-expansion} become
	\eqref{eq:h-expansion} and
	\eqref{eq:centroid-expansion}, respectively.
	Moreover, since the dimensionless critical frequency is $2\sqrt3$,
	\eqref{eq:endpoint-alpha-expansion} gives
	\[
		\frac{\alpha(s,\We)}{\alpha_*}
		=
		1-\frac{2601}{1960\cdot12}\We
		+O\left(s^2+\We^{3/2}\right)
		=
		1-\frac{867}{7840}\We
		+O\left(s^2+\We^{3/2}\right).
	\]
	The condition $\abs{s}\le\eta\sqrt{\We}$ is precisely
	$\abs{s}\le\eta\widehat V$. Smoothness and the local classification are
	preserved by the rescaling.
\end{proof}

\section{Comparison with Benjamin's prediction}
\label{sec:endpoint-benjamin-comparison}

Benjamin's general model has an inviscid exterior liquid and a gas of negligible
inertia; the comparison here uses his fixed-volume specialisation
\cite[p.~350]{MR905816}. Write $\rho,c,\omega,r_{\rm e}$ for Benjamin's
liquid density, axial translation speed, angular speed, and equivalent spherical
radius, respectively. Since $\frac{4\pi}{3}r_{\rm e}^{3}$ is the prescribed
volume \cite[p.~365]{MR905816}, the dictionary with our endpoint problem is
\[
	\rho=\dout,\qquad c=V,\qquad \omega=\alpha,\qquad r_{\rm e}=R.
\]
Benjamin denotes the radius of the centroid helix by $R$; to avoid its clash
with our equivalent radius, we denote that quantity by $R_{\rm B}$ below.
His variational characterisation of steady spinning and spiralling
motions is \cite[\S2.7, eq.~(2.24)]{MR905816}. In \S5 Benjamin derives the
helical-trajectory relations \cite[Fig.~3 and pp.~369--370,
	eqs.~(5.12)--(5.14)]{MR905816} for axisymmetric spiralling bubbles; in
\S5.1 he then restricts the variational problem to oblate spheroids and takes
the small-eccentricity limit. Our goal is to compare this near-spherical
spheroidal approximation with our exact solution branch.

To avoid a second clash with our branch parameter, denote Benjamin's spheroidal
eccentricity by $\eps$. In his notation $\eps=\sin\eta$ and
$a/b=\cos\eta$, where $a$ and $b$ are the polar and equatorial semiaxes of
the oblate spheroid
\cite[p.~366]{MR905816}. Let $\vartheta$ be the
inclination of its symmetry axis to the screw axis. His Weber number is
$W_{\rm B}=2\rho r_{\rm e}c^2/\sigma=2\We$.

Restoring the $O(\eps^4)$ terms in his equation~(5.20), the leading small-eccentricity formula, gives
\begin{equation}
	\label{eq:endpoint-benjamin-weber-match}
	\begin{aligned}
		2\We
		 & =
		\frac{32\eps^2}{9\left(1+\frac32\sin^2\vartheta\right)}
		+O(\eps^4),                                                 \\
		\eps^2
		 & =
		\frac9{16}\We\left(1+\frac32\sin^2\vartheta\right)+O(\We^2) \\
		 & =
		\frac9{16}\We\left(1+O(\vartheta^2+\We)\right).
	\end{aligned}
\end{equation}

For the shape comparison, take the
oblate spheroid of
equivalent radius $R$ whose unit symmetry axis is
$n_{\vartheta}=\cos\vartheta\,e_3
	-\sin\vartheta\,e_1$. Expanding its radial
graph and retaining its $P_2(Z)$- and $XZ$-projections gives
\begin{equation}
	\label{eq:endpoint-benjamin-shape-expansion}
	h
	=
	-\frac{\eps^2}{3}\frac{3Z^2-1}{2}
	+
	\eps^2\vartheta XZ
	+O(\eps^2\vartheta^2+\eps^4).
\end{equation}
Comparison with \eqref{eq:h-expansion} therefore gives
the leading parameter identification
\begin{equation}
	\label{eq:endpoint-benjamin-parameter-match}
	s=
	\eps^2\vartheta+O(\eps^2\vartheta^2+\eps^4),
	\qquad
	\eps^2=\frac9{16}\We+O(\We\vartheta^2+\We^2),
\end{equation}
hence
\begin{equation}
	\label{eq:endpoint-benjamin-parameter-match-2}
	\frac{s}{\We}
	=
	\frac9{16}\vartheta+O(\vartheta^2+\We).
\end{equation}
Thus, the regime $\abs{s}\le\eta\We$, with
$\eta\ll1$, corresponds to the small-inclination part of Benjamin's
small-eccentricity spheroidal regime.

By \eqref{eq:centroid-expansion}, the radius of the
centroid helix is
\begin{equation}
	\label{eq:endpoint-benjamin-radius-match}
	\frac{\abs{\perpProj\bar y_{s,\We}
		}}{R}
	=
	\frac{3\sqrt3}{20}\sqrt{\We}\,\abs{s}
	+O(s^2+\abs{s}\We)
	=
	\frac{\sqrt3}{5}\eps^3\abs{\vartheta}
	+O(\eps^3\vartheta^2+\eps^4\abs{\vartheta}+\eps^5).
\end{equation}
Benjamin's formulas (5.22) and (5.23), with the orders suppressed in his
leading calculation restored from the coefficient expansions on the same page,
are
\begin{align*}
	\frac{R_{\rm B}}{r_{\rm e}}
	 & =
	\frac{\sqrt3}{10}\eps^3\sin(2\vartheta)+O(\eps^5), \\
	\frac{c}{\omega r_{\rm e}}
	 & =
	\frac{2}{3\sqrt3}\eps+O(\eps^3).
\end{align*}
Consequently, \eqref{eq:endpoint-benjamin-radius-match} agrees with (5.22),
because
\[
	\frac{\sqrt3}{10}\eps^3\abs{\sin(2\vartheta)}
	=
	\frac{\sqrt3}{5}\eps^3\abs{\vartheta}
	+O(\eps^3\abs{\vartheta}^3).
\]
Likewise, inside the above regime
the definition of $\alpha_*$, together with the endpoint expansion for
$\alpha$, gives
\[
	\frac{V}{\alpha R}
	=
	\frac{\sqrt{\We}}{2\sqrt3}\left(1+O(\We)\right)
	=
	\frac{2}{3\sqrt3}\eps
	+O(\eps\vartheta^2+\eps^3),
\]
which agrees with (5.23).

There is an apparent factor-of-two typographical error in Benjamin's printed
equation~(5.21): it gives $9\eps^2\sin(2\vartheta)/10$. However, his exact
identity (5.12) and the expansion
$h_{\rm B}-1=9\eps^2/10+O(\eps^4)$ displayed immediately before (5.20),
where $h_{\rm B}=A/B$ is his ratio of inertial coefficients, give instead
\[
	\tan\zeta
	=
	\frac{h_{\rm B}-1}{\cot\vartheta+h_{\rm B}\tan\vartheta}
	=
	\frac9{20}\eps^2\sin(2\vartheta)+O(\eps^4).
\]
The same coefficient is also obtained by combining (5.22) and (5.23),
since (5.12) gives $\tan\zeta=\omega R_{\rm B}/c$. Thus this typo does not affect the
radius and pitch comparisons above.

Finally, our endpoint expansion gives
\[
	\alpha^2
	=
	\frac{12\sigma}{\dout R^3}\left(1+O(\We)\right).
\]
Benjamin's (5.24), again interpreted at the order of his small-eccentricity
calculation, gives
\[
	\omega^2
	=
	\frac{12\sigma}
	{\rho r_{\rm e}^3\left(1+\frac32\sin^2\vartheta\right)}
	\left(1+O(\eps^2)\right)
	=
	\frac{12\sigma}{\dout R^3}
	\left(1+O(\vartheta^2+\eps^2)\right).
\]
Thus the frequency formulas, as well as the radius and pitch formulas above,
agree in the common regime
\[
	\We\ll1,\qquad \frac{\abs{s}}{\We}\ll1,
\]
or equivalently, by \eqref{eq:endpoint-benjamin-parameter-match}--\eqref{eq:endpoint-benjamin-parameter-match-2},
\[
	\eps\ll1,\qquad \abs{\vartheta}\ll1 .
\]
Thus our exact branch rigorously recovers the leading
small-eccentricity, small-inclination asymptotics in Benjamin's
formulas (5.20), (5.22)--(5.24)
for the small-eccentricity limit of his approximation obtained by restricting the variational problem to oblate spheroids.

\section{The resonant case \texorpdfstring{$\tau >0 $}{tau > 0}}

\label{sec:positive-resonant-density-ratios}

As for the endpoint case,
the five exceptional positive values in Lemma~\ref{lem:nonresonance} mean that,
at the round sphere and at $\alpha=\alpha_*$, the kernel is larger than the
desired weight-one plane.

Indeed, at one of the five resonant values
\[
	\tau\in
	\cO = \set{\frac{2}{77}, \frac{2}{7},
		\frac{23}{60},\frac{10}{9},\frac{46}{27}},
\]
the kernel of the linearisation $L_*$ at the spherical solution in the constrained slice is
\[
	K_1\oplus K_m,
\]
where $K_1=\Span\set{XZ,YZ}$, and $K_m$ is the real two-dimensional plane
corresponding to the additional resonant spherical harmonic $Y_\ell^m$:
\[
	\begin{array}{c|c}
		\tau  & (\ell,m) \\
		\hline
		10/9  & (5,4)    \\
		2/7   & (6,5)    \\
		46/27 & (8,8)    \\
		23/60 & (9,9)    \\
		2/77  & (10,10).
	\end{array}
\]
Thus one repeats the endpoint reduction with the fixed splitting
\[
	\H^q_{\rm lin}=K_1\oplus K_m\oplus(K_1\oplus K_m)^\perp.
\]

Restricting to $\abs V<V_{\rm ax}$,
Lemma~\ref{lem:axisymmetric-translating-branch} supplies the same
axisymmetric translating branch simultaneously for every $\tau\ge0$ and
every $\alpha\in\R$. We use its notation and set
\begin{equation}
	\label{eq:positive-resonant-axisymmetric-identification}
	\xi_0(V):=\xi^{\rm ax}(V),
	\qquad
	\Theta(\xi_0(V))=h^{\rm ax}_{0,V}.
\end{equation}
In particular, $\xi_0(V)$ is orthogonal to $K_1\oplus K_m$, the graph is
independent of $\alpha$, and its expansion is
\eqref{eq:axisymmetric-branch-expansion}. The vanishing of the
positive-density interior Bernoulli term on this branch is contained in
\eqref{eq:axisymmetric-bernoulli-reduction}. The resonant
Lyapunov--Schmidt reduction is again based at this already constructed
branch.

Relative to the endpoint computation, with the present value of $\alpha_*$
inserted, the only new contribution to the order-$V^2$ diagonal coefficient
is the interior rotational term
\[
	d_{\ell m}^{\rm in}(\tau)
	=
	\frac{3\tau\alpha_*^2m^2}{16}
	I_{\ell m}\frac{\ell+3}{\ell^2},
\]
With $d_{\ell m}^T$ and $d_{\ell m}^H$ as in
\eqref{eq:endpoint-dT} and \eqref{eq:endpoint-dH}, and using the first
line of \eqref{eq:endpoint-dR} before the endpoint simplification, the
total positive-density diagonal coefficient is
\begin{equation}
	\label{eq:positive-resonant-total-diagonal-coefficient}
	d_{\ell m}(\tau)
	:=
	d_{\ell m}^T+d_{\ell m}^H
	+\frac{3\alpha_*^2m^2}{16}
	I_{\ell m}\frac{\ell-2}{(\ell+1)^2}
	+d_{\ell m}^{\rm in}(\tau),
\end{equation}
where $I_{\ell m}$ is defined in \eqref{eq:endpoint-I-lm}. The order-$V$
coefficients $b_{n\ell}^{(m)}$ are unchanged. Hence, uniformly for
$\abs{\alpha-\alpha_*}\lesssim V^2$,
\[
	\begin{aligned}
		\mu_1(\alpha,V)
		 & =\lambda_{21}(\alpha)+\gamma_{21}(\tau)V^2+O(V^3),         \\
		\mu_m(\alpha,V)
		 & =\lambda_{\ell m}(\alpha)+\gamma_{\ell m}(\tau)V^2+O(V^3),
	\end{aligned}
\]
with
\[
	\gamma_{\ell m}(\tau)
	=
	d_{\ell m}(\tau)
	+
	\sum_{\substack{n\in\set{\ell-1,\ell+1}\\ n\ge\abs m}}
	\frac{b_{\ell n}^{(m)}b_{n\ell}^{(m)}}
	{-\lambda_{nm}(\alpha_*)}.
\]
Here the condition $n\ge\abs m$ retains only $n=\ell+1$ when $\ell=m$.
By \eqref{eq:transversality},
$\partial_\alpha\lambda_{21}(\alpha_*)=-8/\alpha_*\neq0$, so the weight-one
equation selects
\[
	\mu_1(\alpha_1(V),V)=0,
	\qquad
	\alpha_1(V)
	=
	\alpha_*+\frac{\alpha_*}{8}\gamma_{21}(\tau)V^2+O(V^3).
\]
At this value,
\[
	\mu_m(\alpha_1(V),V)
	=
	\Gamma_{\ell m}(\tau)V^2+O(V^3),
	\qquad
	\Gamma_{\ell m}(\tau)
	=
	\gamma_{\ell m}(\tau)
	-
	\frac{\ell(\ell+1)-2}{4}\gamma_{21}(\tau).
\]
Here we used
\[
	\frac{\partial_\alpha\lambda_{\ell m}(\alpha_*)}
	{\partial_\alpha\lambda_{21}(\alpha_*)}
	=
	\frac{\ell(\ell+1)-2}{4},
\]
which follows from \eqref{eq:resonance_condition_tau} and is independent of
$\tau$.

A direct evaluation of the same Schur-complement coefficient as in the endpoint
case gives
\[
	\begin{array}{c|c|c}
		\tau          & (\ell,m) & \Gamma_{\ell m}(\tau) \\
		\hline        &          &                       \\[-0.8em]
		\frac{10}{9}  & (5,4)
		              &
		-\dfrac{745655631939}{404745872960}
		\\[1.2em]
		\frac{2}{7}   & (6,5)
		              &
		-\dfrac{7741466661}{3867399536}
		\\[1.2em]
		\frac{46}{27} & (8,8)
		              &
		-\dfrac{82988030595}{6597344768}
		\\[1.2em]
		\frac{23}{60} & (9,9)
		              &
		-\dfrac{197821204433}{15834895860}
		\\[1.2em]
		\frac{2}{77}  & (10,10)
		              &
		-\dfrac{686527731}{50098048}.
	\end{array}
\]
It remains only to adapt the anisotropic reduction. Writing
\(\Phi_j^\tau\) for the two components of the reduced equation, with
\(a\in K_1\) and \(u\in K_m\),
Lemma~\ref{lem:sharp-weight-one-secondary-variable} supplies all the needed estimates
\eqref{eq:selection-secondary-forcing-simple}--\eqref{eq:selection-primary-secondary-simple}.
Consequently, the Taylor expansions in the proof
of Proposition~\ref{prop:endpoint-anisotropic-reduction-simple} carry
over with \(K_6\) and the weight \(6\) replaced by \(K_m\) and \(m\),
respectively.

Set
\[
	a=VtXZ,
	\qquad
	u=V^2U,
	\qquad
	\alpha=\alpha_1(V)+V^2\widetilde\beta.
\]
For \(j\in\set{1,m}\), put
\[
	M_j^\tau(\widetilde\beta,V)
	:=
	\frac{
		\mu_j(\alpha_1(V)+V^2\widetilde\beta,V)
	}{V^2}.
\]
As in the endpoint proof, these functions extend smoothly through
\(V=0\), with
\[
	M_m^\tau(0,0)=\Gamma_{\ell m}(\tau),
	\qquad
	\partial_{\widetilde\beta}M_1^\tau(0,0)
	=-\frac8{\alpha_*}.
\]

After dividing the secondary equation by \(V^4\) and the scalar
weight-one equation by \(V^3t\), the pure secondary forcing carries
the factor
\[
	V^{m-4}t^m,
\]
while the term linear in \(U\) in the scalar weight-one equation
carries the factor
\[
	V^{m-2}t^{m-2}.
\]
All remaining powers are the same as in the endpoint proof. Since
\(m\ge4\) in all five resonant cases, the rescaled system extends
smoothly through \(V=0\).

At the origin, its derivative with respect to
\((U,\widetilde\beta)\) is
\[
	\begin{pmatrix}
		\Gamma_{\ell m}(\tau)\id_{K_m} & 0           \\
		0                              & -8/\alpha_*
	\end{pmatrix},
\]
which is invertible by the table. The implicit-function theorem and
the same absorption argument as in
Proposition~\ref{prop:endpoint-anisotropic-reduction-simple} give
smooth functions \(U_*(t,V)\) and \(\beta(t,V)\) satisfying
\[
	\norm{U_*(t,V)}_{\L^2}
	\lesssim V^{m-4}\abs{t}^m,
	\qquad
	\abs{\beta(t,V)}
	\lesssim t^2.
\]
Consequently, with
\[
	u=V^2U_*(t,V),
	\qquad
	\alpha=\alpha_1(V)+V^2\beta(t,V),
\]
we have
\[
	\norm u_{\L^2}
	\lesssim\frac{\norm a_{\L^2}^m}{V^2},
	\qquad
	\alpha=\alpha_1(V)+O(\norm a_{\L^2}^2).
\]

For \(t\ne0\), the secondary equation and the scalar weight-one
equation give the full kernel equation by
Remark~\ref{rem:LS_higher_dimensional_kernel}, applied with
\(E=K_1\) and \(F=K_m\). At \(t=0\), symmetry and invertibility of the
\(K_m\)-block give \(u=0\) for every
\(\abs{\alpha-\alpha_1(V)}\lesssim V^2\), while leaving the angular
velocity undetermined. The reconstruction and local-classification
argument are then the same as in the endpoint case.
Further details are left to the reader.

\begin{theorem}
	\label{thm:din-pos-resonant-dimensionless}
	Assume
	\[
		\tau\in\cO
		=
		\set{\frac{2}{77}, \frac{2}{7},
			\frac{23}{60},\frac{10}{9},\frac{46}{27}},
		\qquad
		\alpha_*^2=\frac{24}{3\tau+2}.
	\]
	Set
	\begin{equation}
		\label{eq:positive-resonant-gamma21}
		\gamma_{21}(\tau)
		:=
		-\frac{3\left(1188\tau^3+6605\tau^2+5601\tau+1156\right)}
		{280(2\tau+1)(3\tau+2)(11\tau+7)} .
	\end{equation}
	{For every $k\ge4$, there exist constants $V_0,\eta_0>0$ and
		$0<\eta<\eta_0$, depending only on $k$ and $\tau$. The proof constructs
		smooth solution maps $(\alpha(s,V),h(s,V))$ on the larger tube
		$0<V<V_0$, $\abs{s}\le\eta_0V$. For every $0<V<V_0$ and every
		$\abs{s}\le\eta V$, these maps give smooth solutions of the
		nondimensional free-boundary problem
		\eqref{eq:main-problem-nondimensional}, with}
	prescribed volume $4\pi/3$, angular speed $\alpha=\alpha(s,V)>0$, and
	boundaries
	\[
		\Gamma(s,V)=\set{(1+h(s,V)(\omega))\omega:\omega\in\Sph}.
	\]
	After rotation about the $e_3$-axis, the branch is normalised by the
	$(\ell,m)=(2,1)$ direction $sXZ$. For $\omega=(X,Y,Z) \in \Sph$,
	\begin{equation}
		\label{eq:positive-resonant-h-expansion}
		\begin{aligned}
			h(s,V)
			={} &
			-\frac{3V^2}{16}\frac{3Z^2-1}{2}
			+sXZ                     \\
			    & \quad
			-\sqrt{6(3\tau+2)}\,Vs\,Y
			\left[
				\frac{3}{40(2\tau+1)}
				+
				\frac{5}{22\tau+14}
				\left(Z^2-\frac15\right)
				\right] \\
			    & \quad
			+O_{\H^{k+2}}\bigl(s^2+\abs{s}V^2+V^4\bigr).
		\end{aligned}
	\end{equation}
	The angular speed satisfies
	\begin{equation}
		\label{eq:positive-resonant-alpha-expansion}
		\alpha(s,V)
		=
		\alpha_*
		+
		\frac{\alpha_*}{8}\gamma_{21}(\tau)V^2
		+
		O(s^2+V^3).
	\end{equation}
	At $s=0$, this value of $\alpha$ is only the continuous selection inherited
	from the non-axisymmetric branch. The same axisymmetric surface solves the
	physical problem for every nearby $\alpha$.
		{If $0<\abs{s}\le\eta V$,}
	then
	\begin{equation}
		\label{eq:positive-resonant-centroid-expansion}
		\perpProj\left(
		\frac1{\abs{\Omega(s,V)}}\int_{\Omega(s,V)}y\,\dd y
		\right)
		=
		-\frac{3\sqrt{6(3\tau+2)}}{40(2\tau+1)}\,Vs\,e_2
		+
		O(s^2+\abs{s}V^2),
	\end{equation}
	and this horizontal centroid is non-zero after decreasing $V_0$ and $\eta$.
		{For each fixed $0<V<V_0$, the non-axisymmetric branch is locally unique
			modulo rotations about the $e_3$-axis in the fixed-volume, vertically centred
			radial-graph class. On the axisymmetric set, the shape is locally unique,
			while $\alpha$ remains free.
		}
\end{theorem}
\begin{proof}[Proof of Theorem~\ref{thm:din-pos-dimensional} in the case $\tau \in \cO$]
	Apply Theorem~\ref{thm:din-pos-resonant-dimensionless} with $k=4$ and use
	the inverse nondimensionalisation from
	Section~\ref{sec:nondimensionalisation}, with
	\[
		\tau=\frac{\din}{\dout},
		\qquad
		\widehat V=\sqrt{\We},
		\qquad
		\widehat\alpha_*
		=
		\left(\frac{24}{3\tau+2}\right)^{1/2}.
	\]
	We denote the dimensionless solution objects by hats and set
	$\We_0:=V_0^2$. Under this rescaling,
	\[
		h(s,\We)=\widehat h(s,\widehat V),
		\qquad
		\frac{\alpha(s,\We)}{\alpha_*}
		=
		\frac{\widehat\alpha(s,\widehat V)}{\widehat\alpha_*},
		\qquad
		\frac{\bar y_{s,\We}}{R}
		=
		\widehat{\bar y}_{s,\widehat V}.
	\]
	Hence \eqref{eq:positive-resonant-h-expansion} gives
	\eqref{eq:h-expansion}, while
	\eqref{eq:positive-resonant-alpha-expansion} gives
	\[
		\frac{\alpha(s,\We)}{\alpha_*}
		=
		1+\frac{\gamma_{21}(\tau)}8\We
		+O\left(s^2+\We^{3/2}\right).
	\]
	Likewise, \eqref{eq:positive-resonant-centroid-expansion} becomes
	\eqref{eq:centroid-expansion}, and
	$\abs{s}\le\eta\sqrt{\We}$ is precisely
	$\abs{s}\le\eta\widehat V$. Smoothness and the local uniqueness are
	preserved by the rescaling.
\end{proof}

\begin{remark}
	\label{rem:nonresonant-sharp-alpha-expansion}
	The sharper angular-velocity expansion in \eqref{eq:alpha-expansion}
	holds true also in the nonresonant case $\tau>0$, $\tau\notin\cO$.
	Indeed, the calculation of the weight-one block in
	Section~\ref{sec:positive-resonant-density-ratios} involves only
	azimuthal weight one and is therefore independent of the additional
	resonant plane. Applying this calculation to
	\eqref{eq:kernel-equation-main-proof}, and using
	\eqref{eq:transversality} together with the smooth dependence on
	$\norm{a}_{\L^2}^2$, we obtain
	\[
		\alpha(s,V)
		=
		\alpha_*
		+\frac{\alpha_*}{8}\gamma_{21}(\tau)V^2
		+O\bigl(s^2+V^3\bigr).
	\]
	The inverse nondimensionalisation gives
	\eqref{eq:alpha-expansion}.
\end{remark}

\begin{proof}[Proof of Theorem~\ref{thm:main-informal}]
	For the prescribed volume $m>0$, set
	\[
		R=\left(\frac{3m}{4\pi}\right)^{1/3},
		\qquad
		\tau=\frac{\din}{\dout},
	\]
	and, for a sufficiently small prescribed $\We>0$, choose
	\[
		V=\left(\frac{\sigma\We}{\dout R}\right)^{1/2}>0.
	\]
	Then apply Theorem~\ref{thm:din-pos-dimensional}. Choose a nonzero
	amplitude $0<\abs s\le\eta\sqrt{\We}$, decreasing the upper bound on $\We$ if
	needed in the nonresonant case. The applicable centroid expansion has a
	nonzero horizontal component. Remark~\ref{rem:really_a_helix} and the
	converse construction following \eqref{eq:main-problem} therefore give a
	relative-equilibrium Euler solution of volume $m$ whose centroid follows a
	genuine helix.
		{The local classification in Theorem~\ref{thm:din-pos-dimensional} gives
			uniqueness.}
\end{proof}

\section{Rigidity in the absence of surface tension}\label{sec:no-bubbles}

In the present section we focus on the necessity of surface tension for the existence of bubbles moving on a helical path. Throughout we consider
the bubble regime $0\le\din\le\dout$,
no surface tension $\sigma=0$, and we impose the same topological assumption as in the rest of the manuscript. The inner and outer phases are simply connected and their common boundary is connected.

Before turning to the proof of Theorem \ref{thm:no-bubbles}, we note that \eqref{eq:main-problem} with $\sigma=0$ corresponds to continuity of the pressure across the interface. This imposes no additional restriction in Corollary \ref{cor:sheets}, which concerns regular material vortex sheets for the homogeneous Euler equations.
For a piecewise regular material sheet, the one-sided velocity and pressure fields extend continuously to the interface. We write
\[
	u=\uout+\jump u\,\mathbbm 1_{\Omegain(t)},
	\qquad
	p=\pout+\jump p\,\mathbbm 1_{\Omegain(t)} .
\]
Since $\partial_t\mathbbm 1_{\Omegain(t)}=V_{\cS}\delta_{\cS}$ and $\nabla\mathbbm 1_{\Omegain(t)}=-n\delta_{\cS}$ in the sense of distributions, the singular part of $\partial_t u+\divv(u\otimes u)+\nabla p$ on $\cS$ has coefficient
\[
	V_{\cS}\jump u-\jump{u\otimes u}\,n-\jump p\,n .
\]
The kinematic condition gives $\uin\cdot n=\uout\cdot n=V_{\cS}$, hence $\jump{u\otimes u}\,n=V_{\cS}\jump u$. Thus the singular part vanishes if and only if $\jump p=0$. Therefore, a piecewise regular material sheet can solve the homogeneous Euler equations distributionally only if the pressure is continuous across the interface.

\begin{proof}[Proof of Theorem \ref{thm:no-bubbles}.]
	Without surface tension the Bernoulli equation can be written as
	\begin{equation*}
		\jump{\rho\left(\frac{1}{2}\abs{\nabla\phi}^2-\nabla\phi\cdot W_{\alpha,V}\right)}= c
	\end{equation*}
	for some $c \in \R$.
	After dividing by $\dout$, we can again assume
	\begin{align*}
		\din=\tau,\quad \dout=1
	\end{align*}
	with $\tau =\din/\dout\in [0,1]$. At $\tau=0$, all interior-potential
	terms below are omitted, in accordance with the convention following
	\eqref{eq:main-problem}. Changing the constant and abbreviating $W:=W_{\alpha,V}$,
	we thus consider
	\begin{align}\label{eq:bernoulli_equation_A}
		\tau \left(\frac{1}{2}\abs{\nabla\phiin}^2-\nabla\phiin\cdot W\right)-\left(\frac{1}{2}\abs{\nabla\phiout}^2-\nabla\phiout\cdot W\right)=C
	\end{align}
	on $\Gamma=\partial\Omega$ and for some constant $C \in \R$. We define the energy
	\begin{align*}
		E & :=\tau \int_{\Omega}\abs{\nabla\phiin}^2\dd y+\int_{\R^3\setminus \overline{\Omega}}\abs{\nabla \phiout}^2\dd y.
	\end{align*}
	Our goal is to show that $E=0$.
	In this case it follows that $\nabla\phiout\equiv 0$ and, when
	$\tau>0$, $\nabla\phiin\equiv 0$. The Neumann boundary condition then implies
	\begin{align*}
		W_{\alpha,V}\cdot n=0\quad\text{on }\Gamma.
	\end{align*}
	Since $\Gamma$ is compact, the coordinate function $y\mapsto y_3$
	attains a maximum on $\Gamma$, at which $n=e_3$. Evaluation at this point gives
	\begin{align*}
		0=W_{\alpha,V}\cdot n=W_{\alpha,V}\cdot e_3=V.
	\end{align*}
	Moreover, it follows that
	\begin{align*}
		\alpha(e_3\times y)\cdot n=0,\quad y\in\Gamma.
	\end{align*}
	Consequently, if $\alpha\neq 0$, the vector field $y\mapsto e_3\times y$ is tangent to $\Gamma$. Its flow is $y\mapsto\cR_\theta y$, so $\cR_\theta\Gamma=\Gamma$ for every $\theta\in\R$. Hence $\Gamma$, and therefore $\Omega$, is axially symmetric.

	It thus remains to show that $E=0$. By Green's identity in $\Omega$
	when $\tau>0$, and in the exterior domain after exhaustion by large balls, using
	$\eout\to0$ and
	$\nabla\eout\in \L^2$, we obtain
	\begin{align*}
		E
		 & =
		\tau \int_{\Gamma}\ein\partial_n\ein\dd S
		-
		\int_{\Gamma}\eout\partial_n\eout\dd S \\
		 & =
		\int_{\Gamma}
		\left(\tau \ein-\eout\right)
		(W\cdot n)\dd S.
	\end{align*}
	The boundary terms on the large balls vanish in the limit by the decay
	$\eout=O(|y|^{-1})$, $\nabla\eout=O(|y|^{-2})$ following from
	standard exterior harmonic estimates.
	Thus
	\begin{equation}\label{eq:vortex-sheet-energy-boundary_A}
		E=\int_{\Gamma}d\,(W\cdot n)\dd S,
	\end{equation}
	where $d:=\tau \ein-\eout$ for $\tau \ge 0$.

	We now derive a Pohozaev-type identity using the stress tensor
	\[
		T(\nabla u)=\nabla u\otimes\nabla u
		-\frac12|\nabla u|^2 I
	\]
	and the horizontal dilation field
	\[
		X(y):=(y_1,y_2,0).
	\]
	We have
	\[
		\Div(T(\nabla u)X)
		=
		u_{y_1}^2+u_{y_2}^2-|\nabla u|^2
		=
		-u_{y_3}^2
	\]
	for any harmonic function $u$.

	When $\tau>0$, applying the divergence theorem to $T(\nabla\ein)X$ in $\Omega$ gives
	\begin{equation}\label{eq:vortex-sheet-interior-pohozaev_A}
		\int_{\Gamma}
		\left[
			(W\cdot n)X\cdot\nabla\ein
			-\frac12|\nabla\ein|^2 X\cdot n
			\right]\dd S
		=
		-\int_{\Omega}(\partial_{y_3}\ein)^2\dd y .
	\end{equation}
	Applying the same identity in the exterior domain, again by exhaustion and using the decay at infinity, gives
	\begin{equation}\label{eq:vortex-sheet-exterior-pohozaev_A}
		-\int_{\Gamma}
		\left[
			(W\cdot n)X\cdot\nabla\eout
			-\frac12|\nabla\eout|^2 X\cdot n
			\right]\dd S
		=
		-\int_{\R^3\setminus\overline{\Omega}}
		(\partial_{y_3}\eout)^2\dd y .
	\end{equation}

	For $\tau>0$, multiplying \eqref{eq:vortex-sheet-interior-pohozaev_A}
	by $\tau$ and adding \eqref{eq:vortex-sheet-exterior-pohozaev_A}, and for
	$\tau=0$ using only \eqref{eq:vortex-sheet-exterior-pohozaev_A}, we get
	\begin{align*}
		 & \int_{\Gamma}
		\left[
			(W\cdot n)X\cdot\delta
			-\frac12
			\left(\tau |\nabla\ein|^2
			-|\nabla\eout|^2\right)X\cdot n
			\right]\dd S \\
		 & \qquad =
		-\tau \int_{\Omega}(\partial_{y_3}\ein)^2\dd y
		-\int_{\R^3\setminus\overline{\Omega}}
		(\partial_{y_3}\eout)^2\dd y\ge -E,
	\end{align*}
	Here $\delta:=\tau\nabla\phiin-\nabla\phiout$ for $\tau>0$, while
	$\delta:=-\nabla\phiout$ for $\tau=0$.

	On the left-hand side we use \eqref{eq:bernoulli_equation_A} to find
	\begin{align*}
		-E & \le \int_\Gamma \left[
			                     (W\cdot n)X\cdot\delta
			                     -\left(C+\delta\cdot W\right)X\cdot n
			                     \right]\dd S                               \\
		   & =\int_{\Gamma}\delta\cdot \left((W\cdot n)X-(X\cdot n)W\right)\dd S-2C\abs{\Omega}.
	\end{align*}
	Now the vector field
	\begin{align*}
		Z:=(W\cdot n)X-(X\cdot n)W=n\times(X\times W)
	\end{align*}
	is tangential to $\Gamma$ with
	\begin{align*}
		\divv_\Gamma Z=-n\cdot\left(\nabla \times (X\times W)\right)=2W\cdot n.
	\end{align*}
	Here we used the identity
	\[
		\Div_{\Gamma}(n\times A)
		=
		-n\cdot(\nabla\times A),
	\]
	and
	\[
		X\times W
		=
		\left(Vy_2,-Vy_1,\alpha(y_1^2+y_2^2)\right).
	\]
	Since $Z$ is tangential to $\Gamma$, we have $\delta\cdot Z=\nabla_\Gamma d\cdot Z$. Therefore, integration by parts and \eqref{eq:vortex-sheet-energy-boundary_A} give
	\begin{align*}
		-E\le -\int_\Gamma d\divv_\Gamma Z\dd S-2C\abs{\Omega}=-2\int_\Gamma d\, (W\cdot n)\dd S-2C\abs{\Omega}=-2E-2C\abs{\Omega}.
	\end{align*}
	This implies
	\begin{align}\label{eq:energy_bounded_by_C_A}
		0\le E\le -2C\abs{\Omega}.
	\end{align}

	On the other hand, $d$ has a critical point on $\Gamma$. If $\tau>0$, then at such a point
	\[
		0=\nabla_\Gamma d
		=\tau \nabla_\Gamma\phiin-\nabla_\Gamma\phiout,
	\]
	and hence
	\begin{align*}
		\nabla\phiin   & =(W\cdot n)n+\nabla_{\Gamma}\phiin,      \\
		\nabla \phiout & =(W\cdot n)n+\tau \nabla_{\Gamma}\phiin.
	\end{align*}
	Evaluating \eqref{eq:bernoulli_equation_A} there gives
	\[
		C=\frac{1-\tau}{2}\left((W\cdot n)^2
		+\tau \abs{\nabla_\Gamma \phiin}^2\right)\ge0.
	\]
	If $\tau=0$, then $d=-\phiout$, so criticality gives
	$\nabla_\Gamma\phiout=0$. Hence
	$\nabla\phiout=(W\cdot n)n$ at that point, and the endpoint Bernoulli
	equation gives
	\[
		C=\frac12(W\cdot n)^2\ge0.
	\]
	Thus $C\ge0$ throughout the bubble regime $0\le\tau\le1$, and
	\eqref{eq:energy_bounded_by_C_A} implies $E=0$ as desired.
\end{proof}

\begin{proof}[Proof of Corollary~\ref{cor:sheets}.]
	Since $\Gamma$ is $\C^2$, its embedding in
	$S^3=\R^3\cup\{\infty\}$ is locally flat. The generalised Schoenflies
	theorem \cite{MR117695} therefore implies that the closures of both
	components of $S^3\setminus\Gamma$ are topological three-balls. In particular,
	both $\Omega$ and $\R^3\setminus\overline\Omega$ are simply connected. Hence
	$\uinout=\nabla\phinout$ for one-sided $\C^2$ potentials, and
	incompressibility together with the material condition gives
	\[
		\Delta\phinout=0,
		\qquad
		\partial_n\phiin=\partial_n\phiout=0.
	\]
	The assumption $u\in\L^2$ gives the exterior finite-energy condition, and
	after adding a constant we may normalise
	$\phiout(y)\to0$ as $\abs y\to\infty$.

	The jump calculation above gives $\jump p=0$ on $\Gamma$. Phasewise
	irrotationality and the stationary Euler equation also imply
	\[
		\nabla\!\left(p+\frac12\abs u^2\right)=0.
	\]
	Consequently,
	\[
		\frac12\abs{\nabla\phiin}^2
		-\frac12\abs{\nabla\phiout}^2=c
		\qquad\text{on }\Gamma
	\]
	for some $c\in\R$. Therefore Theorem~\ref{thm:no-bubbles} applies with
	$\din=\dout=1$, $\alpha=V=\sigma=0$, and
	$R=(3\abs{\Omega}/(4\pi))^{1/3}$. It yields
	$\nabla\phiin=\nabla\phiout=0$, hence $u\equiv0$.
\end{proof}

\appendix

\section{Derivation of the free boundary value problem}\label{sec:derivation}

We here address the derivation of \eqref{eq:main-problem}. Let us first
assume that we have a sufficiently regular solution to
\eqref{eq:original_euler}, \eqref{eq:finite_energy_condition}, which is
irrotational in each phase, i.e.\ which in addition satisfies
\eqref{eq:irrotational_condition}. We recall the topological convention
fixed above: for each $t$, both phases are simply connected; in particular
$\Omegain(t)$ is a bounded smooth simply connected domain with connected
boundary and $\Omegaout(t)$ is the simply connected exterior domain.

Since the phases are simply connected and
$\curl \uinout=0$, for each $t\in(0,T)$ there exist single-valued
velocity potentials
\[
	\phinout(t,\cdot)\colon \Omegainout(t)\to\R,
	\qquad
	\uinout(t,\cdot)=\nabla\phinout(t,\cdot).
\]
The potentials are unique up to additive functions of time. This is the standard potential-flow
criterion in simply connected fluid regions; see \cite[Sec.~1.8]{MR1217252}.

In the irrotational case, using
\begin{align*}
	u\cdot\nabla u=\nabla\left(\frac{1}{2}\abs{u}^2\right),
\end{align*}
the Euler equations imply, in each connected phase,
\[
	\nabla\left(
	\pinout
	+\rho^{\rm in/out}
	\left(\partial_t\phinout
	+\frac12\abs{\nabla\phinout}^2\right)
	\right)=0 .
\]
Combining this with the
Young--Laplace condition yields the Bernoulli system
\begin{align}\label{eq:bernoulli_system_original_coordinates}
	\Delta \phi                                                                   & =0\phantom{V_{\cS}c_0(t)}\quad\text{in }\Oin\cup\Oout,              \\
	\partial_n\phiin=\partial_n\phiout                                            & =V_{\cS}\phantom{0c_0(t)} \quad\text{on }\cS,             \nonumber \\
	\jump{\rho\left(\partial_t\phi+\frac{1}{2}\abs{\nabla\phi}^2\right)}+\sigma H & =c_0(t)\phantom{V_{\cS}0}\quad\text{on }\cS,             \nonumber  \\
	\nabla\phiout\in \L^2(\Omegaout                                               & (t);\R^3),\quad t\in(0,T),                 \nonumber
\end{align}
where we again understand
$\phi=\phiin\mathbbm{1}_{\Oin}+\phiout\mathbbm{1}_{\Oout}$,
$\partial_n\phinout=\nabla\phinout\cdot n$, and $c_0(t)$ denotes a time-dependent constant.

Next we phrase \eqref{eq:bernoulli_system_original_coordinates} in rotating and translating coordinates.
Recall that $\cR_\theta$ denotes the rotation around $\R e_3$ by angle $\theta\in\R$, cf. \eqref{eq:rotationmatrix}.
Fixing $\alpha,V\in\R$, we write
\begin{align*}
	x=x(t,y)=\cR_{\alpha t}y+tVe_3
\end{align*}
and introduce transformed sets and potentials according to
\begin{gather*}
	\tilde{\phi}(t,y):=\phi(t,x),\quad \tilde{\rho}(t,y):=\rho(t,x),\\
	\widetilde{\cD}^{\rm in / \rm out}:=\set{(t,y):(t,x)\in\Oinout},\quad
	\tilde{\cS}:=\set{(t,y):(t,x)\in\cS}.
\end{gather*}
Moreover, the slices $\tilde{\Omega}^{\rm in /\rm out}(t)$ are defined in analogy to $\Omegainout(t)$, such that, for example,
\begin{align*}
	\Omegain(t)=\cR_{\alpha t}\tilde{\Omega}^{\rm in}(t)+tVe_3.
\end{align*}
The first and last condition in \eqref{eq:bernoulli_system_original_coordinates}, then translate to
\begin{align}
	\Delta\tilde\phi
	 & =0
	\quad\text{in }
	\widetilde{\cD}^{\rm in}\cup\widetilde{\cD}^{\rm out},
	\label{eq:equations_transformed1}                      \\
	\nabla\tphiout
	 & \in\L^2\bigl(\tilde{\Omega}^{\rm out}(t);\R^3\bigr)
	\quad\text{for every }t\in(0,T).
	\label{eq:equations_transformed2}
\end{align}

Furthermore, the exterior unit normal to $\partial\tilde{\Omega}^{\rm in}(t)$ is given by $\tilde n(t,y)=\cR_{\alpha t}^T n(t,x)$, such that for both phases, $(t,x)\in\cS$, hence $(t,y)\in\tilde{\cS}$, there holds
\begin{align*}
	\partial_n\phi(t,x)=\nabla\phi(t,x)\cdot n(t,x)=\nabla\tilde \phi(t,y)\cdot \tilde{n}(t,y)=\partial_{\tilde{n}}\tilde \phi(t,y).
\end{align*}
On the other hand, recalling \eqref{eq:def-WalphaV}, one observes that
\begin{align}\label{eq:time_derivative_coordinate_change}
	\partial_tx=\alpha\cR_{\alpha t}\left(e_3\times y\right)+Ve_3=\cR_{\alpha t} W_{\alpha,V}(y),
\end{align}
The identity
$\partial_t\cR_{\alpha t}y
	=\alpha\cR_{\alpha t}(e_3\times y)$
and all subsequent coordinate identities hold for every
$\alpha\in\R$, including $\alpha=0$ and $\alpha<0$.
This implies that the normal velocity $V_{\cS}$ in $x$-coordinates
relates to the normal velocity $V_{\tilde{\cS}}$ in $y$-coordinates according to
\begin{align*}
	V_{\cS}(t,x)=V_{\tilde{\cS}}(t,y)+W_{\alpha,V}(y)\cdot\tilde{n}(t,y).
\end{align*}
It follows that the second line in \eqref{eq:bernoulli_system_original_coordinates} translates to
\begin{align}\label{eq:equations_transformed3}
	\partial_{\tilde n}\tphiin=\partial_{\tilde n}\tphiout & =V_{\tilde{\cS}}+W_{\alpha,V}\cdot\tilde{n}\quad \text{on }\tilde{\cS}.
\end{align}

Regarding the third equation in \eqref{eq:bernoulli_system_original_coordinates}, formula \eqref{eq:time_derivative_coordinate_change} and the fact that the mean curvature is invariant under rotations and translations imply
\begin{equation}\label{eq:equations_transformed4}
	\jump{\tilde{\rho}\left(\partial_t\tilde{\phi}
		-\nabla\tilde\phi\cdot W_{\alpha,V}
		+\frac{1}{2}\abs{\nabla \tilde{\phi}}^2\right)}
	+\sigma \tilde{H}
	=\tilde c(t)
	\quad\text{on }\tilde{\cS}.
\end{equation}
We observe that
\begin{equation*}
	-\nabla\tilde\phi\cdot W_{\alpha,V}
	+\frac{1}{2}\abs{\nabla \tilde{\phi}}^2
	+\frac{V^2}{2}
	=
	\frac{1}{2}\abs{\nabla \tilde{\phi}-Ve_3}^2
	+\alpha\, y\cdot(e_3\times\nabla \tilde{\phi}).
\end{equation*}
Since $\tilde\rho$ is constant in each phase, adding the term
$\tilde\rho V^2/2$ inside the jump changes the left-hand side of \eqref{eq:equations_transformed4} only by the
constant $(\din-\dout)V^2/2$, which can be absorbed into the Bernoulli
constant $\tilde{c}(t)$.

For a relative equilibrium, $\nabla\tilde\phi$ is time independent in each phase. Using the freedom to add a function of time to each phase potential, we may therefore choose $\tilde\phi$ itself to be time independent.
Finally, after redefining the Bernoulli constant as above, we restrict equations \eqref{eq:equations_transformed1},
\eqref{eq:equations_transformed2}, \eqref{eq:equations_transformed3}, and
\eqref{eq:equations_transformed4} to the stationary case, i.e.
$\tilde{\phi}(t,y)=\tilde{\phi}(y)$,
$\tilde{\Omega}^{\rm in / \rm out}(t)=\tilde{\Omega}^{\rm in / \rm out}$,
$V_{\tilde{\cS}}=0$, $\tilde{c}(t)=\tilde{c}$,
and simplify notation by replacing all quantities $\tilde{f}$ again by $f$.
We then arrive at system \eqref{eq:main-problem} with $\Omega:=\tilde\Omegain$ and $\Gamma:=\partial\tilde\Omegain$.
Note that \eqref{eq:main-problem} in addition normalises the additive constant for the outer potential, and that it specifies the volume of the interior phase.

Conversely, reversing the steps, any solution to \eqref{eq:main-problem} gives rise to a solution of the incompressible irrotational two-phase Euler system \eqref{eq:original_euler}, \eqref{eq:finite_energy_condition}, \eqref{eq:irrotational_condition} whose domains, density, pressure, and velocity satisfy \eqref{eq:rigid_motion_Omega_in}.

\bibliographystyle{abbrv}
\bibliography{daVinci_paradox}

{}

\vfill
\end{document}